\documentclass[a4paper,11pt]{article}
\usepackage{amsmath,amsfonts,amssymb,amsthm}
\usepackage{pstricks}
\usepackage[labelsep=none]{caption}

\newtheorem{theorem}{Theorem}[section]
\newtheorem{corollary}[theorem]{Corollary}
\newtheorem{lemma}[theorem]{Lemma}
\newtheorem{proposition}[theorem]{Proposition}
\newtheorem{claim}[theorem]{Claim}

\theoremstyle{definition}
\newtheorem{definition}[theorem]{Definition}
\newtheorem{remark}[theorem]{Remark}
\newtheorem{convention}[theorem]{Convention}
\newtheorem{construction}[theorem]{Construction}

\newcommand{\sphere}{\mathbb{S}^2}  														%2-sphere
\newcommand{\Sphere}{\mathbb{S}^3}  														%3-sphere
\newcommand{\disc}{\mathbb{D}}   																%disc
\newcommand{\intvl}{\mathrm{I}}              										%[0,1]
\DeclareMathOperator*{\Int}{int}  															%interior
\DeclareMathOperator{\ext}{ext}   															%S^3\L
\DeclareMathOperator*{\nhd}{\mathcal{N}} 												%neighbourhood in mfld

\newcommand{\rap}[2]{\Delta_{#1}(#2)}														%reduced alexander poly of #1 evaluated at #2
\newcommand{\nrap}[2]{\Delta^{0}_{#1}(#2)}											%reduced alexander poly of #1 evaluated at #2, normalised

\DeclareMathOperator*{\dist}{d}   															%distance 
\DeclareMathOperator{\V}{V}   																	%vertices
\DeclareMathOperator{\E}{E}   																	%edges
\newcommand{\Ends}{\varepsilon}  																%endpoints of an edge
\newcommand{\ocon}[1]{$#1$--connected}   												%#1-connected, cf defn of digraph
\newcommand{\ball}[3]{\mathcal{B}(#1,#2,#3)}  									%#3-ball in #1 about #2

\newcommand{\spantree}[1]{directed $#1$--spanning subtree}			%rooted SUBtree reaching verts of #1
\DeclareMathOperator{\Tr}{Tr}  																	%rooted SUBtree in graph with given origin
\DeclareMathOperator{\chr}{char}     														%characteristic of tree in graph
\newcommand{\mmtree}[1]{$#1$--maximal directed spanning subtree}%rooted SUBtree with maximal no. edges in #1
\DeclareMathOperator{\HTr}{Tr^{\mathcal{H}}}  									%\mmtree{H} in graph with origin
\DeclareMathOperator{\KTr}{Tr^{\mathcal{K}}}  									%\mmtree{K} in graph with origin
\newcommand{\graph}[2]{G^{\mathcal{#1}}(#2)}										%graph of diagram #2
\newcommand{\graphm}[1]{G^{\mathcal{#1}}}		  				 					%graph as a map
\newcommand{\nmfld}[2]{N^{\mathcal{#1}}(#2)}     								%construct product mfld from graph
\newcommand{\mmfld}[2]{M^{\mathcal{#1}}(#2)}     								%complementary to N^#1

\DeclareMathOperator{\ms}{MS}    																%Kakimizu complex
\DeclareMathOperator{\is}{IS}    																%incompressible Kakimizu complex

\begin{document}
\title{Homogeneous links, Seifert surfaces, digraphs and the reduced Alexander polynomial}
\author{Jessica E. Banks}
\date{}
\maketitle

\begin{abstract}
We give a geometric proof of the following result of Juhasz.

\emph{Let $a_g$ be the leading coefficient of the Alexander polynomial of an alternating knot $K$. If $|a_g|<4$ then $K$ has a unique minimal genus Seifert surface.}

In doing so, we are able to generalise the result, replacing `minimal genus' with `incompressible' and `alternating' with `homogeneous'. We also examine the implications of our proof for alternating links in general.
\end{abstract}

\tableofcontents
%------------------------------

\section{Introduction}
%Master document is alexpolypaper.tex.

The Alexander polynomial was the first knot polynomial, being defined by Alexander in 1928 (\cite{MR1501429}). 
Crowell and Murasugi have proved the following result relating the genus of an alternating link $L$ to its reduced Alexander polynomial $\nrap{L}{t}$.

\begin{theorem}[\cite{MR0099665} Theorem 3.5; \cite{MR0099664} II Theorem 4.1]\label{degreerapthm}
For an alternating link $L$ with $m$ link components, let $R$ be a Seifert surface given by applying Seifert's algorithm to an alternating diagram for $L$. Then
$\deg\nrap{L}{t}=2g(R)+m-1=1-\chi(R)$, where $\deg$ denotes degree.
\end{theorem}

In \cite{MR2443240} Juhasz gives the following relationship between the coefficients of $\Delta^0_L$ and the Seifert surfaces for $L$.
He proves this using sutured Floer homology.

\begin{theorem}[\cite{MR2443240} Corollary 2.4]\label{juhaszthm}
Suppose that $K$ is an alternating knot in $\Sphere$ of genus $g$ and let $n>0$. If the leading coefficient $a_g$ of its Alexander polynomial satisfies $|a_g|<2^{n+1}$ then $K$ can have at most $n$ distinct minimal genus Seifert surfaces that are disjoint in their interiors. 
In particular, if $|a_g|<4$ then $K$ has a unique minimal genus Seifert surface.
\end{theorem}

We provide an alternative proof of the case $|a_g|<4$, extended as follows. 
Recall that the class of homogeneous links generalises both alternating links and positive links.

\begin{theorem}\label{incompthm}
Let $L$ be a homogeneous link that is not split, and let $a_g$ be the leading coefficient of the reduced Alexander polynomial of $L$. If $|a_g|<4$ then $L$ has a unique incompressible Seifert surface.
\end{theorem}

This proof is based on those of Crowell and Murasugi, and involves studying certain digraphs defined from link diagrams. Our main result here is that $|a_g|$ defines a finite set of building blocks from which the digraph given by $L$ can be constructed. This has the following as a corollary.

\begin{theorem}\label{finitenessthm}
For fixed $n\in\mathbb{N}$, there is a finite set $\mathcal{S}$ of surfaces embedded in $\Sphere$ with the following property. Any non-split, homogeneous link $L$ with $\nrap{L}{0}\leq n$ has a minimal genus Seifert surface $R$ built from surfaces in $\mathcal{S}$ by reflection, Murasugi sum and plumbing with Hopf bands. 

If $\nrap{L}{0}$ is prime, $R$ can be formed using only one element of $\mathcal{S}$.
\end{theorem}

In Section \ref{prelimsec} we give standard definitions we will need and set out conventions we will adopt, some of which are non-standard. In addition, we recall some known results, and prove a number of others. 
In Section \ref{alexpolysec} we examine the definition of the reduced Alexander polynomial, as considered by Alexander, Murasugi and Crowell. From this we define the digraphs referred to above. These digraphs are the focus of Section \ref{infgraphssec}, in which we prove Theorem \ref{finitenessthm}.
We then complete the proof of Theorem \ref{incompthm} in Section \ref{juhaszsec}.

I wish to thank Marc Lackenby for his help and guidance over the course of this work.

%------------------------------

\section{Preliminaries}\label{prelimsec}
\subsection{Links and Seifert surfaces}
%Master document is alexpolypaper.tex.

We will define homogeneous links in Definition \ref{homogeneousdefn}, but for the majority of this paper we will only need to consider alternating links.

\begin{convention}
We consider oriented links in $\Sphere$. In addition, we study links that are not split and link diagrams that are connected. If a link $L$ is not split, any diagram of $L$ is connected. Conversely for an alternating link $L$, Menasco has shown  (\cite{MR721450} Theorem 1(a)) that if an alternating diagram $D$ of $L$ is connected then $L$ is not split.
\end{convention}

\begin{lemma}\label{connectedsumlemma}
Let $D$ be a diagram of a link $L$. Suppose there is a simple closed curve $\rho$ in $\sphere$  missing the crossings of $D$ and meeting the edges exactly twice transversely. Then $L=L_1\# L_2$ for some links $L_1,L_2$ with diagrams $D_1,D_2$ respectively. If $D$ is alternating then so are $D_1$ and $D_2$.
\end{lemma}

\begin{lemma}[\cite{MR721450} Theorem 1]
Let $D$ be an alternating diagram of a link $L$ with no nugatory crossings. Then $L$ is prime if and only if, whenever $\rho$ is as described above, either the inside or the outside of $\rho$ contains no crossings of $D$.
\end{lemma}

\begin{convention}
In other words, a link $L$ with alternating diagram $D$ is prime if and only if $D$ looks prime. We will therefore use this as the definition of prime for alternating links.
\end{convention}

\begin{convention}
If $M$ is a manifold and $W\subseteq M$, then $\nhd(W)$ will denote a regular open neighbourhood of $W$ in $M$, unless otherwise stated.
\end{convention}

\begin{definition}
A \textit{Seifert surface} for a link $L$ is a compact, connected surface $R$ embedded in $\Sphere$ such that $R$ is oriented and $\partial R=L$ as an oriented link. We consider such surfaces up to ambient isotopy in $\Sphere$. The surface $R$ can also be viewed as properly embedded in $\Sphere\setminus\nhd(L)$, up to ambient isotopy of $\Sphere\setminus\nhd(L)$. We will not explicitly distinguish between these two settings.
\end{definition}

\begin{theorem}[\cite{MR860665} Theorem 4]
Let $L$ be an alternating link. If $R$ is a surface given
by applying Seifert's algorithm to an alternating diagram of $L$, then $R$ is a minimal genus Seifert surface.
\end{theorem}

\begin{definition}[see \cite{przytycki-2010}]
Let $L$ be a link, and let $\ext(L)=\Sphere\setminus\nhd(L)$. Define the \textit{Kakimizu complex} $\ms(L)$ of $L$ to be the following flag simplicial complex. Its vertices are ambient isotopy classes of minimal genus Seifert surfaces for $L$. Two distinct vertices span an edge if they have representatives $R,R'$ such that a lift of $\ext(L)\setminus R'$ to the infinite cyclic cover of $\ext(L)$ intersects exactly two lifts of $\ext(L)\setminus R$.

Define $\is(L)$ to be the analogous complex whose vertices are ambient isotopy classes of incompressible Seifert surfaces for $L$.
\end{definition}

\begin{remark}
If the link $L$ is not split and is not a boundary link, two Seifert surfaces span an edge in $\ms(L)$ or in $\is(L)$ if and only if they can be isotoped to be disjoint. A link is a \textit{boundary link} if it has a disconnected Seifert surface.
\end{remark}

\begin{theorem}[\cite{MR1177053} Theorem A]\label{connectedthm}
$\ms(L)$ and $\is(L)$ are connected.
\end{theorem}

\subsection{Sutured Manifolds}
%Master document is alexpolypaper.tex.

\begin{definition}
A \textit{sutured manifold} $(M,s)$ is a compact, orientable 3--manifold $M$, together with a finite set $s$ of disjoint simple closed curves on $\partial M$, called the \textit{sutures}. The sutures divide $\partial M$ into two (possibly disconnected) compact, oriented surfaces $S_+(M)$ and $S_-(M)$ such that $S_+(M)\cap S_-(M)=s$ and, if $\rho$ is a suture, $S_+(M)$ and $S_-(M)$ meet at $\rho$ with opposite orientations. In addition, for $\rho\in s$ we choose a product neighbourhood $\gamma(\rho)=\rho\times[-1,1]$ of $\rho$ in $\partial M$, so $\gamma(s)$ consists of $|s|$ disjoint annuli.
\end{definition}

\begin{remark}
We could instead first choose suitable annuli $\gamma(s)$, and then take $s$ to be a set of oriented core curves of $\gamma(s)$.
\end{remark}

\begin{definition}
A \textit{product sutured manifold} is a sutured manifold $(M,s)$ that is homeomorphic to $S_+(M)\times[-1,1]$ with $s=\partial S_+(M)\times\{0\}$ (and $\gamma(s)=\partial M\times[-1,1]$).
\end{definition}

\begin{definition}
Let $T$ be a surface properly embedded in $M$ with ${\partial T=s}$. Say $T$ is \textit{parallel} to $S_+(M)$ if there is an embedding $\eta\colon T\times[0,1]\to M$ such that $\eta(\partial T\times[0,1])\subseteq\gamma(s)$, while $\eta(T\times\{0\})=T$ and $\eta(T\times\{1\})={S_+(M)\setminus\Int_{\partial M}(\gamma(s))}$.
\end{definition}

\begin{definition}
A sutured manifold $(M,s)$ is an \textit{almost product sutured manifold} if every incompressible surface $T$ properly embedded in $M$ with $\partial T=s$ is parallel to $S_+(M)$ or to $S_-(M)$.
\end{definition}

\begin{definition}
A disc $T$ properly embedded in a sutured manifold $(M,s)$ is a \textit{product disc} if $\partial T$ meets $s$ at exactly two points, where it crosses $s$ transversely. Up to isotopy of $T$, or of $\gamma(s)$, we may assume $\partial T\cap\gamma(s)$ consists of two simple arcs that are essential in $\gamma(s)$. 
\end{definition}

\begin{definition}
Let $(M,s)$ be a sutured manifold that contains a product disc $T$. Let $\rho$ be a simple arc on $T$ joining the two points of $\partial T\cap s$ and let $T\times[-1,1]$ be a product neighbourhood of $T$ in $M$. The sutured manifold $(M',s')$ obtained from $(M,s)$ by a \textit{product disc decomposition} along $T$ has $M'=M\setminus T\times (-1,1)$ and $s'=(s\cap M')\cup (\rho\times\{\pm 1\})$. Figure \ref{suturedmfldspic1} shows what happens in a neighbourhood of $T$.
\begin{figure}[htbp]
\centering
%LaTeX with PSTricks extensions
%%Creator: 0.46
%%Please note this file requires PSTricks extensions
\psset{xunit=.45pt,yunit=.45pt,runit=.45pt}
\begin{pspicture}(740,220)
{
\newgray{lightgrey}{.8}
\newgray{lightishgrey}{.7}
\newgray{grey}{.6}
\newgray{midgrey}{.4}
\newgray{darkgrey}{.3}
}
{
\pscustom[linewidth=2,linecolor=darkgrey,linestyle=dashed,dash=4 2]
{
\newpath
\moveto(83,60.00000262)
\lineto(170,100.00000262)
}
}
{
\pscustom[linewidth=1,linecolor=black,linestyle=dashed,dash=2 2]%circle
{
\newpath
\moveto(170,90)
\curveto(170,45.84)(149.84,10)(125,10)
\curveto(100.16,10)(80,45.84)(80,90)
\curveto(80,134.16)(100.16,170)(125,170)
\curveto(149.84,170)(170,134.16)(170,90)
\closepath
}
}
{
\pscustom[linewidth=1,linecolor=black,fillstyle=solid,fillcolor=lightishgrey]%disc
{
\newpath
\moveto(207,89.999997)
\curveto(207,45.839997)(186.84,9.999997)(162,9.999997)
\curveto(137.16,9.999997)(117,45.839997)(117,89.999997)
\curveto(117,134.159997)(137.16,169.999997)(162,169.999997)
\curveto(186.84,169.999997)(207,134.159997)(207,89.999997)
\closepath
}
}
{
\pscustom[linewidth=2,linecolor=black,linestyle=dashed,dash=12 4]
{
\newpath
\moveto(54.999997,100.00000262)
\lineto(315,100.00000262)
}
}
{
\pscustom[linewidth=2,linecolor=darkgrey]
{
\newpath
\moveto(120,60.00000262)
\lineto(207,100.00000262)
}
}
{
\pscustom[linewidth=1,linecolor=black]
{
\newpath
\moveto(34.99999,170.00000262)
\lineto(295,170.00000262)
}
}
{
\pscustom[linewidth=1,linecolor=black]
{
\newpath
\moveto(35,9.99998262)
\lineto(295,9.99998262)
}
}
{
\pscustom[linewidth=1,linecolor=black,fillstyle=solid,fillcolor=lightishgrey]%disc
{
\newpath
\moveto(595,90)
\curveto(595,45.84)(574.84,10)(550,10)
\curveto(525.16,10)(505,45.84)(505,90)
\curveto(505,134.16)(525.16,170)(550,170)
\curveto(574.84,170)(595,134.16)(595,90)
\closepath
}
}
{
\pscustom[linewidth=2,linecolor=black,linestyle=dashed,dash=12 4]
{
\newpath
\moveto(474.99999,100.00002262)
\lineto(735,100.00002262)
}
}
{
\pscustom[linewidth=2,linecolor=darkgrey]
{
\newpath
\moveto(508,60.00000262)
\lineto(595,100.00000262)
}
}
{
\pscustom[linewidth=1,linecolor=black]
{
\newpath
\moveto(454.99999,170.00000262)
\lineto(550,170.00000262)
}
}
{
\pscustom[linewidth=1,linecolor=black]
{
\newpath
\moveto(455,9.99998262)
\lineto(550,9.99998262)
}
}
{
\pscustom[linewidth=1,linecolor=black,fillstyle=solid,fillcolor=lightishgrey]%disc
{
\newpath
\moveto(680,90)
\curveto(680,45.84)(659.84,10)(635,10)
\curveto(610.16,10)(590,45.84)(590,90)
\curveto(590,134.16)(610.16,170)(635,170)
\curveto(659.84,170)(680,134.16)(680,90)
\closepath
}
}
{
\pscustom[linewidth=1,linecolor=black]
{
\newpath
\moveto(680,90)
\curveto(680,45.84)(659.84,10)(635,10)
\curveto(610.16,10)(590,45.84)(590,90)
\curveto(590,134.16)(610.16,170)(635,170)
\curveto(659.84,170)(680,134.16)(680,90)
\closepath
}
}
{
\pscustom[linewidth=2,linecolor=darkgrey]
{
\newpath
\moveto(593,60.00000262)
\lineto(680,100.00000262)
}
}
{
\pscustom[linewidth=2,linecolor=black]
{
\newpath
\moveto(593,60.00000262)
\lineto(685,60.00000262)
}
}
{
\pscustom[linewidth=1,linecolor=black]
{
\newpath
\moveto(635,170.00000262)
\lineto(715,170.00000262)
}
}
{
\pscustom[linewidth=1,linecolor=black]
{
\newpath
\moveto(635,9.99998262)
\lineto(715,9.99998262)
}
}
{
\pscustom[linewidth=3,linecolor=black]%arrow
{
\newpath
\moveto(340,100.00000262)
\curveto(360,100.00000262)(420,100.00000262)(420,100.00000262)
}
}
{
\newrgbcolor{curcolor}{0 0 0}
\pscustom[linewidth=3,linecolor=black,fillstyle=solid,fillcolor=black]%arrowhead
{
\newpath
\moveto(390,100.00000262)
\lineto(378,88.00000262)
\lineto(420,100.00000262)
\lineto(378,112.00000262)
\lineto(390,100.00000262)
\closepath
}
}
{
\pscustom[linewidth=1,linecolor=black,fillstyle=solid,fillcolor=black]%arrowhead
{
\newpath
\moveto(455,60.00000262)
\lineto(451,56.00000262)
\lineto(465,60.00000262)
\lineto(451,64.00000262)
\lineto(455,60.00000262)
\closepath
}
}
{
\pscustom[linewidth=1,linecolor=black,fillstyle=solid,fillcolor=black]%arrowhead
{
\newpath
\moveto(650,60.00000262)
\lineto(646,56.00000262)
\lineto(660,60.00000262)
\lineto(646,64.00000262)
\lineto(650,60.00000262)
\closepath
}
}
{
\pscustom[linewidth=1,linecolor=black,fillstyle=solid,fillcolor=black]%arrowhead
{
\newpath
\moveto(60,60.00000262)
\lineto(56,56.00000262)
\lineto(70,60.00000262)
\lineto(56,64.00000262)
\lineto(60,60.00000262)
\closepath
}
}
{
\pscustom[linewidth=1,linecolor=black,fillstyle=solid,fillcolor=black]%arrowhead
{
\newpath
\moveto(715,100.00000262)
\lineto(719,104.00000262)
\lineto(705,100.00000262)
\lineto(719,96.00000262)
\lineto(715,100.00000262)
\closepath
}
}
{
\pscustom[linewidth=1,linecolor=black,fillstyle=solid,fillcolor=black]%arrowhead
{
\newpath
\moveto(490,100.00000262)
\lineto(494,104.00000262)
\lineto(480,100.00000262)
\lineto(494,96.00000262)
\lineto(490,100.00000262)
\closepath
}
}
{
\pscustom[linewidth=1,linecolor=black,fillstyle=solid,fillcolor=black]%arrowhead
{
\newpath
\moveto(85,100.00000262)
\lineto(89,104.00000262)
\lineto(75,100.00000262)
\lineto(89,96.00000262)
\lineto(85,100.00000262)
\closepath
}
}
{
\pscustom[linewidth=1,linecolor=black,linestyle=dashed,dash=2 2]
{
\newpath
\moveto(245,89.999991)
\curveto(245,45.839991)(224.84,9.999991)(200,9.999991)
\curveto(175.16,9.999991)(155,45.839991)(155,89.999991)
\curveto(155,134.159991)(175.16,169.999991)(200,169.999991)
\curveto(224.84,169.999991)(245,134.159991)(245,89.999991)
\closepath
}
}
{
\pscustom[linewidth=2,linecolor=darkgrey,linestyle=dashed,dash=4 2]
{
\newpath
\moveto(160,60.00000262)
\lineto(247,100.00000262)
}
}
{
\pscustom[linewidth=2,linecolor=black]
{
\newpath
\moveto(5,60.00000262)
\lineto(265,60.00000262)
}
}
{
\pscustom[linewidth=2,linecolor=black]
{
\newpath
\moveto(425,60.00002262)
\lineto(508,60.00000262)
}
}
{
\put(133,78){\scriptsize$\rho$}
\put(253,48){\scriptsize$s$}
\put(145,144){\scriptsize$T$}
\put(20,145){\scriptsize$T\!\times\!\{-1\}$}
\put(235,145){\scriptsize$T\!\times\!\{1\}$}
\put(360,120){$T$}
\put(60,180){$M$}
\put(500,180){$M'$}
}
\end{pspicture}
\caption{\label{suturedmfldspic1}}
\end{figure}
\end{definition}

\begin{proposition}[see \cite{MR1026928} Lemmas 2.1, 2.2]
Let $(M,s)\stackrel{T}{\rightarrow}(M',s')$ be a product disc decomposition. Then $M$ is an almost product sutured manifold if and only if $M'$ is.
\end{proposition}

\begin{definition}
For a sutured manifold $(M,s)$ embedded in $\Sphere$, the \textit{complementary sutured manifold} $(M',s')$ is defined by $M'=\Sphere\setminus \Int_{\Sphere}(M)$ and $s'=s$.

By the \textit{complementary sutured manifold to a Seifert surface $R$} we mean the complementary sutured manifold to the product sutured manifold given by a product neighbourhood of $R$.
\end{definition}

\begin{remark}
Let $(M,s_M)$ be the complementary sutured manifold to a minimal genus/incompressible Seifert surface $R$. By Theorem \ref{connectedthm}, $M$ is an almost product sutured manifold if and only if $R$ is unique.
\end{remark}
\subsection{Graphs}
%Master document is alexpolypaper.tex.

\begin{definition}
A \textit{graph} $\mathcal{G}$ consists of a set of vertices, denoted $\V(\mathcal{G})$, a set of edges, denoted $\E(\mathcal{G})$, and a function $\Ends{}$ that assigns to each edge $e\in\E(\mathcal{G})$ two vertices, called the endpoints of $e$.
\end{definition}

\begin{convention}
Unless otherwise stated, we assume $\V(G)$ and $\E(G)$ are finite. In general we allow a graph to contain multiedges (distinct $e,e'\in\E(G)$ with $\Ends{}(e)=\Ends{}(e')$) and loops ($e\in\E(G)$ whose two endpoints are the same). By convention these are usually excluded in the definition of the term `graph', but we will need them later.
\end{convention}

\begin{convention}
We will always assume a graph to be connected (although we may consider subgraphs that are disconnected).
\end{convention}

\begin{definition}
Given a set $A\subseteq\V(\mathcal{G})$, the \textit{induced subgraph} $\mathcal{G}[A]$ is the graph with vertex set $A$ and edge set $\{e\in\E(\mathcal{G}):\Ends{}(e)\subseteq A\}$.

For $B\subseteq\E(\mathcal{G})$, denote by $\mathcal{G}\setminus B$ the graph obtained by deleting all edges of $B$ from $\mathcal{G}$. That is, $\V(\mathcal{G}\setminus B)=\V(\mathcal{G})$ and $\E(\mathcal{G}\setminus B)=\E(\mathcal{G})\setminus B$.

Given $e\in\E(\mathcal{G})$, $\mathcal{G}/ e$ is the graph obtained by contracting $e$ to a point. This means $\V(\mathcal{G}/ e)=\left(\V(\mathcal{G})\setminus \Ends{}(e\right))\cup\{v_e\}$ where $v_e\notin\V(\mathcal{G})$, while $\E(\mathcal{G}/ e)=\E(\mathcal{G})\setminus e$ and $v_e$ replaces both ends of $e$ in $\Ends{}$. If $B=\{e_1,\cdots,e_n\}$ for some $n\in\mathbb{N}$ and $e_i\in\E(\mathcal{G})$ then $\mathcal{G}/B=((\cdots(\mathcal{G}/e_1)/e_2)\cdots)/e_n$.
\end{definition}

\begin{definition}
A \textit{pointed graph} $(\mathcal{G},v)$ is a graph $\mathcal{G}$ with a distinguished vertex $v$.
\end{definition}

\begin{definition}
For $u,v\in\V(\mathcal{G})$, the distance $\dist(u,v)$ is the minimum length of a path between $u$ and $v$. 
The \textit{radius} of a pointed graph $(\mathcal{G},v)$ is $\max\{\dist(v,w):w\in\V(\mathcal{G})\}$.
\end{definition}

\begin{convention}
For $n\in\mathbb{N}$ and $v\in\V(\mathcal{G})$, denote by $\ball{\mathcal{G}}{v}{n}$ the digraph $\mathcal{G}[A_n]$, where $A_n=\{w\in\V(\mathcal{G}):\dist(v,w)\leq n\}$.
\end{convention}

\begin{definition}
A \textit{digraph} $(\mathcal{G},\mathcal{O})$ is a graph $\mathcal{G}$ together with an orientation $\mathcal{O}$, which assigns to each $e\in\E(\mathcal{G})$ an initial endpoint $\iota(e)$ and a terminal endpoint $\tau(e)$ such that $\{\iota(e),\tau(e)\}=\Ends{}(e)$. We say that $e$ starts at $\iota(e)$ and ends at $\tau(e)$.

Define the \textit{in-degree} of a vertex $v\in\V(G)$ to be the number of edges $e\in\E(G)$ with $\tau(e)=v$, and define the \textit{out-degree} analogously. 

$\mathcal{G}$ is called the \textit{underlying graph} of $(\mathcal{G},\mathcal{O})$. We will at times consider more than one orientation on the same graph $\mathcal{G}$. Where the choice of orientation is clear, or not important, we will denote $(\mathcal{G},\mathcal{O})$ by $\mathcal{G}$.
\end{definition}

\begin{definition}
A \textit{directed path} in $G$ is a path $v_0,e_1,\cdots,e_n,v_n$ such that $\iota(e_i)=v_{i-1}$ and $\tau(e_i)=v_i$ for $1\leq i \leq n$.

$\mathcal{G}$ is \textit{\ocon{\mathcal{O}}{}} if for any $u,v\in\V(\mathcal{G})$ there is a directed path in $(\mathcal{G},\mathcal{O})$ from $u$ to $v$.

A \textit{cycle} is a directed path $v_0,e_1,\dots,e_n,v_n$ with $v_0=v_n$.

Two directed paths $v_0,e_1,\cdots,e_n,v_n$ and $v'_0,e'_1,\cdots,e'_m,v'_m$ are said to be \textit{edge-disjoint} if there do not exist $n_0,m_0$ with $e_{n_0}=e'_{m_0}$.
\end{definition}

\begin{convention}
A digraph is \textit{planar} if it has an embedding into $\sphere$.
We shall regard this embedding as fixed (some authors call such a graph `plane').
\end{convention}

\begin{definition}
Given a planar graph $\mathcal{G}$, we may define the \textit{dual graph} $\mathcal{G}^*$, which is again planar. It has a vertex for each region of $\sphere\setminus\mathcal{G}$. There is one edge $e'$ in $\mathcal{G}^*$ for each $e\in\E(\mathcal{G})$, joining the vertices corresponding to the regions of $\sphere\setminus\mathcal{G}$ adjacent to $e$.
\end{definition}

\begin{definition}
Given a link $L$ with diagram $D$, the \textit{underlying graph} $\mathcal{G}$ has a vertex at each crossing in $D$, and an edge for each arc in $D$ joining two crossings.
The \textit{induced orientation} $O$ is that given by the orientation of the link $L$.
We will later put other orientations on the underlying graph.
\end{definition}

\begin{remark}
$\mathcal{G}$ is planar.
We can reconstruct $D$ from $(\mathcal{G},\mathcal{O})$, with its embedding into $\sphere$, provided we also know, for each crossing, which arc is the overcrossing and which is the undercrossing.
\end{remark}
\subsection{Special alternating links and Murasugi sums}
%Master document is alexpolypaper.tex.

\begin{definition}
A \textit{Seifert circle} $C$ in a diagram $D$ is any of the simple closed curves in $\sphere$ created by Seifert's algorithm.
$C$ may also be seen as a cycle in the underlying digraph $(\mathcal{G},O)$ that turns at every crossing it meets (the direction it turns will always be determined by $O$). We will not explicitly distinguish between these viewpoints.
\end{definition}

\begin{definition}
Let $D$ be a diagram of a link $L$. A Seifert circle $C$ in $D$ is \textit{special} if it bounds a disc in $\sphere\setminus D$. The diagram $D$ is \textit{special} if every Seifert circle in $D$ is special. $L$ is \textit{special} if some diagram of $L$ is special. 
\end{definition}

\begin{remark}
A special, alternating link diagram is either positive or negative. That is, either every crossing is positive, or every crossing is negative.
\end{remark}

Let $C$ be a non-special Seifert circle in $D$. We can split $D$ along $C$ to create two new non-trivial link diagrams $D_1$ and $D_2$ as follows. View $C$ as a simple closed curve in the underlying digraph $(\mathcal{G},O)$. Let $S$ be one component of $\sphere\setminus C$ (so $S$ is an open disc). Let $A=\{v\in V(\mathcal{G}):v\notin S\}$. Then all vertices of $\mathcal{G}[A]$ are 4--valent except for some of the vertices on $C$, which are 2--valent. Since $C$ is a cycle, such a vertex has in-degree and out-degree 1. Let $(\mathcal{G}_1,O_1)$ be the digraph obtained by contracting each edge of $\mathcal{G}[A]$ whose terminal vertex is 2--valent. Now take $(\mathcal{G}_1,O_1)$ as the underlying graph of $D_1$. The choice of undercrossing and overcrossing arcs at a crossing $c$ in $D_1$ is induced by that at $c$ in $D$. The diagram $D_2$ is defined analogously from the other component of $\sphere\setminus C$.

If $D$ is alternating, then so are $D_1$ and $D_2$ (see Figure \ref{speciallinkspic1}).
\begin{figure}[htbp]
\centering
%LaTeX with PSTricks extensions
%%Creator: 0.46
%%Please note this file requires PSTricks extensions
\psset{xunit=.45pt,yunit=.45pt,runit=.45pt}
\begin{pspicture}(630,200)
{
\pscustom[linewidth=1,linecolor=black]
{
\newpath
\moveto(25,85.00002262)
\curveto(26.227824,87.45566262)(68.13467,80.65896262)(80,95.00002262)
\curveto(94.03133,111.95902262)(100.00523,164.99610262)(100,165.00002262)
}
}
{
\pscustom[linewidth=1,linecolor=black]
{
\newpath
\moveto(50,165.00002262)
\curveto(50,165.00002262)(53.675445,122.64913262)(60,110.00002262)
\curveto(65,100.00002262)(70,95.00002262)(70,95.00002262)
}
}
{
\pscustom[linewidth=1,linecolor=black]
{
\newpath
\moveto(110,9.99998262)
\curveto(110,9.99998262)(113.67545,52.35088262)(120,65.00002262)
\curveto(125,75.00002262)(130,80.00002262)(130,80.00002262)
}
}
{
\pscustom[linewidth=1,linecolor=black]
{
\newpath
\moveto(170,165.00002262)
\curveto(170,165.00002262)(173.67544,122.64913262)(180,110.00002262)
\curveto(185,100.00002262)(190,95.00002262)(190,95.00002262)
}
}
{
\pscustom[linewidth=1,linecolor=black]
{
\newpath
\moveto(80,90.00002262)
\curveto(80,85.16999262)(127.49925,92.96255262)(140,80.00000262)
\curveto(155.27937,64.15618262)(160.00523,10.00398262)(160,9.99998262)
}
}
{
\pscustom[linewidth=1,linecolor=black]
{
\newpath
\moveto(140,85.00002262)
\curveto(140,89.83005262)(187.49925,82.03749262)(200,95.00004262)
\curveto(215.27937,110.84386262)(220.00523,164.99604262)(220,165.00004262)
}
}
{
\pscustom[linewidth=1,linecolor=black]
{
\newpath
\moveto(200,90.00002262)
\curveto(205,85.00002262)(240,85.00002262)(240,85.00002262)
}
}
{
\pscustom[linewidth=3,linecolor=black,linestyle=dashed,dash=3 3]
{
\newpath
\moveto(5,85.00000262)
\lineto(25,85.00000262)
}
}
{
\pscustom[linewidth=3,linecolor=black,linestyle=dashed,dash=3 3]
{
\newpath
\moveto(260,85.00000262)
\lineto(240,85.00000262)
}
}
{
\pscustom[linewidth=3,linecolor=black]
{
\newpath
\moveto(25,85.00000262)
\curveto(65,85.00000262)(75,90.00000262)(75,90.00000262)
}
}
{
\pscustom[linewidth=3,linecolor=black]
{
\newpath
\moveto(79,90.00000262)
\curveto(84.349889,86.74945262)(103.0777,87.96132262)(110,88.00000262)
\curveto(126.90448,88.09445262)(133,84.00000262)(135,83.00000262)
}
}
{
\pscustom[linewidth=3,linecolor=black]
{
\newpath
\moveto(139,85.00000262)
\curveto(144.34989,88.25055262)(163.0777,87.03868262)(170,87.00000262)
\curveto(186.90448,86.90555262)(193,91.00000262)(195,92.00000262)
}
}
{
\pscustom[linewidth=3,linecolor=black]
{
\newpath
\moveto(200,90.00000262)
\curveto(207,84.00000262)(239.3242,84.50179262)(241,85.00000262)
}
}
{
\pscustom[linewidth=1,linecolor=black,fillstyle=solid,fillcolor=black]%arrowhead
{
\newpath
\moveto(216.96116135,141.80580937)
\lineto(213.82330319,146.51259662)
\lineto(215,132.00000262)
\lineto(221.66794859,144.94366754)
\lineto(216.96116135,141.80580937)
\closepath
}
}
{
\pscustom[linewidth=1,linecolor=black,fillstyle=solid,fillcolor=black]%arrowhead
{
\newpath
\moveto(95.96116135,142.80580937)
\lineto(92.82330319,147.51259662)
\lineto(94,133.00000262)
\lineto(100.66794859,145.94366754)
\lineto(95.96116135,142.80580937)
\closepath
}
}
{
\pscustom[linewidth=1,linecolor=black,fillstyle=solid,fillcolor=black]%arrowhead
{
\newpath
\moveto(113.96116135,38.80578937)
\lineto(110.82330319,43.51257662)
\lineto(112,28.99998262)
\lineto(118.66794859,41.94364754)
\lineto(113.96116135,38.80578937)
\closepath
}
}
{
\pscustom[linewidth=1,linecolor=black,fillstyle=solid,fillcolor=black]%arrowhead
{
\newpath
\moveto(53.96116324,136.19418624)
\lineto(58.66795108,133.05632898)
\lineto(52,145.99999262)
\lineto(50.82330598,131.48739839)
\lineto(53.96116324,136.19418624)
\closepath
}
}
{
\pscustom[linewidth=1,linecolor=black,fillstyle=solid,fillcolor=black]%arrowhead
{
\newpath
\moveto(173.96116324,136.19418624)
\lineto(178.66795108,133.05632898)
\lineto(172,145.99999262)
\lineto(170.82330598,131.48739839)
\lineto(173.96116324,136.19418624)
\closepath
}
}
{
\pscustom[linewidth=1,linecolor=black,fillstyle=solid,fillcolor=black]%arrowhead
{
\newpath
\moveto(156.96116135,34.19417586)
\lineto(161.66794859,31.0563177)
\lineto(155,43.99998262)
\lineto(153.82330319,29.48738862)
\lineto(156.96116135,34.19417586)
\closepath
}
}
{
\pscustom[linewidth=1,linecolor=black]
{
\newpath
\moveto(385,85.00000262)
\curveto(386.22782,87.45564262)(428.13467,80.65894262)(440,95.00000262)
\curveto(454.03133,111.95900262)(460.00523,164.99608262)(460,165.00000262)
}
}
{
\pscustom[linewidth=1,linecolor=black]
{
\newpath
\moveto(410,165.00000262)
\curveto(410,165.00000262)(413.67545,122.64911262)(420,110.00000262)
\curveto(425,100.00000262)(430,95.00000262)(430,95.00000262)
}
}
{
\pscustom[linewidth=1,linecolor=black]
{
\newpath
\moveto(530,165.00000262)
\curveto(530,165.00000262)(533.67544,122.64911262)(540,110.00000262)
\curveto(545,100.00000262)(550,95.00000262)(550,95.00000262)
}
}
{
\pscustom[linewidth=1,linecolor=black]
{
\newpath
\moveto(580,165.00002262)
\curveto(580.00523,164.99602262)(575.27937,110.84384262)(560,95.00002262)
\curveto(547.49925,82.03747262)(455,80.00000262)(440,90.00000262)
}
}
{
\pscustom[linewidth=1,linecolor=black]
{
\newpath
\moveto(560,90.00000262)
\curveto(565,85.00000262)(600,85.00000262)(600,85.00000262)
}
}
{
\pscustom[linewidth=3,linecolor=black,linestyle=dashed,dash=3 3]
{
\newpath
\moveto(365,84.99998262)
\lineto(385,84.99998262)
}
}
{
\pscustom[linewidth=3,linecolor=black,linestyle=dashed,dash=3 3]
{
\newpath
\moveto(620,84.99998262)
\lineto(600,84.99998262)
}
}
{
\pscustom[linewidth=3,linecolor=black]
{
\newpath
\moveto(385,84.99998262)
\curveto(425,84.99998262)(435,89.99998262)(435,89.99998262)
}
}
{
\pscustom[linewidth=3,linecolor=black]
{
\newpath
\moveto(440,90.00000262)
\curveto(453,78.70247262)(556.4109,84.38733262)(555,92.00000262)
}
}
{
\pscustom[linewidth=3,linecolor=black]
{
\newpath
\moveto(560,89.99998262)
\curveto(567,83.99998262)(599.3242,84.50177262)(601,84.99998262)
}
}
{
\pscustom[linewidth=1,linecolor=black,fillstyle=solid,fillcolor=black]%arrowhead
{
\newpath
\moveto(576.96116135,141.80578937)
\lineto(573.82330319,146.51257662)
\lineto(575,131.99998262)
\lineto(581.66794859,144.94364754)
\lineto(576.96116135,141.80578937)
\closepath
}
}
{
\pscustom[linewidth=1,linecolor=black,fillstyle=solid,fillcolor=black]%arrowhead
{
\newpath
\moveto(455.96116135,142.80578937)
\lineto(452.82330319,147.51257662)
\lineto(454,132.99998262)
\lineto(460.66794859,145.94364754)
\lineto(455.96116135,142.80578937)
\closepath
}
}
{
\pscustom[linewidth=1,linecolor=black,fillstyle=solid,fillcolor=black]%arrowhead
{
\newpath
\moveto(413.96116324,136.19416624)
\lineto(418.66795108,133.05630898)
\lineto(412,145.99997262)
\lineto(410.82330598,131.48737839)
\lineto(413.96116324,136.19416624)
\closepath
}
}
{
\pscustom[linewidth=1,linecolor=black,fillstyle=solid,fillcolor=black]%arrowhead
{
\newpath
\moveto(533.96116324,136.19416624)
\lineto(538.66795108,133.05630898)
\lineto(532,145.99997262)
\lineto(530.82330598,131.48737839)
\lineto(533.96116324,136.19416624)
\closepath
}
}
{
\pscustom[linewidth=2,linecolor=black]%arrow
{
\newpath
\moveto(280,85.00000262)
\lineto(345,85.00000262)
}
}
{
\pscustom[linewidth=2,linecolor=black,fillstyle=solid,fillcolor=black]%arrowhead
{
\newpath
\moveto(325,85.00000262)
\lineto(317,77.00000262)
\lineto(345,85.00000262)
\lineto(317,93.00000262)
\lineto(325,85.00000262)
\closepath
}
}
{
\put(50,60){$C$}
\put(410,60){$C$}
\put(10,170){$D$}
\put(360,170){$D_1$}
}
\end{pspicture}
\caption{\label{speciallinkspic1}}
\end{figure}

\begin{definition}[\cite{MR0137107} 3.1]\label{starproddefn}
Let $L$ be a link with a diagram $D$ that can be split into diagrams $D_1,D_2$ as above. Let $L_1,L_2$ be the links with diagrams $D_1,D_2$ respectively. We say that $L$ is the \textit{$*$--product} of $L_1$ and $L_2$, written $L=L_1 * L_2$. Note that $D_1$ and $D_2$ alone do not tell us how to construct $D$.
\end{definition}

\begin{remark}
Let $C$ be the non-special Seifert circle along which $D$ was split. For $i=1,2$, let $R_i$ be the Seifert surface for $L_i$ given by applying Seifert's algorithm to $D_i$, and let $S_i$ be the disc in $R_i$ bounded by $C$. Let $R$ be given by identifying $S_1$ and $S_2$. Then $R$ is given by applying Seifert's algorithm to $D$.

The interaction of $*$--product with link diagrams makes it a useful tool in \cite{MR0099665} and \cite{MR0099664}, and below.
\end{remark}

\begin{definition}[\cite{MR1002465} p536]\label{homogeneousdefn}
A link diagram $D$ is \textit{homogeneous} if it is formed from special alternating link digrams by taking connected sums as in Lemma \ref{connectedsumlemma} and $*$--products.
A link $L$ is \textit{homogeneous} if it has a homogeneous diagram.
\end{definition}

\begin{proposition}[\cite{MR1002465} Corollary 4.1]
The Seifert surface given by applying Seifert's algorithm to a homogeneous diagram of a link is minimal genus.
\end{proposition}

\begin{proposition}[\cite{MR1002465} Corollary 3.1, Theorem 8]
A link with a connected homogeneous diagram is not split and is not a boundary link.
\end{proposition}

Definition \ref{starproddefn} has since been generalised as follows.

\begin{definition}\label{msumdefn}
Let $R,R_1,R_2$ be Seifert surfaces for links $L,L_1,L_2$ respectively. $R$ is the \textit{Murasugi sum} of $R_1$ and $R_2$ if the following hold.
\begin{itemize}
	\item There is a 2--sphere $S\subset\Sphere$ dividing $\Sphere$ into two closed 3--balls $V_1$ and $V_2$.
	\item $R_1\subset V_1$, $R_2\subset V_2$ and $R=R_1\cup R_2$.
	\item $T=R_1\cap S=R_2\cap S$ is a closed disc.
	\item $T$ is a $2n$--gon for some $n\in\mathbb{N}$. That is, $\partial T$ consists of $2n$ simple arcs $\rho^1_1,\rho^2_1,\cdots,\rho^1_n,\rho^2_n$ such that, for all $i$, the arc $\rho^1_i$ is part of $L_1=\partial R_1$ and properly embedded in $R_2$ whereas $\rho^2_i$ is part of $L_2=\partial R_2$ and properly embedded in $R_1$.
\end{itemize}

When $n=2$, this operation is known as \textit{plumbing}.
\end{definition}

\begin{remark}
The connected sum of two links can be seen as a Murasugi sum, for example by taking $n=1$.
\end{remark}

\subsection{Controlling surfaces under Murasugi summation}
%Master document is alexpolypaper.tex.

When one of the surfaces involved is a Hopf band, plumbing becomes a fairly rigid process, with product disc decomposition providing a reverse operation as follows. 
See \cite{coward-2008} for the definitions.

\begin{theorem}[\cite{coward-2008} Theorem 2.3]
Let $R$ be a Seifert surface for a link $L$. Then there is a bijective correspondence between the following:

(i) decompositions of $R$, up to equivalence, as the plumbing of two surfaces, where the first surface is a Hopf band;

(ii) clean alternating directed product discs for $R$, up to ambient isotopy that leaves $\nhd(R)$ invariant and maintains the disc as a product disc throughout.
\end{theorem}

Let $D$ be a special, alternating diagram of a link $L$. We may colour the regions of $D$ in a checkerboard pattern by making the inside of each Seifert circle black and colouring the remaining regions white. 
Then the Seifert surface $R$ given by applying Seifert's algorithm to $D$ is formed from the black regions. If a region $r$ of $D$ is a white bigon, it defines a product disc in the complement of $R$. 
The effect on $D$ of the product disc decomposition along this disc is to remove the region $r$, replacing the two crossings of $r$ with a single crossing. Such a change to $D$ therefore has the effect of pulling off a Hopf band from $R$.

The behaviour of Seifert surfaces under Murasugi summation depends on whether or not the links involved are fibred.

\begin{theorem}[\cite{MR1026928} Theorem 5.1]\label{msumthm}
Let $L_1,L_2$ be links with minimal genus Seifert surfaces $R_1,R_2$ respectively, and let $R$ be a Murasugi sum of $R_1$ and $R_2$. Then $L=\partial R$ has a unique minimal genus Seifert surface if and only if $L_1,L_2$ each have a unique minimal genus Seifert surface and $L_i$ is fibred with fibre $R_i$ for either $i=1$ or $i=2$.
\end{theorem}

\begin{proposition}[\cite{MR2131376} Proposition 2.3]\label{fibredsumprop}
$L$ is fibred with fibre $R$ if and only if both $L_1$ and $L_2$ are fibred with fibres $R_1$ and $R_2$ respectively.
\end{proposition}

\begin{definition}\label{dualdefn}
If $R$ is the Murasugi sum of $R_1, R_2$ as in Definition \ref{msumdefn}, let $T'=S\setminus\Int_S(T)$. Then $R'=(R\setminus T)\cup T'$ is another Seifert surface for $L$ (see Figure \ref{incompressiblepic1}). Kakimizu \cite{MR2131376} calls $R'$ a \textit{dual} of $R$.
\begin{figure}[htbp]
\centering
%LaTeX with PSTricks extensions
%%Creator: 0.46
%%Please note this file requires PSTricks extensions
\psset{xunit=.5pt,yunit=.5pt,runit=.5pt}
\begin{pspicture}(340,300)
{
\newgray{lightgrey}{.8}
\newgray{lightishgrey}{.7}
\newgray{grey}{.6}
\newgray{midgrey}{.4}
\newgray{darkgrey}{.3}
}
{
\pscustom[linestyle=none,fillstyle=solid,fillcolor=grey]
{
\newpath
\moveto(200,9.99998262)
\curveto(198,26.99998262)(200,69.99999262)(200,69.99999262)
\lineto(101,69.99999262)
\curveto(99.672786,50.11038262)(99.851226,30.05358262)(100,9.99998262)
\curveto(100,9.99998262)(197,9.99998262)(200,9.99998262)
\closepath
}
}
{
\pscustom[linestyle=none,fillstyle=solid,fillcolor=grey]
{
\newpath
\moveto(268,49.99998262)
\curveto(266.76669,95.61667262)(266.08146,99.95927262)(266,100.00000262)
\lineto(166,100.00000262)
\curveto(166.92371,82.82297262)(168.18988,66.50178262)(168,49.99998262)
\lineto(268,49.99998262)
\closepath
}
}
{
\pscustom[linestyle=none,fillstyle=solid,fillcolor=lightgrey]
{
\newpath
\moveto(118.14913626,230.00002603)
\lineto(356.29826041,230.00002603)
\lineto(258.14914953,60.00000573)
\lineto(20.00002539,60.00000573)
\lineto(118.14913626,230.00002603)
\closepath
}
}
{
\pscustom[linewidth=1,linecolor=black,linestyle=dashed,dash=2 4]
{
\newpath
\moveto(118.14913626,230.00002603)
\lineto(356.29826041,230.00002603)
\lineto(258.14914953,60.00000573)
\lineto(20.00002539,60.00000573)
\lineto(118.14913626,230.00002603)
\closepath
}
}
{
\pscustom[linewidth=1,linecolor=black]
{
\newpath
\moveto(213,120.00002262)
\curveto(196.96487,120.56149262)(199.2089,11.38448262)(200,9.99998262)
}
}
{
\pscustom[linestyle=none,fillstyle=solid,fillcolor=grey]
{
\newpath
\moveto(148,179.00000262)
\curveto(148,179.00000262)(113,120.00000262)(113,120.00000262)
\curveto(95.98589,128.79949262)(79.67373,236.64365262)(82,240.00000262)
\lineto(117,300.00000262)
\curveto(123.60605,254.82734262)(127.57852,203.46249262)(148,179.00000262)
\closepath
}
}
{
\pscustom[linewidth=3,linecolor=black]
{
\newpath
\moveto(112,120.00000262)
\curveto(92,140.00000262)(82,240.00000262)(82,240.00000262)
}
}
{
\pscustom[linestyle=none,fillstyle=solid,fillcolor=midgrey]
{
\newpath
\moveto(147.32037699,180.00000452)
\lineto(247.32037699,180.00000452)
\lineto(212.67951882,120.00000512)
\lineto(112.67951882,120.00000512)
\lineto(147.32037699,180.00000452)
\closepath
}
}
{
\pscustom[linewidth=3,linecolor=black]
{
\newpath
\moveto(147,180.00000262)
\curveto(127,200.00000262)(117,300.00000262)(117,300.00000262)
}
}
{
\pscustom[linewidth=1,linecolor=black]
{
\newpath
\moveto(248,180.00002262)
\curveto(268,180.00002262)(268,49.99998262)(268,49.99998262)
}
}
{
\pscustom[linewidth=1,linecolor=black]
{
\newpath
\moveto(113,120.00004262)
\curveto(96.96487,120.56151262)(99.2089,11.38448262)(100,9.99998262)
}
}
{
\pscustom[linewidth=3,linecolor=black]
{
\newpath
\moveto(112,120.00000262)
\lineto(213,120.00000262)
}
}
{
\pscustom[linewidth=3,linecolor=black]
{
\newpath
\moveto(147,180.00000262)
\lineto(248,180.00000262)
}
}
{
\pscustom[linewidth=1,linecolor=black,linestyle=dashed,dash=4 2]
{
\newpath
\moveto(147,180.00000262)
\curveto(167,180.00000262)(167,49.99998262)(167,49.99998262)
}
}
{
\pscustom[linestyle=none,fillstyle=solid,fillcolor=grey]
{
\newpath
\moveto(212,120.00000262)
\curveto(222.61747,122.17334262)(227.90667,137.14599262)(228,137.00000262)
\lineto(263,197.00000262)
\lineto(247,180.00000262)
\lineto(212,120.00000262)
\closepath
}
}
{
\pscustom[linestyle=none,fillstyle=solid,fillcolor=grey]
{
\newpath
\moveto(234,155.00000262)
\curveto(242.73572,181.73260262)(243.72168,239.66517262)(243,240.00000262)
\lineto(278,300.00000262)
\curveto(276.2867,234.48416262)(268.51167,206.82091262)(259,190.00000262)
\curveto(249.19472,181.34213262)(241.90363,166.82728262)(234,155.00000262)
\closepath
}
}
{
\pscustom[linewidth=1,linecolor=black,linestyle=dashed,dash=4 2]
{
\newpath
\moveto(247,180.00000262)
\lineto(150,180.00000262)
}
}
{
\pscustom[linewidth=1,linecolor=black,linestyle=dashed,dash=4 2]
{
\newpath
\moveto(248,180.00000262)
\curveto(278,190.00000262)(278,300.00000262)(278,300.00000262)
}
}
{
\pscustom[linewidth=3,linecolor=black]
{
\newpath
\moveto(213,120.00000262)
\curveto(243,130.00000262)(243,240.00000262)(243,240.00000262)
}
}
{
\pscustom[linewidth=3,linecolor=black]
{
\newpath
\moveto(255,185.00000262)
\curveto(278,207.00000262)(277.7355,298.81472262)(278,300.00000262)
}
}
{
\pscustom[linewidth=1,linecolor=black,linestyle=dashed,dash=4 2]
{
\newpath
\moveto(248,180.00000262)
\curveto(268,180.00000262)(268,49.99998262)(268,49.99998262)
}
}
{
\put(290,170){$T'$}
\put(60,200){$L_1$}
\put(95,245){$R_1$}
\put(150,20){$R_2$}
\put(180,140){$T$}
}
\end{pspicture}
\caption{\label{incompressiblepic1}}
\end{figure}
\end{definition}

$R'$ is equivalent to $R$ exactly if $L_i$ is fibred with fibre $R_i$ for either $i=1$ or $i=2$ (see \cite{MR2131376} Propositions 2.4, 2.5), which is part of the reason for Theorem \ref{msumthm}. 
In the context of Murasugi sums arising from non-special diagrams, this dual corresponds to a different choice of ordering of the heights of the discs in $\Sphere$ when applying Seifert's algorithm. For a special diagram, no such choice is available.

We wish to prove the `if' direction of Theorem \ref{msumthm} for incompressible surfaces.

First we consider in more detail the effect of the Murasugi sum operation on the complementary sutured manifolds to the surfaces involved. 
We retain the notation of Definitions \ref{msumdefn}, \ref{dualdefn}. 
Let $\nhd(R)=R\times[1,2]$ be a product neighbourhood of $R$ in $\Sphere$, and $\nhd(R_1),\nhd(R_2)$ product neighbourhoods of $R_1,R_2$ respectively, with $T\times\{1\}\subseteq\V_1$ in each case. Let $(M,s),(M_1,s_1),(M_2,s_2)$ be the corresponding complementary sutured manifolds. Then, for example, $\gamma(s)=\partial R\times[1,2]$. In addition, for $i=1,2$, let $(M'_i,s'_i)$ be the sutured manifold with $M'_i=M\cap V_i$ and $\gamma(s'_i)=\gamma(s)\cap V_i$.

We see that $(M'_i,s'_i)$ is homeomorphic to $(M_i,s_i)$ as a sutured manifold for each $i$, and that $(M,s)$ is given by gluing $(M'_1,s'_1)$ to $(M'_2,s'_2)$ along $T'$. The homeomorphism from $M'_1$ to $M_1$ takes $T'$ to $T\times\{2\}$. Thus, as seen from $M_1$, the Murasugi sum is given by attaching another manifold along $T\times\{2\}\subseteq\partial M_1$. This is clearly unaffected by any changes made to $M_1$ that leave a neighbourhood of $T\times\{2\}$ unchanged.

\begin{remark}
If $R_2$ is a disc then $R$ is ambient isotopic to $R_1$ in $\Sphere$.
\end{remark}

Now focus on $L_2$, and suppose it is fibred. Then $M_2$ has a product structure $R_2\times[1,2]$. Choose this product structure so that $R_2\times\{2\}\subseteq M_2$ is identified with $R_2\times\{2\}\subseteq \nhd(R_2)$ by the identity map. In general, the identification between $R_2\times\{1\}\subseteq M_2$ and $R_2\times\{1\}\subseteq \nhd(R_2)$ will not be the identity.
For an arc $\rho$ properly embedded in $R_2$, the surface $\rho\times[1,2]$ forms a product disc in $M_2$.

\begin{lemma}
Let $S$ be a connected, compact, orientable surface that is not a disc. Let $T\subset S$ be a $2n$--gon in $S$ for some $n\in\mathbb{N}$. That is, $T$ is an embedded disc and  $\partial T$ consists of $2n$ simple arcs $\rho^1_1,\rho^2_1,\cdots,\rho^1_n,\rho^2_n$ such that, for all $i$, the arc $\rho^1_i$ is part of $\partial S$ whereas $\rho^2_i$ is properly embedded in $S$. Then there is a non-separating arc properly embedded in $S$ that is disjoint from $T$.
\end{lemma}
\begin{proof}
Suppose $\rho^2_m$ is non-separating for some $m\leq n$. Up to isotopy, this arc can be made disjoint from $T$, and we are done.

Suppose instead that $\rho^2_i$ is separating for each $i\leq n$. Cut $S$ along each $\rho^2_i$, and let $S'$ be a component of the resulting surface that is not a disc (since each arc we have cut along is separating in $S$, and $T$ is a disc while $S$ is not, at least one such component exists). Let $\rho'$ be a non-separating arc properly embedded in $S'$. In $S$, the arc $\rho'$ may have one or both of its endpoints on $\partial T$. If $\rho'(0)$ lies on $\rho^2_m$, add part of $\rho^2_m$ to the start of $\rho'$. Similarly add part of $\partial T$ to the end of $\rho'$ if needed. This gives an arc $\rho''$ that is properly embedded in $S$ and is disjoint from $T$.  Note that if both ends of $\rho'$ lie on some $\rho^2_i$, we can clearly still ensure that $\rho''$ is embedded. Since $\rho''$ is isotopic to $\rho'$ in $S'$, it is non-separating in $S$.
\end{proof}

\begin{corollary}\label{incompcor}
Let $L_1,L_2$ be links with incompressible Seifert surfaces $R_1,R_2$ respectively, such that $L_2$ is fibred with fibre $R_2$. Let $R$ be a Murasugi sum of $R_1$ and $R_2$, and let $L=\partial R$. Then $L$ has a unique incompressible Seifert surface if and only if $L_1$ does.
\end{corollary}
\begin{proof}
Let $T\subseteq R_2$ be the $2n$--gon along which $R_1$ and $R_2$ are joined in the Murasugi sum. Inductively construct disjoint arcs $\sigma_1,\cdots,\sigma_m$, all properly embedded in $R_2$ and disjoint from $T$, such that cutting $R_2$ along $\sigma_1,\cdots,\sigma_m$ gives a single disc.

Let $M_1,M_2$ be the complementary sutured manifolds constructed from $R_1,R_2$ respectively, and let $M$ be obtained by gluing $M_1$ and $M_2$ as described above. Then $M$ is the complementary sutured manifold given by $R$.
Choose a product structure $R_2\times[1,2]$ for $M_2$ as above.
Let $\mathbb{D}_i$ be the product disc $\sigma_i\times[1,2]$ in $M_2$ for each $i$. Since $\sigma_i$ is disjoint from $T$, the disc $\mathbb{D}_i$ is disjoint from a neighbourhood of $T\times\{2\}$ in $M_2\subseteq M$, and so forms a product disc in $M$.

Let $M'$ be the sutured manifold given by decomposing $M$ along each of the discs $\mathbb{D}_i$. Then $M'$ is an almost product sutured manifold if and only if $M$ is. Since $\sigma_1,\cdots,\sigma_m$ divide $R_2$ into a disc, we see that $M'$ is constructed by gluing $M_1$ to a sphere with a single suture. Thus $M'$ is the complementary sutured manifold to a surface $R'$, which is the Murasugi sum of $R_1$ with a disc. This gives that $M'$ is homeomorphic to $M_1$.
\end{proof}
\subsection{Trees}
%Master document is alexpolypaper.tex.

\begin{definition}
A subgraph $\mathcal{T}$ of a digraph $(\mathcal{G},\mathcal{O})$ is a \textit{directed tree} if it is connected, it contains no simple closed curves, and any $v\in\V(\mathcal{G})$ is the terminal vertex of at most one edge of $\mathcal{T}$. 
There is then one vertex $u\in\V(\mathcal{T})$ that is not the terminal vertex of any edge of $\mathcal{T}$. 
This vertex is called the \textit{origin} of $\mathcal{T}$. 
For $v\in\V(\mathcal{T})$, the unique simple path from $u$ to $v$ in $\mathcal{T}$ is a directed path. 
Any $v\in\V(\mathcal{T})$ that is not the initial vertex of any edge of $\mathcal{T}$ is called a \textit{leaf}.

A directed tree $\mathcal{T}$ is a \textit{directed spanning subtree} of $\mathcal{G}$ if $\V(\mathcal{T})=\V(\mathcal{G})$. 
Define $\Tr(\mathcal{G},v)=\{\mathcal{T}:\textrm{$\mathcal{T}$ is a directed spanning subtree of $\mathcal{G}$ with origin $v$}\}$.

Figure \ref{treedefnspic1} shows an example of a digraph $\mathcal{G}$ with a directed spanning subtree $\mathcal{T}$. 
The origin of $\mathcal{T}$ is $u$, and $v,v'$ are leaves.
\begin{figure}[htbp]
\centering
%LaTeX with PSTricks extensions
%%Creator: 0.47
%%Please note this file requires PSTricks extensions
\psset{xunit=.35pt,yunit=.35pt,runit=.35pt}
\begin{pspicture}(219.99990845,269.99990845)
{
\pscustom[linewidth=2,linecolor=black,fillstyle=solid,fillcolor=black]%vertex
{
\newpath
\moveto(30,159.99990583)
\curveto(30,154.47705833)(25.5228475,149.99990583)(20,149.99990583)
\curveto(14.4771525,149.99990583)(10,154.47705833)(10,159.99990583)
\curveto(10,165.52275333)(14.4771525,169.99990583)(20,169.99990583)
\curveto(25.5228475,169.99990583)(30,165.52275333)(30,159.99990583)
\closepath
}
}
{
\pscustom[linewidth=2,linecolor=black,fillstyle=solid,fillcolor=black]%vertex
{
\newpath
\moveto(210,39.99990583)
\curveto(210,34.47705833)(205.5228475,29.99990583)(200,29.99990583)
\curveto(194.4771525,29.99990583)(190,34.47705833)(190,39.99990583)
\curveto(190,45.52275333)(194.4771525,49.99990583)(200,49.99990583)
\curveto(205.5228475,49.99990583)(210,45.52275333)(210,39.99990583)
\closepath
}
}
{
\pscustom[linewidth=2,linecolor=black,fillstyle=solid,fillcolor=black]%vertex
{
\newpath
\moveto(190,179.99990583)
\curveto(190,174.47705833)(185.5228475,169.99990583)(180,169.99990583)
\curveto(174.4771525,169.99990583)(170,174.47705833)(170,179.99990583)
\curveto(170,185.52275333)(174.4771525,189.99990583)(180,189.99990583)
\curveto(185.5228475,189.99990583)(190,185.52275333)(190,179.99990583)
\closepath
}
}
{
\pscustom[linewidth=2,linecolor=black,fillstyle=solid,fillcolor=black]%vertex
{
\newpath
\moveto(90,239.99990583)
\curveto(90,234.47705833)(85.5228475,229.99990583)(80,229.99990583)
\curveto(74.4771525,229.99990583)(70,234.47705833)(70,239.99990583)
\curveto(70,245.52275333)(74.4771525,249.99990583)(80,249.99990583)
\curveto(85.5228475,249.99990583)(90,245.52275333)(90,239.99990583)
\closepath
}
}
{
\pscustom[linewidth=2,linecolor=black,fillstyle=solid,fillcolor=black]%vertex
{
\newpath
\moveto(90,79.99990583)
\curveto(90,74.47705833)(85.5228475,69.99990583)(80,69.99990583)
\curveto(74.4771525,69.99990583)(70,74.47705833)(70,79.99990583)
\curveto(70,85.52275333)(74.4771525,89.99990583)(80,89.99990583)
\curveto(85.5228475,89.99990583)(90,85.52275333)(90,79.99990583)
\closepath
}
}
{
\pscustom[linewidth=1,linecolor=black]%edge
{
\newpath
\moveto(180,179.99990845)
\curveto(180,179.99990845)(202.02931,198.66746845)(205,209.99994845)
\curveto(209.14963,225.82981845)(195.84324,245.73425845)(180,244.99994845)
\curveto(164.17807,244.26661845)(151.92968,224.83687845)(155,209.99994845)
\curveto(157.37406,198.52763845)(180,179.99990845)(180,179.99990845)
\closepath
}
}
{
\pscustom[linewidth=1,linecolor=black]%edge
{
\newpath
\moveto(80,79.99990845)
\lineto(80,239.99990845)
\lineto(180,179.99990845)
\lineto(80,79.99990845)
\closepath
}
}
{
\pscustom[linewidth=1,linecolor=black]%edge
{
\newpath
\moveto(200,39.99990845)
\lineto(80,79.99990845)
}
}
{
\pscustom[linewidth=1,linecolor=black]%edge
{
\newpath
\moveto(20,159.99990845)
\lineto(80,239.99990845)
}
}
{
\pscustom[linewidth=1,linecolor=black]%edge
{
\newpath
\moveto(83.75644,80.00500845)
\curveto(83.75644,80.00500845)(138.21261,98.77423845)(160,89.99990845)
\curveto(181.78739,81.22557845)(200,39.99990845)(200,39.99990845)
\curveto(200,39.99990845)(145.41047,9.51828845)(120,19.99988845)
\curveto(98.39845,28.91038845)(83.75644,80.00500845)(83.75644,80.00500845)
\closepath
}
}
{
\pscustom[linewidth=3,linecolor=black]%tree edge
{
\newpath
\moveto(200,39.99990845)
\lineto(80,79.99990845)
\lineto(80,239.99990845)
\lineto(20,159.99990845)
}
}
{
\pscustom[linewidth=3,linecolor=black]%tree edge
{
\newpath
\moveto(80,79.99990845)
\lineto(180,179.99990845)
}
}
{
\pscustom[linewidth=2,linecolor=black,fillstyle=solid,fillcolor=black]%arrowhead
{
\newpath
\moveto(128.85786438,128.85777282)
\lineto(128.85786438,117.54406432)
\lineto(143,142.99990845)
\lineto(117.54415588,128.85777282)
\lineto(128.85786438,128.85777282)
\closepath
}
}
{
\pscustom[linewidth=2,linecolor=black,fillstyle=solid,fillcolor=black]%arrowhead
{
\newpath
\moveto(50.82727327,201.12800836)
\lineto(49.10694261,212.31015764)
\lineto(39,184.99990845)
\lineto(62.00942255,202.84833902)
\lineto(50.82727327,201.12800836)
\closepath
}
}
{
\pscustom[linewidth=2,linecolor=black,fillstyle=solid,fillcolor=black]%arrowhead
{
\newpath
\moveto(80,144.99990845)
\lineto(88,136.99990845)
\lineto(80,164.99990845)
\lineto(72,136.99990845)
\lineto(80,144.99990845)
\closepath
}
}
{
\pscustom[linewidth=2,linecolor=black,fillstyle=solid,fillcolor=black]%arrowhead
{
\newpath
\moveto(136.86652239,60.82314076)
\lineto(126.88383842,55.49904265)
\lineto(156,54.99990845)
\lineto(131.54242427,70.80582473)
\lineto(136.86652239,60.82314076)
\closepath
}
}
{
\pscustom[linewidth=2,linecolor=black,fillstyle=solid,fillcolor=black]%arrowhead
{
\newpath
\moveto(132.14985851,208.70999334)
\lineto(143.12576796,211.4539707)
\lineto(115,218.99990845)
\lineto(134.89383588,197.73408389)
\lineto(132.14985851,208.70999334)
\closepath
}
}
{
\pscustom[linewidth=2,linecolor=black,fillstyle=solid,fillcolor=black]%arrowhead
{
\newpath
\moveto(178.11145618,245.94418036)
\lineto(167.37832989,242.36647159)
\lineto(196,236.99990845)
\lineto(174.53374742,256.67730665)
\lineto(178.11145618,245.94418036)
\closepath
}
}
{
\pscustom[linewidth=2,linecolor=black,fillstyle=solid,fillcolor=black]%arrowhead
{
\newpath
\moveto(164.85786438,88.14204407)
\lineto(153.54415588,88.14204407)
\lineto(179,73.99990845)
\lineto(164.85786438,99.45575257)
\lineto(164.85786438,88.14204407)
\closepath
}
}
{
\pscustom[linewidth=2,linecolor=black,fillstyle=solid,fillcolor=black]%arrowhead
{
\newpath
\moveto(121.05152457,19.82980563)
\lineto(132.34017552,20.58237434)
\lineto(106,32.99990845)
\lineto(121.80409327,8.54115468)
\lineto(121.05152457,19.82980563)
\closepath
}
}
{
\put(0,230){$\mathcal{G}$}
\put(150,120){$\mathcal{T}$}
\put(45,60){$u$}
\put(35,140){$v$}
\put(200,160){$v'$}
}
\end{pspicture}
\caption{\label{treedefnspic1}}
\end{figure}
\end{definition}

\begin{lemma}\label{deletecontractlemma}
Let $v\in\V(\mathcal{G})$ and $e\in\E(\mathcal{G})$. 
Then ${|\Tr(\mathcal{G},v)|}={|\Tr(\mathcal{G}\setminus e,v)|}+{|\Tr(\mathcal{G}/ e,v)|}$.
\end{lemma}

\begin{lemma}
Suppose $\mathcal{G}$ is \ocon{\mathcal{O}}. 
Then any directed tree 
$\mathcal{T}$ in $\mathcal{G}$ can be extended to a directed spanning subtree with the same origin.
\end{lemma}

\begin{definition}
For a set $\mathcal{F}\subseteq\V(\mathcal{G})$, by a \textit{\spantree{\mathcal{F}}} of $\mathcal{G}$ we will mean a directed tree $\mathcal{T}$ such that $\mathcal{F}\subseteq\V(\mathcal{T})$ and every leaf of $\mathcal{T}$ is in $\mathcal{F}$.
\end{definition}

\begin{lemma}\label{extendtreeslemma}
Suppose $\mathcal{G}$ is \ocon{\mathcal{O}}, and that, for some $v\in\V(\mathcal{G})$, there are $n$ distinct \spantree{\mathcal{F}}s of $\mathcal{G}$ with origin $v$. 
Then $|\Tr(\mathcal{G},v)|\geq n$. In particular, if $v\in\mathcal{F}$ then $|\Tr(\mathcal{G},v)|\geq |\Tr(\mathcal{G}[\mathcal{F}],v)|$.
\end{lemma}
\begin{proof}
Since $\mathcal{G}$ is \ocon{\mathcal{O}}, any \spantree{\mathcal{F}} can be extended to a directed spanning subtree with the same origin.

Let $\mathcal{T}_1$ and $\mathcal{T}_2$ be extensions of distinct \spantree{\mathcal{F}}s $\mathcal{T}^{\mathcal{F}}_1$ and $\mathcal{T}^{\mathcal{F}}_2$ respectively. 
Then there exists $w\in\mathcal{F}$ such that the directed path from $v$ to $w$ in $\mathcal{T}^{\mathcal{F}}_1$ is different to that in $\mathcal{T}^{\mathcal{F}}_2$. 
Then the directed path from $v$ to $w$ in $\mathcal{T}_1$ is different to that in $\mathcal{T}_2$. Thus $\mathcal{T}_1\neq\mathcal{T}_2$.
\end{proof}

\begin{lemma}
Suppose that $\mathcal{G}$ is planar and that for every region $r$ of $\sphere\setminus\mathcal{G}$ the boundary of $r$ is a cycle. 
Then $\mathcal{G}$ is \ocon{\mathcal{O}}.
In particular, if incoming and outgoing edges alternate at every vertex of $\mathcal{G}$ then $\mathcal{G}$ is $\mathcal{O}$--connected.
\end{lemma}

\begin{lemma}
Let $e\in\E(\mathcal{G})$ be a loop and let $v\in\V(\mathcal{G})$. Then ${|\Tr(\mathcal{G}\setminus e,v)|}={|\Tr(\mathcal{G},v)|}$.
\end{lemma}

\begin{lemma}
Let $e\in\E(\mathcal{G})$ be such that $\tau(e)$ has in-degree 1. 
Then $|\Tr(\mathcal{G}/ e,v)|=|\Tr(\mathcal{G},v)|$ for any $v\in\V(\mathcal{G})\setminus \{\tau(e)\}$.
\end{lemma}
\begin{proof}
Any directed spanning subtree in $\mathcal{G}$ contains $e$.
\end{proof}

\begin{definition}\label{movedefn}
Call removing a loop from a digraph $\mathcal{G}$ a \textit{move 1} on $\mathcal{G}$, and collapsing an edge whose terminal vertex has no other incoming edge a \textit{move 2} on $\mathcal{G}$.
\end{definition}

%------------------------------

\section{The Alexander polynomial}\label{alexpolysec}
\subsection{Alexander's definition}
%Master document is alexpolypaper.tex.

Alexander (\cite{MR1501429}) defines a link invariant as follows.

Take a link $L$ with a reduced diagram $D$. Let $\mathcal{G}$ be the underlying graph of $D$ with induced orientation $O$. Each region $r_i$ of $\sphere{}\setminus\mathcal{G}$ has a corner $\hat{c}_{ij}$ at each crossing $c_j$ on its boundary. At each crossing, two of the four corners are dotted and each is assigned a value in $\{\pm t,\pm 1\}$ as shown in Figure \ref{alexanderpic1}.
\begin{figure}[htbp]
\centering
%LaTeX with PSTricks extensions
%%Creator: 0.46
%%Please note this file requires PSTricks extensions
\psset{xunit=.5pt,yunit=.5pt,runit=.5pt}
\begin{pspicture}(140,140)
{
\pscustom[linewidth=1,linecolor=black]%\ line
{
\newpath
\moveto(110,29.99998262)
\curveto(30,110.00000262)(30,110.00000262)(30,110.00000262)
}
}
{
\pscustom[linewidth=1,linecolor=black,fillstyle=solid,fillcolor=black]%arrowhead
{
\newpath
\moveto(37.07106693,102.92893392)
\lineto(42.72792118,102.92893321)
\lineto(30,110.00000262)
\lineto(37.07106622,97.27207967)
\lineto(37.07106693,102.92893392)
\closepath
}
}
{
\pscustom[linewidth=1,linecolor=black]%bottom / line
{
\newpath
\moveto(30,29.99998262)
\lineto(65,65.00000262)
}
}
{
\pscustom[linewidth=1,linecolor=black]%top / line
{
\newpath
\moveto(75,75.00000262)
\lineto(110,110.00000262)
}
}
{
\pscustom[linewidth=1,linecolor=black,fillstyle=solid,fillcolor=black]%circle
{
\newpath
\moveto(75,55)
\curveto(75,53.62)(73.88,52.5)(72.5,52.5)
\curveto(71.12,52.5)(70,53.62)(70,55)
\curveto(70,56.38)(71.12,57.5)(72.5,57.5)
\curveto(73.88,57.5)(75,56.38)(75,55)
\closepath
}
}
{
\pscustom[linewidth=1,linecolor=black,fillstyle=solid,fillcolor=black]%circle
{
\newpath
\moveto(57.5,72.5)
\curveto(57.5,71.12)(56.38,70)(55,70)
\curveto(53.62,70)(52.5,71.12)(52.5,72.5)
\curveto(52.5,73.88)(53.62,75)(55,75)
\curveto(56.38,75)(57.5,73.88)(57.5,72.5)
\closepath
}
}
{
\put(20,65){$-t$}
\put(68,30){$t$}
\put(90,65){$-1$}
\put(68,95){$1$}
}

\end{pspicture}
\caption{\label{alexanderpic1}}
\end{figure}

Let $a_{ij}$ be the value so assigned to the corner $\hat{c}_{ij}$. The matrix $\textbf{A}=(a_{ij})$ is called the \textit{Alexander matrix}. Since $\mathcal{G}$ is 4--valent, $\textbf{A}$ is an $(n+2)\times n$ matrix, where $n$ is the number of crossings in $D$. Choose adjacent regions $r_p,r_q$, and denote by $\textbf{A}(p,q)$ the matrix given by deleting rows $p$ and $q$ from $\textbf{A}$.
Let $\rap{L}{t}=\det\textbf{A}(p,q)$.

Alexander shows that, up to a factor of $\pm t^m$ for some $m\in\mathbb{Z}$, this definition of $\rap{L}{t}$ is independent of the choice of the pair of adjacent regions $r_p,r_q$ and of the choice of the diagram $D$. We define $\nrap{L}{t}$ to be $\rap{L}{t}$ normalised such that $\nrap{L}{0}$ is defined and strictly positive, except when $\rap{L}{t}= 0$.

\begin{definition}
The \textit{reduced Alexander polynomial} of $L$ is $\nrap{L}{t}$.
\end{definition}

\begin{proposition}[see \cite{MR1501429} p301]
If $L=L_1 \# L_2$ for some links $L_1,L_2$ then $\nrap{L}{t}{}=\nrap{L_1}{t}{}\cdot\nrap{L_2}{t}{}$.
\end{proposition}

\begin{lemma}[see \cite{MR1277811} 8C4, 7A]
Let $D'$ be the reflection of $D$, and let $L'$ be the link with diagram $D'$. Then $\nrap{L}{t}=\nrap{L'}{t}$.  
\end{lemma}

\begin{remark}
Note that reflection changes positive crossings to negative ones, and vice versa.
\end{remark}

\begin{lemma}[\cite{MR1277811} 8C7]
Suppose $\nrap{L}{t}=\sum_{i=0}^n a_i t^i$. Then $|a_i|=|a_{n-i}|$ for $0 \leq i\leq n$.
\end{lemma}

\begin{proposition}[see \cite{MR0137107} (3.7) and the proof of Lemma 3.6]\label{alexpolyproductprop}
Suppose $L=L_1*\cdots *L_m$ for some links $L_i$. Then $\nrap{L}{0}=\prod_{i=1}^{m} \nrap{L_i}{0}$.
\end{proposition}

\subsection{Murasugi's proof}
%Master document is alexpolypaper.tex.

\newcommand{\g}[1]{G_b(#1)} %black-regions centred graph

Murasugi (\cite{MR0099664}) considers Alexander's definition from the following viewpoint. 

Once $r_p,r_q$ have been fixed, $\det\textbf{A}(p,q)$ is formed of terms given by choosing (row, column) pairs $(i_1,j_1),\ldots ,(i_n,j_n)$ in such a way that each row and each column is chosen exactly once and then multiplying together the $a_{i_k j_k}$. This is equivalent to choosing a bijection between the crossings $c_j$ and regions $r_i$ other than $r_p,r_q$, or choosing one corner $\hat{c}_{ij}$ at each crossing $c_j$ provided $\hat{c}_{pj},\hat{c}_{qj}$ are never chosen. 
Call such a bijection an L$^s$\textit{--correspondence} if the resulting product is $\pm t^s$, or equivalently if $s$ of the chosen corners are dotted.

Murasugi shows (\cite{MR0099664} I Lemma 4.2; II Lemmas 6.8, 8.1) that any two L$^s$--correspondences give terms in the determinant with the same sign. Thus, to find $\nrap{L}{0}$ for a link $L$, we need only count ways of choosing a corner for each crossing as above so that as few dotted corners as possible are chosen. Alternatively, we can look for choices where as many dotted corners as possible are chosen.

Now consider a special alternating link diagram $D$ with the regions of $D$ coloured in a checkerboard pattern so that the black regions form the Seifert surface for $D$ given by Seifert's algorithm. That is, colour the inside of each Seifert circle black, and the remaining area white. Then, for a black region $r$ of $D$, either every corner of $r$ is dotted or every corner of $r$ is undotted. For a white region, corners are alternately dotted and undotted (\cite{MR0099664} II Lemma 6.3).
Let $x$ be the number of black regions with dotted corners. Then the black regions contribute a constant factor of $t^x$ to $|\prod_k a_{i_k,j_k}|$ and so may be safely ignored for our purposes. 
This is in fact the power of $t$ we cancel when normalising $\rap{L}{t}$ to $\nrap{L}{t}$ (see \cite{MR0099664} I Lemmas 3.1, 4.1, 5.4). Call an L$^{s-x}$--correspondence an L$^s_0$\textit{--correspondence} (this is actually Murasugi's definition of an L$^s$--correspondence).

Indeed, Murasugi shows that we can forget the black regions altogether by defining a digraph $\graph{M}{D}$ from the diagram $D$ as follows.

$\graph{M}{D}$ has a vertex at the centre of each white region of $D$, and one edge for each crossing, joining the centres of the white regions meeting at the crossing. $\graph{M}{D}$ is therefore planar. 
At each crossing, one white corner is dotted, and the other is undotted. Orient each edge from the undotted side to the dotted side (note that this is the reverse of in \cite{MR0099664}).

As dotted and undotted corners of white regions alternate, the boundary of any region $r$ of $\sphere\setminus\graph{M}{D}$ is a cycle with respect to the above orientation $o$. 
Thus $\graph{M}{D}$ is \ocon{o}{}, and in particular, for $v\in\V(\graph{M}{D}{})$, $|\Tr(\graph{M}{D}{},v)|\geq 1$.

Define a second (unoriented) graph $\g{D}$ with a vertex at the centre of each black region of $D$ and an edge through each crossing. Then $\g{D}$ is the dual graph $\graph{M}{D}^*$ of $\graph{M}{D}$.

For a directed spanning subtree $\mathcal{T}$ of $\graph{M}{D}$ let $\mathcal{T}^*$ be the subgraph of $\g{D}$ consisting of all edges in $\E(\g{D}{})$ that do not cross any edge in $\mathcal{T}$. Then $\mathcal{T}^*$ is a tree (\cite{MR0099664} I Lemma 5.2). By counting edges and vertices we see that $\mathcal{T}^*$ spans $\g{D}$.

Let $v_p\in\V(\graph{M}{D}{})$ and $v_q\in\V(\g{D}{})$ be the vertices corresponding to $r_p$ and $r_q$ respectively. Let $\mathcal{T}\in\Tr(\graph{M}{D}{},v_p)$. There is then a unique way of orienting the edges of $\mathcal{T}^*$ so it becomes a directed tree with origin $v_q$.
Let $y$ be the number of white regions. Each region $r\in D$ other than $r_p,r_q$ is the terminal vertex of exactly one edge $e_r$ of $\mathcal{T}$ or $\mathcal{T}^*$. By pairing $r$ with the crossing that $e_r$ corresponds to, we can construct an L$_0^y$--correspondence. Clearly there is no L$_0^s$--correspondence for any $s>y$.

Conversely, given an L$_0^y$--correspondence this process can be reversed to give a directed subgraph $\mathcal{T}$ of $\graph{M}{D}$ with exactly one edge ending at each vertex other than $v_p$. Suppose $\mathcal{T}$ contains an embedded closed curve, dividing $\sphere$ into two discs.
Considering the Euler characteristic $\chi$ of the disc that does not contain $r_p,r_q$ gives a contradiction. Hence $\mathcal{T}\in\Tr(\graph{M}{D}{},v_p)$.

For general $s$, an L$_0^s$--correspondence can be used to construct a directed spanning subtree $\mathcal{T}$ of the underlying graph of $\graph{M}{D}$ with an orientation that will not in general agree with that of $\graph{M}{D}$. The number of edges of $\mathcal{T}$ where these two orientations (dis)agree is determined by $s$. Murasugi and Stoimenow (\cite{MR2001624}) use this to assign a polynomial to any connected digraph in which the in-degree equals the out-degree at each vertex. 

Given a planar digraph $(\mathcal{G},\mathcal{O})$ in which incoming and outgoing edges alternate at each vertex, we can construct a product sutured manifold $\nmfld{M}{\mathcal{G}}$ embedded in $\Sphere$ in such a way that there is an `obvious' projection of the sutures onto $\sphere$ that gives a link diagram. Roughly speaking, this gives an inverse to $\graphm{M}$. We shall examine this in more detail later.

\begin{construction}
We first build the 3--manifold $N=\nmfld{M}{\mathcal{G}}$. 

Centre a 0--handle $\disc ^3$ on each vertex of $\mathcal{G}^*\subset\sphere\subset\Sphere$. These should be taken to be sufficiently small that they do not intersect. 

Attach a 1--handle $\intvl\times\disc ^2$ for each edge $e\in\E(\mathcal{G}^*)$ with $\intvl\times\{0\}$ running along $e$ and $\partial\intvl\times\disc^2$ glued to the 0--handles corresponding to the endpoints of $e$.

We now define the sutures $s$. For a 0--handle $V_0$, let $W_1$ be the union of the 1--handles that meet $V_0$. Then let $V_0\cap s=(\partial V_0\setminus \Int_N(W_1))\cap \sphere$.

Now let $V_1$ be a 1--handle. Then $V_1\cap s$ is made up of two disjoint simple arcs, one running from $\{0\}\times\{1\}\in (\intvl\times\disc ^2)\cap\sphere$ to $\{1\}\times\{-1\}$ and the other running from $\{0\}\times\{-1\}$ to $\{1\}\times\{1\}$. 
The arcs twist around $V_1$ in the direction shown in Figure \ref{mgraphpic1}, where the dashed line denotes an arc passing underneath the manifold.
\begin{figure}[htbp]
\centering
%LaTeX with PSTricks extensions
%%Creator: 0.47
%%Please note this file requires PSTricks extensions
\psset{xunit=.5pt,yunit=.5pt,runit=.5pt}
\begin{pspicture}(90,135)
{
\newgray{lightgrey}{.8}
\newgray{lightishgrey}{.7}
\newgray{grey}{.6}
\newgray{midgrey}{.4}
\newgray{darkgrey}{.3}
}
{
\pscustom[linewidth=1,linecolor=black]%graph arrow
{
\newpath
\moveto(5,70)
\lineto(25,70)
}
}
{
\pscustom[linewidth=3,linecolor=black,fillstyle=solid,fillcolor=black]%arrowhead
{
\newpath
\moveto(15,70)
\lineto(19,74)
\lineto(5,70)
\lineto(19,66)
\lineto(15,70)
\closepath
}
}
{
\pscustom[linewidth=1,linecolor=black]%graph arrow
{
\newpath
\moveto(63,70)
\lineto(85,70)
}
}
{
\pscustom[linewidth=1.5,linecolor=darkgrey,fillstyle=solid,fillcolor=lightishgrey]%1-handle
{
\newpath
\moveto(30.19104004,109.80865217)
\lineto(57.4989624,109.80865217)
\lineto(57.4989624,30.19132734)
\lineto(30.19104004,30.19132734)
\lineto(30.19104004,109.80865217)
\closepath
}
}
{
\pscustom[linewidth=1.5,linecolor=darkgrey,fillstyle=solid,fillcolor=lightishgrey]%1-handle
{
\newpath
\moveto(19.95824228,8.77860678)
\curveto(19.56113001,22.33109273)(30.22569289,33.63947569)(43.77821017,34.03658704)
\curveto(57.33072745,34.4336984)(68.63913654,23.76916017)(69.03624882,10.21667422)
\curveto(69.04863423,9.79398987)(69.05009642,9.37105608)(69.04063392,8.9482962)
}
}
{
\pscustom[linewidth=1.5,linecolor=darkgrey,fillstyle=solid,fillcolor=lightishgrey]%1-handle
{
\newpath
\moveto(19.95795058,129.04730322)
\curveto(19.56083831,115.49481727)(30.22540119,104.18643431)(43.77791847,103.78932296)
\curveto(57.33043575,103.3922116)(68.63884484,114.05674983)(69.03595712,127.60923578)
\curveto(69.04834253,128.03192013)(69.04980472,128.45485392)(69.04034222,128.8776138)
}
}
{
\pscustom[linewidth=3,linecolor=black]
{
\newpath
\moveto(19.95795058,129.04730322)
\curveto(19.71249752,120.67058108)(23.75705162,112.74734938)(30.68530176,108.0325543)
}
}
{
\pscustom[linewidth=3,linecolor=black]
{
\newpath
\moveto(57.21802844,107.3313956)
\curveto(64.72878775,111.88189718)(69.23685106,120.09770961)(69.04034222,128.8772138)
}
}
{
\pscustom[linewidth=3,linecolor=black]
{
\newpath
\moveto(19.95777648,9.33120678)
\curveto(19.71232342,17.70792892)(23.75687752,25.63116062)(30.68512766,30.3459557)
}
}
{
\pscustom[linewidth=3,linecolor=black]
{
\newpath
\moveto(57.21785434,31.0471144)
\curveto(64.72861365,26.49661282)(69.23667696,18.28080039)(69.04016812,9.5012962)
}
}
{
\pscustom[linewidth=3,linecolor=black]
{
\newpath
\moveto(30,109)
\lineto(58,30)
}
}
{
\pscustom[linewidth=3,linecolor=black,linestyle=dashed,dash=12 12]
{
\newpath
\moveto(30,30)
\lineto(58,108)
}
}
{
\pscustom[linewidth=1,linecolor=black,fillstyle=solid,fillcolor=black]%arrowhead
{
\newpath
\moveto(39.51123442,82.63670822)
\lineto(44.66104489,80.29588528)
\lineto(36,92)
\lineto(37.17041147,77.48689775)
\lineto(39.51123442,82.63670822)
\closepath
}
}
{
\pscustom[linewidth=1,linecolor=black,fillstyle=solid,fillcolor=black]%arrowhead
{
\newpath
\moveto(39.51121902,55.86327755)
\lineto(37.17038761,61.01308418)
\lineto(36,46.49998)
\lineto(44.66102565,58.20410896)
\lineto(39.51121902,55.86327755)
\closepath
}
}
{
\pscustom[linestyle=none,fillstyle=solid,fillcolor=white]
{
\newpath
\moveto(15,134.99999738)
\lineto(75,134.99999738)
\lineto(75,127.05891634)
\lineto(15,127.05891634)
\lineto(15,134.99999738)
\closepath
}
}
{
\pscustom[linestyle=none,fillstyle=solid,fillcolor=white]
{
\newpath
\moveto(10,9.82116438)
\lineto(75,9.82116438)
\lineto(75,-0.0000236)
\lineto(10,-0.0000236)
\lineto(10,9.82116438)
\closepath
}
}
\end{pspicture}
\caption{\label{mgraphpic1}}
\end{figure}

Using the orientation $\mathcal{O}$, we can define an orientation on the arcs of $s$ that run along 1--handles, as shown in Figure \ref{mgraphpic1}. Since incoming and outgoing edges alternate at every vertex of $\mathcal{G}$, this definition of the orientation of $s$ is locally consistent, as shown in Figure \ref{mgraphpic2}.
\begin{figure}[htbp]
\centering
%LaTeX with PSTricks extensions
%%Creator: inkscape 0.47
%%Please note this file requires PSTricks extensions
\psset{xunit=.5pt,yunit=.5pt,runit=.5pt}
{
\newgray{lightgrey}{.8}
\newgray{lightishgrey}{.7}
\newgray{grey}{.6}
\newgray{midgrey}{.4}
\newgray{darkgrey}{.3}
}
\begin{pspicture}(260,160)
{
\pscustom[linewidth=1,linecolor=black]%graph arrow
{
\newpath
\moveto(244,80)
\lineto(134,29.99998)
}
}
{
\pscustom[linewidth=1,linecolor=black,fillstyle=solid,fillcolor=black]%arrowhead
{
\newpath
\moveto(234.89633585,75.86196919)
\lineto(232.91008251,70.5652912)
\lineto(244,80)
\lineto(229.59965786,77.84822252)
\lineto(234.89633585,75.86196919)
\closepath
}
}
{
\pscustom[linewidth=1,linecolor=black]%graph arrow
{
\newpath
\moveto(124,29.99998)
\lineto(14,80)
}
}
{
\pscustom[linewidth=1,linecolor=black,fillstyle=solid,fillcolor=black]%arrowhead
{
\newpath
\moveto(114.89633585,34.13801081)
\lineto(109.59965786,32.15175748)
\lineto(124,29.99998)
\lineto(112.91008251,39.4346888)
\lineto(114.89633585,34.13801081)
\closepath
}
}
{
\pscustom[linewidth=1,linecolor=black,fillstyle=solid,fillcolor=black]%graph vertex
{
\newpath
\moveto(134,29.99999438)
\curveto(134,27.23857063)(131.76142375,24.99999438)(129,24.99999438)
\curveto(126.23857625,24.99999438)(124,27.23857063)(124,29.99999438)
\curveto(124,32.76141813)(126.23857625,34.99999438)(129,34.99999438)
\curveto(131.76142375,34.99999438)(134,32.76141813)(134,29.99999438)
\closepath
}
}
{
\pscustom[linewidth=1.5,linecolor=darkgrey,fillstyle=solid,fillcolor=lightishgrey]%1-handle
{
\newpath
\moveto(144.35897262,98.9230235)
\lineto(178.99998862,118.9230235)
\lineto(238.99998862,14.9999755)
\lineto(204.35897262,-5.0000245)
\lineto(144.35897262,98.9230235)
\closepath
}
}
{
\pscustom[linewidth=1.5,linecolor=darkgrey,fillstyle=solid,fillcolor=lightishgrey]%1-handle
{
\newpath
\moveto(113.64102495,98.92302401)
\lineto(79.00000895,118.92302401)
\lineto(19.00000895,14.99997601)
\lineto(53.64102495,-5.00002399)
\lineto(113.64102495,98.92302401)
\closepath
}
}
{
\pscustom[linewidth=1.5,linecolor=darkgrey,fillstyle=solid,fillcolor=grey]%0-handle
{
\newpath
\moveto(186.52966325,152.03929777)
\curveto(195.94021044,120.2665421)(177.81205867,86.88087872)(146.03930299,77.47033153)
\curveto(114.26654731,68.05978434)(80.88088394,86.18793612)(71.47033675,117.96069179)
\curveto(68.42887639,128.22955022)(68.19143522,139.12532922)(70.78266206,149.51693535)
}
}
{
\pscustom[linewidth=3,linecolor=black]
{
\newpath
\moveto(187.2930012,149.21006201)
\curveto(190.41145008,136.41745933)(189.23516733,122.95424768)(183.94551892,110.89626167)
}
}
{
\pscustom[linewidth=3,linecolor=black]
{
\newpath
\moveto(70.3756688,148.93527601)
\curveto(67.20202826,135.91626431)(68.47793024,122.20955828)(73.99999708,109.9999949)
}
}
{
\pscustom[linewidth=3,linecolor=black]
{
\newpath
\moveto(154.03325147,80.45580432)
\curveto(137.97308034,73.1865598)(119.5287154,73.34794609)(103.59820377,80.89710577)
}
}
{
\pscustom[linewidth=3,linecolor=black]
{
\newpath
\moveto(74,110)
\lineto(64,9.99998)
}
}
{
\pscustom[linewidth=3,linecolor=black]
{
\newpath
\moveto(154,80)
\lineto(244,9.99998)
}
}
{
\pscustom[linewidth=3,linecolor=black,linestyle=dashed,dash=12 12]
{
\newpath
\moveto(14,9.99998)
\lineto(104,81)
}
}
{
\pscustom[linewidth=3,linecolor=black,linestyle=dashed,dash=12 12]
{
\newpath
\moveto(194,14)
\lineto(184,111)
}
}
{
\pscustom[linewidth=1,linecolor=black,fillstyle=solid,fillcolor=black]%arrowhead
{
\newpath
\moveto(70.90535746,76.95893206)
\lineto(67.28392762,81.30464787)
\lineto(70,67)
\lineto(75.25107327,80.58036191)
\lineto(70.90535746,76.95893206)
\closepath
}
}
{
\pscustom[linewidth=1,linecolor=black,fillstyle=solid,fillcolor=black]%arrowhead
{
\newpath
\moveto(174.19131191,64.24695048)
\lineto(168.56905648,63.62225543)
\lineto(182,58)
\lineto(173.56661686,69.8692059)
\lineto(174.19131191,64.24695048)
\closepath
}
}
{
\pscustom[linewidth=1,linecolor=black,fillstyle=solid,fillcolor=black]%arrowhead
{
\newpath
\moveto(44.86266529,34.18759806)
\lineto(43.93268418,28.6077114)
\lineto(53,39.99998)
\lineto(39.28277863,35.11757917)
\lineto(44.86266529,34.18759806)
\closepath
}
}
{
\pscustom[linewidth=1,linecolor=black,fillstyle=solid,fillcolor=black]%arrowhead
{
\newpath
\moveto(192.10431526,36.06114265)
\lineto(196.5215763,32.52733382)
\lineto(191,45.99998)
\lineto(188.57050643,31.64388161)
\lineto(192.10431526,36.06114265)
\closepath
}
}
{
\pscustom[linestyle=none,fillstyle=solid,fillcolor=white]
{
\newpath
\moveto(9,14.99999738)
\lineto(249,14.99999738)
\lineto(249,-7)
\lineto(9,-7)
\lineto(9,14.99999738)
\closepath
}
}
{
\pscustom[linestyle=none,fillstyle=solid,fillcolor=white]
{
\newpath
\moveto(64,154.99999738)
\lineto(194,154.99999738)
\lineto(194,144.99999738)
\lineto(64,144.99999738)
\lineto(64,154.99999738)
\closepath
}
}
\end{pspicture}
\caption{\label{mgraphpic2}}
\end{figure}
Therefore the sutures around a 0--handle are all oriented either clockwise or anticlockwise, with the orientations on adjacent vertices going in opposite directions.
From this we see that $(N,s)$ is a sutured manifold.
\end{construction}

\begin{definition}
Define $\mmfld{M}{\mathcal{G}}{}=\Sphere\setminus\Int_{\Sphere}(\nmfld{M}{\mathcal{G}})$ to be the complementary sutured manifold to $\nmfld{M}{\mathcal{G}}$.
\end{definition}

\begin{lemma}\label{mgraphmoveonelemma}
Let $\mathcal{G}$ be a planar digraph in which incoming and outgoing edges alternate at each vertex. Let $e\in\E(\mathcal{G})$ be a loop that bounds a disc in $\sphere\setminus\mathcal{G}$.
Then $\nmfld{M}{\mathcal{G}}$ and $\nmfld{M}{\mathcal{G}\setminus e}$ (and hence also $\mmfld{M}{\mathcal{G}}$ and $\mmfld{M}{\mathcal{G}\setminus e}$) are equivalent as sutured manifolds embedded in $\Sphere$.
\end{lemma}
\begin{proof}
This can be checked locally, as shown in Figure \ref{mgraphpic3}.
\end{proof}
\begin{figure}[htbp]
\centering
$\begin{array}{ccc}
%LaTeX with PSTricks extensions
%%Creator: 0.46
%%Please note this file requires PSTricks extensions
\psset{xunit=.5pt,yunit=.5pt,runit=.5pt}
\begin{pspicture}(85,115)
{
\pscustom[linewidth=2,linecolor=black,fillstyle=solid,fillcolor=black]%vertex
{
\newpath
\moveto(55.5,45.50004)
\curveto(55.5,39.98004)(51.02,35.50004)(45.5,35.50004)
\curveto(39.98,35.50004)(35.5,39.98004)(35.5,45.50004)
\curveto(35.5,51.02004)(39.98,55.50004)(45.5,55.50004)
\curveto(51.02,55.50004)(55.5,51.02004)(55.5,45.50004)
\closepath
}
}
{
\pscustom[linewidth=1,linecolor=black]
{
\newpath
\moveto(5.5,5.49998262)
\lineto(45.5,45.49998262)
\lineto(85.5,5.49998262)
}
}
{
\pscustom[linewidth=1,linecolor=black,fillstyle=solid,fillcolor=black]%arrowhead
{
\newpath
\moveto(25.15686219,25.15691481)
\lineto(25.15686219,19.50006056)
\lineto(32.22793,32.22798262)
\lineto(19.50000794,25.15691481)
\lineto(25.15686219,25.15691481)
\closepath
}
}
{
\pscustom[linewidth=1,linecolor=black,fillstyle=solid,fillcolor=black]%arrowhead
{
\newpath
\moveto(67.15685219,23.84315043)
\lineto(61.49999794,23.84315043)
\lineto(74.22792,16.77208262)
\lineto(67.15685219,29.50000468)
\lineto(67.15685219,23.84315043)
\closepath
}
}
{
\pscustom[linewidth=1,linecolor=black]%loop
{
\newpath
\moveto(45,44.22798262)
\curveto(45,44.22798262)(67.029312,62.89554262)(70,74.22802262)
\curveto(74.149629,90.05789262)(60.843244,109.96233262)(45,109.22802262)
\curveto(29.178069,108.49469262)(16.929679,89.06495262)(20,74.22802262)
\curveto(22.374055,62.75571262)(45,44.22798262)(45,44.22798262)
\closepath
}
}
{
\pscustom[linewidth=1,linecolor=black,fillstyle=solid,fillcolor=black]%arrowhead
{
\newpath
\moveto(42.07722123,109.46836996)
\lineto(37.61197079,105.99539739)
\lineto(52,108.22802262)
\lineto(38.60424866,113.93362041)
\lineto(42.07722123,109.46836996)
\closepath
}
}
\end{pspicture}&%LaTeX with PSTricks extensions
%%Creator: 0.46
%%Please note this file requires PSTricks extensions
\psset{xunit=.5pt,yunit=.5pt,runit=.5pt}
\begin{pspicture}(52,18)
{
\pscustom[linewidth=2,linecolor=black]
{
\newpath
\moveto(1,9)
\lineto(51,9)
}
}
{
\pscustom[linewidth=2,linecolor=black,fillstyle=solid,fillcolor=black]
{
\newpath
\moveto(31,9)
\lineto(23,1)
\lineto(51,9)
\lineto(23,17)
\lineto(31,9)
\closepath
}
}
\end{pspicture}&%LaTeX with PSTricks extensions
%%Creator: 0.46
%%Please note this file requires PSTricks extensions
\psset{xunit=.5pt,yunit=.5pt,runit=.5pt}
\begin{pspicture}(85,115)
{
\pscustom[linewidth=2,linecolor=black,fillstyle=solid,fillcolor=black]%vertex
{
\newpath
\moveto(55.5,45.50004)
\curveto(55.5,39.98004)(51.02,35.50004)(45.5,35.50004)
\curveto(39.98,35.50004)(35.5,39.98004)(35.5,45.50004)
\curveto(35.5,51.02004)(39.98,55.50004)(45.5,55.50004)
\curveto(51.02,55.50004)(55.5,51.02004)(55.5,45.50004)
\closepath
}
}
{
\pscustom[linewidth=1,linecolor=black]
{
\newpath
\moveto(5.5,5.49998262)
\lineto(45.5,45.49998262)
\lineto(85.5,5.49998262)
}
}
{
\pscustom[linewidth=1,linecolor=black,fillstyle=solid,fillcolor=black]%arrowhead
{
\newpath
\moveto(25.15686219,25.15691481)
\lineto(25.15686219,19.50006056)
\lineto(32.22793,32.22798262)
\lineto(19.50000794,25.15691481)
\lineto(25.15686219,25.15691481)
\closepath
}
}
{
\pscustom[linewidth=1,linecolor=black,fillstyle=solid,fillcolor=black]%arrowhead
{
\newpath
\moveto(67.15685219,23.84315043)
\lineto(61.49999794,23.84315043)
\lineto(74.22792,16.77208262)
\lineto(67.15685219,29.50000468)
\lineto(67.15685219,23.84315043)
\closepath
}
}
\end{pspicture}\\
%LaTeX with PSTricks extensions
%%Creator: 0.46
%%Please note this file requires PSTricks extensions
\psset{xunit=.25pt,yunit=.25pt,runit=.25pt}
\begin{pspicture}(440,330)
{
\newgray{lightgrey}{.8}
\newgray{lightishgrey}{.7}
\newgray{grey}{.6}
\newgray{midgrey}{.4}
\newgray{darkgrey}{.3}
}
{
\pscustom[linewidth=3,linecolor=darkgrey,fillstyle=solid,fillcolor=lightishgrey]%1-handle
{
\newpath
\moveto(170,290)
\lineto(240,290)
\lineto(240,160)
\lineto(170,160)
\lineto(170,290)
\closepath
}
}
{
\pscustom[linewidth=3,linecolor=darkgrey,fillstyle=solid,fillcolor=lightishgrey]%1-handle
{
\newpath
\moveto(406.22652673,87.04325264)
\lineto(355.9748552,132.96307166)
\lineto(289.99998019,10.00001323)
\lineto(340.25165172,-35.91980579)
\lineto(406.22652673,87.04325264)
\closepath
}
}
{
\pscustom[linewidth=3,linecolor=darkgrey,fillstyle=solid,fillcolor=lightishgrey]%1-handle
{
\newpath
\moveto(3.77346849,87.0432393)
\lineto(54.02514002,132.96305832)
\lineto(120.00001503,9.9999999)
\lineto(69.74834349,-35.91981913)
\lineto(3.77346849,87.0432393)
\closepath
}
}
{
\pscustom[linewidth=3,linecolor=darkgrey,fillstyle=solid,fillcolor=grey]%0-handle
{
\newpath
\moveto(10,39.99998262)
\curveto(60,39.99998262)(90,90.00000262)(90,140.00000262)
\curveto(90,190.00000262)(105,280.00000262)(205,280.00000262)
\curveto(305,280.00000262)(320,190.00000262)(320,140.00000262)
\curveto(320,90.00000262)(350,39.99998262)(400,39.99998262)
\lineto(400,340.00000262)
\lineto(10,340.00000262)
\lineto(10,39.99998262)
\closepath
}
}
{
\pscustom[linewidth=3,linecolor=darkgrey,fillstyle=solid,fillcolor=grey]%0-handle
{
\newpath
\moveto(249.909086,144.90910879)
\curveto(249.909086,120.06910934)(229.74908645,99.90910979)(204.909087,99.90910979)
\curveto(180.06908755,99.90910979)(159.909088,120.06910934)(159.909088,144.90910879)
\curveto(159.909088,169.74910824)(180.06908755,189.90910779)(204.909087,189.90910779)
\curveto(229.74908645,189.90910779)(249.909086,169.74910824)(249.909086,144.90910879)
\closepath
}
}
{
\pscustom[linewidth=6,linecolor=black]
{
\newpath
\moveto(10,39.99998262)
\curveto(16,39.99998262)(22,40.99998262)(28,41.99998262)
}
}
{
\pscustom[linewidth=6,linecolor=black]
{
\newpath
\moveto(400,39.99998262)
\curveto(394,39.99998262)(388,40.99998262)(382,41.99998262)
}
}
{
\pscustom[linewidth=6,linecolor=black]
{
\newpath
\moveto(332,87.00000262)
\lineto(345,9.99998262)
}
}
{
\pscustom[linewidth=6,linecolor=black,linestyle=dashed,dash=24 24]
{
\newpath
\moveto(78,87.00000262)
\lineto(65,9.99998262)
}
}
{
\pscustom[linewidth=6,linecolor=black,linestyle=dashed,dash=24 24]
{
\newpath
\moveto(382,41.99998262)
\lineto(315,9.99998262)
}
}
{
\pscustom[linewidth=6,linecolor=black]
{
\newpath
\moveto(28.000004,41.99998262)
\lineto(95,9.99998262)
}
}
{
\pscustom[linewidth=6,linecolor=black]
{
\newpath
\moveto(171,174.00000262)
\curveto(163,164.00000262)(159,150.00000262)(161,134.00000262)
\curveto(163.40832,114.73345262)(184.97502,100.00000262)(205,100.00000262)
\curveto(225,100.00000262)(244,114.00000262)(249,134.00000262)
\curveto(255,160.00000262)(239,174.00000262)(239,174.00000262)
}
}
{
\pscustom[linewidth=6,linecolor=black]
{
\newpath
\moveto(78,86.99997262)
\curveto(88,108.99997262)(91,130.00000262)(90,150.00000262)
\curveto(90,164.00000262)(93,182.00000262)(97,198.00000262)
\curveto(107,234.00000262)(134,265.00000262)(170,276.00000262)
}
}
{
\pscustom[linewidth=6,linecolor=black]
{
\newpath
\moveto(332,86.99997262)
\curveto(322,108.99997262)(319,130.00000262)(320,150.00000262)
\curveto(320,164.00000262)(317,182.00000262)(313,198.00000262)
\curveto(303,234.00000262)(276,265.00000262)(240,276.00000262)
}
}
{
\pscustom[linewidth=6,linecolor=black]
{
\newpath
\moveto(169,277.00000262)
\lineto(240,173.00000262)
}
}
{
\pscustom[linewidth=6,linecolor=black,linestyle=dashed,dash=24 24]
{
\newpath
\moveto(241,277.00000262)
\lineto(170,173.00000262)
}
}
{
\pscustom[linestyle=none,fillstyle=solid,fillcolor=white]
{
\newpath
\moveto(-9.97011328,349.59635458)
\lineto(409.98181735,350.22191482)
\lineto(409.98181735,329.99042564)
\lineto(-9.97011328,329.3648654)
\lineto(-9.97011328,349.59635458)
\closepath
}
}
{
\pscustom[linestyle=none,fillstyle=solid,fillcolor=white]
{
\newpath
\moveto(-9.59327888,340.0725708)
\lineto(10.31336021,340.0725708)
\lineto(10.31336021,-0.10028076)
\lineto(-9.59327888,-0.10028076)
\lineto(-9.59327888,340.0725708)
\closepath
}
}
{
\pscustom[linestyle=none,fillstyle=solid,fillcolor=white]
{
\newpath
\moveto(399.65759277,340.13867188)
\lineto(419.56405449,340.13867188)
\lineto(419.56405449,-10.12255859)
\lineto(399.65759277,-10.12255859)
\lineto(399.65759277,340.13867188)
\closepath
}
}
{
\pscustom[linestyle=none,fillstyle=solid,fillcolor=white]
{
\newpath
\moveto(-9.96357059,10.12585449)
\lineto(419.97746944,10.12585449)
\lineto(419.97746944,-39.77116776)
\lineto(-9.96357059,-39.77116776)
\lineto(-9.96357059,10.12585449)
\closepath
}
}
\end{pspicture}&&%LaTeX with PSTricks extensions
%%Creator: 0.46
%%Please note this file requires PSTricks extensions
\psset{xunit=.25pt,yunit=.25pt,runit=.25pt}
\begin{pspicture}(440,330)
{
\newgray{lightgrey}{.8}
\newgray{lightishgrey}{.7}
\newgray{grey}{.6}
\newgray{midgrey}{.4}
\newgray{darkgrey}{.3}
}
{
\pscustom[linewidth=3,linecolor=darkgrey,fillstyle=solid,fillcolor=lightishgrey]%1-handle
{
\newpath
\moveto(406.22652673,87.04325264)
\lineto(355.9748552,132.96307166)
\lineto(289.99998019,10.00001323)
\lineto(340.25165172,-35.91980579)
\lineto(406.22652673,87.04325264)
\closepath
}
}
{
\pscustom[linewidth=3,linecolor=darkgrey,fillstyle=solid,fillcolor=lightishgrey]%1-handle
{
\newpath
\moveto(3.77346849,87.0432393)
\lineto(54.02514002,132.96305832)
\lineto(120.00001503,9.9999999)
\lineto(69.74834349,-35.91981913)
\lineto(3.77346849,87.0432393)
\closepath
}
}
{
\pscustom[linewidth=3,linecolor=darkgrey,fillstyle=solid,fillcolor=grey]%0-handle
{
\newpath
\moveto(10,39.99998262)
\curveto(60,39.99998262)(90,90.00000262)(90,140.00000262)
\curveto(90,190.00000262)(105,280.00000262)(205,280.00000262)
\curveto(305,280.00000262)(320,190.00000262)(320,140.00000262)
\curveto(320,90.00000262)(350,39.99998262)(400,39.99998262)
\lineto(400,340.00000262)
\lineto(10,340.00000262)
\lineto(10,39.99998262)
\closepath
}
}
{
\pscustom[linewidth=6,linecolor=black]
{
\newpath
\moveto(10,39.99998262)
\curveto(16,39.99998262)(22,40.99998262)(28,41.99998262)
}
}
{
\pscustom[linewidth=6,linecolor=black]
{
\newpath
\moveto(400,39.99998262)
\curveto(394,39.99998262)(388,40.99998262)(382,41.99998262)
}
}
{
\pscustom[linewidth=6,linecolor=black]
{
\newpath
\moveto(332,87.00000262)
\lineto(345,9.99998262)
}
}
{
\pscustom[linewidth=6,linecolor=black,linestyle=dashed,dash=24 24]
{
\newpath
\moveto(78,87.00000262)
\lineto(65,9.99998262)
}
}
{
\pscustom[linewidth=6,linecolor=black,linestyle=dashed,dash=24 24]
{
\newpath
\moveto(382,41.99998262)
\lineto(315,9.99998262)
}
}
{
\pscustom[linewidth=6,linecolor=black]
{
\newpath
\moveto(28.000004,41.99998262)
\lineto(95,9.99998262)
}
}
{
\pscustom[linewidth=6,linecolor=black]
{
\newpath
\moveto(78,86.99997262)
\curveto(88,108.99997262)(91,130.00000262)(90,150.00000262)
\curveto(90,164.00000262)(93,182.00000262)(97,198.00000262)
\curveto(107,234.00000262)(134,265.00000262)(170,276.00000262)
\curveto(189,280.00000262)(192,280.00000262)(203,280.00000262)
\curveto(214,280.00000262)(231,279.00000262)(240,276.00000262)
}
}
{
\pscustom[linewidth=6,linecolor=black]
{
\newpath
\moveto(332,86.99997262)
\curveto(322,108.99997262)(319,130.00000262)(320,150.00000262)
\curveto(320,164.00000262)(317,182.00000262)(313,198.00000262)
\curveto(303,234.00000262)(276,265.00000262)(240,276.00000262)
}
}
{
\pscustom[linestyle=none,fillstyle=solid,fillcolor=white]
{
\newpath
\moveto(-9.97011328,349.59635458)
\lineto(409.98181735,350.22191482)
\lineto(409.98181735,329.99042564)
\lineto(-9.97011328,329.3648654)
\lineto(-9.97011328,349.59635458)
\closepath
}
}
{
\pscustom[linestyle=none,fillstyle=solid,fillcolor=white]
{
\newpath
\moveto(-9.59327888,340.0725708)
\lineto(10.31336021,340.0725708)
\lineto(10.31336021,-0.10028076)
\lineto(-9.59327888,-0.10028076)
\lineto(-9.59327888,340.0725708)
\closepath
}
}
{
\pscustom[linestyle=none,fillstyle=solid,fillcolor=white]
{
\newpath
\moveto(399.65759277,340.13867188)
\lineto(419.56405449,340.13867188)
\lineto(419.56405449,-10.12255859)
\lineto(399.65759277,-10.12255859)
\lineto(399.65759277,340.13867188)
\closepath
}
}
{
\pscustom[linestyle=none,fillstyle=solid,fillcolor=white]
{
\newpath
\moveto(-9.96357059,10.12585449)
\lineto(419.97746944,10.12585449)
\lineto(419.97746944,-39.77116776)
\lineto(-9.96357059,-39.77116776)
\lineto(-9.96357059,10.12585449)
\closepath
}
}
\end{pspicture}
\end{array}$
\caption{\label{mgraphpic3}}
\end{figure}

\begin{lemma}\label{mgraphmovetwolemma}
Let $\mathcal{G}$ be a planar digraph in which incoming and outgoing edges alternate at each vertex, and let $e\in\E(\mathcal{G})$ be such that $\tau(e)$ is 2--valent. 
Then $\mmfld{M}{\mathcal{G}/ e}$ is obtained from $\mmfld{M}{\mathcal{G}}$ by a product disc decomposition. In particular, $\mmfld{M}{\mathcal{G}}$ is an almost product sutured manifold if and only if $\mmfld{M}{\mathcal{G}/ e}$ is.
\end{lemma}
\begin{proof}
See Figure \ref{mgraphpic4}.
\end{proof}
\begin{figure}[htb]
\centering
$\begin{array}{ccc}
%LaTeX with PSTricks extensions
%%Creator: 0.46
%%Please note this file requires PSTricks extensions
\psset{xunit=.5pt,yunit=.5pt,runit=.5pt}
\begin{pspicture}(175,75)
{
\pscustom[linewidth=2,linecolor=black,fillstyle=solid,fillcolor=black]%vertex
{
\newpath
\moveto(46,41.000057)
\curveto(46,35.480057)(41.52,31.000057)(36,31.000057)
\curveto(30.48,31.000057)(26,35.480057)(26,41.000057)
\curveto(26,46.520057)(30.48,51.000057)(36,51.000057)
\curveto(41.52,51.000057)(46,46.520057)(46,41.000057)
\closepath
}
}
{
\pscustom[linewidth=2,linecolor=black,fillstyle=solid,fillcolor=black]%vertex
{
\newpath
\moveto(101,41.000017)
\curveto(101,35.480017)(96.52,31.000017)(91,31.000017)
\curveto(85.48,31.000017)(81,35.480017)(81,41.000017)
\curveto(81,46.520017)(85.48,51.000017)(91,51.000017)
\curveto(96.52,51.000017)(101,46.520017)(101,41.000017)
\closepath
}
}
{
\pscustom[linewidth=2,linecolor=black,fillstyle=solid,fillcolor=black]%vertex
{
\newpath
\moveto(156,41.000017)
\curveto(156,35.480017)(151.52,31.000017)(146,31.000017)
\curveto(140.48,31.000017)(136,35.480017)(136,41.000017)
\curveto(136,46.520017)(140.48,51.000017)(146,51.000017)
\curveto(151.52,51.000017)(156,46.520017)(156,41.000017)
\closepath
}
}
{
\pscustom[linewidth=1,linecolor=black]
{
\newpath
\moveto(5.5,40.49998262)
\lineto(175.5,40.49998262)
}
}
{
\pscustom[linewidth=1,linecolor=black]
{
\newpath
\moveto(15.5,75.50002262)
\lineto(35.5,40.49998262)
\lineto(15.5,5.49998262)
}
}
{
\pscustom[linewidth=1,linecolor=black]
{
\newpath
\moveto(165.5,75.50002262)
\lineto(145.5,40.49998262)
\lineto(165.5,5.49998262)
}
}
{
\pscustom[linewidth=1,linecolor=black,fillstyle=solid,fillcolor=black]%arrowhead
{
\newpath
\moveto(60.5,40.49998262)
\lineto(56.5,36.49998262)
\lineto(70.5,40.49998262)
\lineto(56.5,44.49998262)
\lineto(60.5,40.49998262)
\closepath
}
}
{
\pscustom[linewidth=1,linecolor=black,fillstyle=solid,fillcolor=black]%arrowhead
{
\newpath
\moveto(22.02786405,64.44429453)
\lineto(16.6613009,66.23314891)
\lineto(26.5,55.50002262)
\lineto(23.81671843,69.81085767)
\lineto(22.02786405,64.44429453)
\closepath
}
}
{
\pscustom[linewidth=1,linecolor=black,fillstyle=solid,fillcolor=black]%arrowhead
{
\newpath
\moveto(115.5,40.49998262)
\lineto(111.5,36.49998262)
\lineto(125.5,40.49998262)
\lineto(111.5,44.49998262)
\lineto(115.5,40.49998262)
\closepath
}
}
{
\pscustom[linewidth=1,linecolor=black,fillstyle=solid,fillcolor=black]%arrowhead
{
\newpath
\moveto(22.02786405,15.55571071)
\lineto(23.81671843,10.18914756)
\lineto(26.5,24.49998262)
\lineto(16.6613009,13.76685633)
\lineto(22.02786405,15.55571071)
\closepath
}
}
{
\pscustom[linewidth=1,linecolor=black,fillstyle=solid,fillcolor=black]%arrowhead
{
\newpath
\moveto(157.02786405,21.44425453)
\lineto(151.6613009,23.23310891)
\lineto(161.5,12.49998262)
\lineto(158.81671843,26.81081767)
\lineto(157.02786405,21.44425453)
\closepath
}
}
{
\pscustom[linewidth=1,linecolor=black,fillstyle=solid,fillcolor=black]%arrowhead
{
\newpath
\moveto(157.02786405,59.55575071)
\lineto(158.81671843,54.18918756)
\lineto(161.5,68.50002262)
\lineto(151.6613009,57.76689633)
\lineto(157.02786405,59.55575071)
\closepath
}
}
{
\pscustom[linewidth=1,linecolor=black,fillstyle=solid,fillcolor=black]%arrowhead
{
\newpath
\moveto(15.5,40.49998262)
\lineto(19.5,44.49998262)
\lineto(5.5,40.49998262)
\lineto(19.5,36.49998262)
\lineto(15.5,40.49998262)
\closepath
}
}
{
\pscustom[linewidth=1,linecolor=black,fillstyle=solid,fillcolor=black]%arrowhead
{
\newpath
\moveto(165.5,40.49998262)
\lineto(161.5,36.49998262)
\lineto(175.5,40.49998262)
\lineto(161.5,44.49998262)
\lineto(165.5,40.49998262)
\closepath
}
}
\end{pspicture}&%LaTeX with PSTricks extensions
%%Creator: 0.46
%%Please note this file requires PSTricks extensions
\psset{xunit=.5pt,yunit=.5pt,runit=.5pt}
\begin{pspicture}(52,18)
{
\pscustom[linewidth=2,linecolor=black]
{
\newpath
\moveto(1,9)
\lineto(51,9)
}
}
{
\pscustom[linewidth=2,linecolor=black,fillstyle=solid,fillcolor=black]
{
\newpath
\moveto(31,9)
\lineto(23,1)
\lineto(51,9)
\lineto(23,17)
\lineto(31,9)
\closepath
}
}
\end{pspicture}&%LaTeX with PSTricks extensions
%%Creator: 0.46
%%Please note this file requires PSTricks extensions
\psset{xunit=.5pt,yunit=.5pt,runit=.5pt}
\begin{pspicture}(175,75)
{
\pscustom[linewidth=2,linecolor=black,fillstyle=solid,fillcolor=black]%vertex
{
\newpath
\moveto(46,41.000057)
\curveto(46,35.480057)(41.52,31.000057)(36,31.000057)
\curveto(30.48,31.000057)(26,35.480057)(26,41.000057)
\curveto(26,46.520057)(30.48,51.000057)(36,51.000057)
\curveto(41.52,51.000057)(46,46.520057)(46,41.000057)
\closepath
}
}
{
\pscustom[linewidth=2,linecolor=black,fillstyle=solid,fillcolor=black]%vertex
{
\newpath
\moveto(156,41.000017)
\curveto(156,35.480017)(151.52,31.000017)(146,31.000017)
\curveto(140.48,31.000017)(136,35.480017)(136,41.000017)
\curveto(136,46.520017)(140.48,51.000017)(146,51.000017)
\curveto(151.52,51.000017)(156,46.520017)(156,41.000017)
\closepath
}
}
{
\pscustom[linewidth=1,linecolor=black]
{
\newpath
\moveto(5.5,40.49998262)
\lineto(175.5,40.49998262)
}
}
{
\pscustom[linewidth=1,linecolor=black]
{
\newpath
\moveto(15.5,75.50002262)
\lineto(35.5,40.49998262)
\lineto(15.5,5.49998262)
}
}
{
\pscustom[linewidth=1,linecolor=black]
{
\newpath
\moveto(165.5,75.50002262)
\lineto(145.5,40.49998262)
\lineto(165.5,5.49998262)
}
}
{
\pscustom[linewidth=1,linecolor=black,fillstyle=solid,fillcolor=black]%arrowhead
{
\newpath
\moveto(85,40.49998262)
\lineto(81,36.49998262)
\lineto(95,40.49998262)
\lineto(81,44.49998262)
\lineto(85,40.49998262)
\closepath
}
}
{
\pscustom[linewidth=1,linecolor=black,fillstyle=solid,fillcolor=black]%arrowhead
{
\newpath
\moveto(22.02786405,64.44429453)
\lineto(16.6613009,66.23314891)
\lineto(26.5,55.50002262)
\lineto(23.81671843,69.81085767)
\lineto(22.02786405,64.44429453)
\closepath
}
}
{
\pscustom[linewidth=1,linecolor=black,fillstyle=solid,fillcolor=black]%arrowhead
{
\newpath
\moveto(22.02786405,15.55571071)
\lineto(23.81671843,10.18914756)
\lineto(26.5,24.49998262)
\lineto(16.6613009,13.76685633)
\lineto(22.02786405,15.55571071)
\closepath
}
}
{
\pscustom[linewidth=1,linecolor=black,fillstyle=solid,fillcolor=black]%arrowhead
{
\newpath
\moveto(157.02786405,21.44425453)
\lineto(151.6613009,23.23310891)
\lineto(161.5,12.49998262)
\lineto(158.81671843,26.81081767)
\lineto(157.02786405,21.44425453)
\closepath
}
}
{
\pscustom[linewidth=1,linecolor=black,fillstyle=solid,fillcolor=black]%arrowhead
{
\newpath
\moveto(157.02786405,59.55575071)
\lineto(158.81671843,54.18918756)
\lineto(161.5,68.50002262)
\lineto(151.6613009,57.76689633)
\lineto(157.02786405,59.55575071)
\closepath
}
}
{
\pscustom[linewidth=1,linecolor=black,fillstyle=solid,fillcolor=black]%arrowhead
{
\newpath
\moveto(15.5,40.49998262)
\lineto(19.5,44.49998262)
\lineto(5.5,40.49998262)
\lineto(19.5,36.49998262)
\lineto(15.5,40.49998262)
\closepath
}
}
{
\pscustom[linewidth=1,linecolor=black,fillstyle=solid,fillcolor=black]%arrowhead
{
\newpath
\moveto(165.5,40.49998262)
\lineto(161.5,36.49998262)
\lineto(175.5,40.49998262)
\lineto(161.5,44.49998262)
\lineto(165.5,40.49998262)
\closepath
}
}
\end{pspicture}\\
%LaTeX with PSTricks extensions
%%Creator: 0.46
%%Please note this file requires PSTricks extensions
\psset{xunit=.5pt,yunit=.5pt,runit=.5pt}
\begin{pspicture}(145,155)
{
\newgray{lightgrey}{.8}
\newgray{lightishgrey}{.7}
\newgray{grey}{.6}
\newgray{midgrey}{.4}
\newgray{darkgrey}{.3}
}
{
\pscustom[linewidth=1.5,linecolor=darkgrey,fillstyle=solid,fillcolor=lightishgrey]%1-handle
{
\newpath
\moveto(28.9856205,114.45751953)
\lineto(54.9856205,114.45751953)
\lineto(54.9856205,24.26251984)
\lineto(28.9856205,24.26251984)
\lineto(28.9856205,114.45751953)
\closepath
}
}
{
\pscustom[linewidth=1.5,linecolor=darkgrey,fillstyle=solid,fillcolor=lightishgrey]%1-handle
{
\newpath
\moveto(82.98562622,114.65246582)
\lineto(108.98562622,114.65246582)
\lineto(108.98562622,24.45746613)
\lineto(82.98562622,24.45746613)
\lineto(82.98562622,114.65246582)
\closepath
}
}
{
\pscustom[linewidth=1.5,linecolor=darkgrey,fillstyle=solid,fillcolor=grey]%0-handle
{
\newpath
\moveto(2.98675225,-2.21549303)
\curveto(15.8632781,34.21587645)(55.88128011,53.33285098)(92.31264959,40.45632513)
\curveto(113.5721863,32.94222456)(128.5806667,17.54875219)(135.55400815,-3.89424475)
}
}
{
\pscustom[linewidth=1.5,linecolor=darkgrey,fillstyle=solid,fillcolor=grey]%0-handle
{
\newpath
\moveto(135.20550307,141.76547166)
\curveto(122.67953247,105.21209069)(82.84701525,85.71158372)(46.29363428,98.23755431)
\curveto(25.15669904,105.48067717)(10.10606954,120.49800499)(2.8161442,141.61884419)
}
}
{
\pscustom[linewidth=3,linecolor=black]
{
\newpath
\moveto(28.985621,106.33995262)
\lineto(54.985621,42.45748262)
}
}
{
\pscustom[linewidth=3,linecolor=black,linestyle=dashed,dash=12 12]
{
\newpath
\moveto(28.985621,32.38008262)
\lineto(54.985621,96.45751262)
}
}
{
\pscustom[linewidth=3,linecolor=black]
{
\newpath
\moveto(82.985617,96.45751262)
\lineto(108.98561,32.57508262)
}
}
{
\pscustom[linewidth=3,linecolor=black,linestyle=dashed,dash=12 12]
{
\newpath
\moveto(82.985617,42.45748262)
\lineto(108.98561,106.53495262)
}
}
{
\pscustom[linewidth=3,linecolor=black]
{
\newpath
\moveto(132.98561,2.45748262)
\curveto(124.98561,18.45748262)(120.98561,22.45748262)(108.98561,32.45748262)
}
}
{
\pscustom[linewidth=3,linecolor=black]
{
\newpath
\moveto(4.9856207,2.45748262)
\curveto(12.985621,18.45748262)(16.985621,22.45748262)(28.985621,32.45748262)
}
}
{
\pscustom[linewidth=3,linecolor=black]
{
\newpath
\moveto(4.9856207,136.45750262)
\curveto(12.985621,120.45750262)(16.985621,116.45750262)(28.985621,106.45750262)
}
}
{
\pscustom[linewidth=3,linecolor=black]
{
\newpath
\moveto(132.98561,136.45750262)
\curveto(124.98561,120.45750262)(120.98561,116.45750262)(108.98561,106.45750262)
}
}
{
\pscustom[linewidth=3,linecolor=black]
{
\newpath
\moveto(53.985621,43.45748262)
\curveto(63.985621,44.45748262)(73.985621,44.45748262)(83.985617,43.45748262)
}
}
{
\pscustom[linewidth=3,linecolor=black]
{
\newpath
\moveto(53.985621,95.45753262)
\curveto(63.985621,94.45753262)(73.985621,94.45753262)(83.985617,95.45753262)
}
}
{
\pscustom[linewidth=2,linecolor=black,linestyle=dashed,dash=4 4]
{
\newpath
\moveto(104.99764,69.48130783)
\curveto(104.99764,41.88781246)(89.31225392,19.49309158)(69.9856175,19.49309158)
\curveto(50.65898108,19.49309158)(34.973595,41.88781246)(34.973595,69.48130783)
\curveto(34.973595,97.0748032)(50.65898108,119.46952408)(69.9856175,119.46952408)
\curveto(89.31225392,119.46952408)(104.99764,97.0748032)(104.99764,69.48130783)
\closepath
}
}
{
\pscustom[linestyle=none,fillstyle=solid,fillcolor=white]
{
\newpath
\moveto(0,4.82116699)
\lineto(139.96365356,4.82116699)
\lineto(139.96365356,-5.00002098)
\lineto(0,-5.00002098)
\lineto(0,4.82116699)
\closepath
}
}
{
\pscustom[linestyle=none,fillstyle=solid,fillcolor=white]
{
\newpath
\moveto(-0.16256142,143.97436523)
\lineto(135.17056847,143.97436523)
\lineto(135.17056847,136.03328419)
\lineto(-0.16256142,136.03328419)
\lineto(-0.16256142,143.97436523)
\closepath
}
}
\end{pspicture}&&%LaTeX with PSTricks extensions
%%Creator: inkscape 0.47
%%Please note this file requires PSTricks extensions
\psset{xunit=.5pt,yunit=.5pt,runit=.5pt}
\begin{pspicture}(140,140)
{
\newgray{lightgrey}{.8}
\newgray{lightishgrey}{.7}
\newgray{grey}{.6}
\newgray{midgrey}{.4}
\newgray{darkgrey}{.3}
}
{
\pscustom[linewidth=1.5,linecolor=darkgrey,fillstyle=solid,fillcolor=lightishgrey]%1-handle
{
\newpath
\moveto(29,114.99999738)
\lineto(109,114.99999738)
\lineto(109,24.99999738)
\lineto(29,24.99999738)
\lineto(29,114.99999738)
\closepath
}
}
{
\pscustom[linewidth=1.5,linecolor=darkgrey,fillstyle=solid,fillcolor=grey]%0-handle
{
\newpath
\moveto(2.98675225,-2.21549565)
\curveto(15.86992047,34.23466699)(55.86248695,53.33949074)(92.31264959,40.45632251)
\curveto(112.83892224,33.20139128)(128.82118467,16.80915777)(135.55400815,-3.89424737)
}
}
{
\pscustom[linewidth=1.5,linecolor=darkgrey,fillstyle=solid,fillcolor=grey]%0-handle
{
\newpath
\moveto(135.20550307,141.76546905)
\curveto(122.67307093,105.19323197)(82.86587136,85.70511956)(46.29363428,98.2375517)
\curveto(25.89235567,105.2285826)(9.85234865,121.2330988)(2.8161442,141.61884158)
}
}
{
\pscustom[linewidth=3,linecolor=black,linestyle=dashed,dash=12 12]
{
\newpath
\moveto(29,31.99998)
\lineto(109,107)
}
}
{
\pscustom[linewidth=3,linecolor=black]
{
\newpath
\moveto(109,31.99998)
\lineto(29,107)
}
}
{
\pscustom[linewidth=3,linecolor=black]
{
\newpath
\moveto(108.7407011,32.0732635)
\curveto(121.38238207,23.35048048)(130.80431292,10.71205513)(135.55425815,-3.89400737)
}
}
{
\pscustom[linewidth=3,linecolor=black]
{
\newpath
\moveto(2.98700225,-2.21525565)
\curveto(7.89078636,11.65895024)(17.04221526,23.6349196)(29.14101663,32.0110579)
}
}
{
\pscustom[linewidth=3,linecolor=black]
{
\newpath
\moveto(108.61317842,106.84992664)
\curveto(121.28276617,115.57363712)(130.72604416,128.22774346)(135.48325815,142.85615737)
}
}
{
\pscustom[linewidth=3,linecolor=black]
{
\newpath
\moveto(2.91600225,141.17740565)
\curveto(7.8357399,127.25806264)(17.03034657,115.25051956)(29.18530669,106.87144606)
}
}
{
\pscustom[linestyle=none,fillstyle=solid,fillcolor=white]
{
\newpath
\moveto(-0.16256142,143.97436262)
\lineto(140.17056847,143.97436262)
\lineto(140.17056847,136.03328157)
\lineto(-0.16256142,136.03328157)
\lineto(-0.16256142,143.97436262)
\closepath
}
}
{
\pscustom[linestyle=none,fillstyle=solid,fillcolor=white]
{
\newpath
\moveto(0,4.82116438)
\lineto(139.96365356,4.82116438)
\lineto(139.96365356,-5.0000236)
\lineto(0,-5.0000236)
\lineto(0,4.82116438)
\closepath
}
}
\end{pspicture}\\
\\
%LaTeX with PSTricks extensions
%%Creator: 0.46
%%Please note this file requires PSTricks extensions
\psset{xunit=.5pt,yunit=.5pt,runit=.5pt}
\begin{pspicture}(85,120)
{
\newrgbcolor{curcolor}{0 0 0}
\pscustom[linestyle=none,fillstyle=solid,fillcolor=curcolor]
{
\newpath
\moveto(55.5,45.50004)
\curveto(55.5,39.98004)(51.02,35.50004)(45.5,35.50004)
\curveto(39.98,35.50004)(35.5,39.98004)(35.5,45.50004)
\curveto(35.5,51.02004)(39.98,55.50004)(45.5,55.50004)
\curveto(51.02,55.50004)(55.5,51.02004)(55.5,45.50004)
\closepath
}
}
{
\newrgbcolor{curcolor}{0 0 0}
\pscustom[linewidth=2,linecolor=curcolor]
{
\newpath
\moveto(55.5,45.50004)
\curveto(55.5,39.98004)(51.02,35.50004)(45.5,35.50004)
\curveto(39.98,35.50004)(35.5,39.98004)(35.5,45.50004)
\curveto(35.5,51.02004)(39.98,55.50004)(45.5,55.50004)
\curveto(51.02,55.50004)(55.5,51.02004)(55.5,45.50004)
\closepath
}
}
{
\newrgbcolor{curcolor}{0 0 0}
\pscustom[linewidth=1,linecolor=curcolor]
{
\newpath
\moveto(5.5,5.49998262)
\lineto(45.5,45.49998262)
\lineto(85.5,5.49998262)
}
}
{
\newrgbcolor{curcolor}{0 0 0}
\pscustom[linewidth=1,linecolor=curcolor]
{
\newpath
\moveto(29.22793,29.22798262)
\lineto(32.22793,32.22798262)
}
}
{
\newrgbcolor{curcolor}{0 0 0}
\pscustom[linestyle=none,fillstyle=solid,fillcolor=curcolor]
{
\newpath
\moveto(25.15686219,25.15691481)
\lineto(25.15686219,19.50006056)
\lineto(32.22793,32.22798262)
\lineto(19.50000794,25.15691481)
\lineto(25.15686219,25.15691481)
\closepath
}
}
{
\newrgbcolor{curcolor}{0 0 0}
\pscustom[linewidth=1,linecolor=curcolor]
{
\newpath
\moveto(25.15686219,25.15691481)
\lineto(25.15686219,19.50006056)
\lineto(32.22793,32.22798262)
\lineto(19.50000794,25.15691481)
\lineto(25.15686219,25.15691481)
\closepath
}
}
{
\newrgbcolor{curcolor}{0 0 0}
\pscustom[linewidth=1,linecolor=curcolor]
{
\newpath
\moveto(71.22792,19.77208262)
\lineto(74.22792,16.77208262)
}
}
{
\newrgbcolor{curcolor}{0 0 0}
\pscustom[linestyle=none,fillstyle=solid,fillcolor=curcolor]
{
\newpath
\moveto(67.15685219,23.84315043)
\lineto(61.49999794,23.84315043)
\lineto(74.22792,16.77208262)
\lineto(67.15685219,29.50000468)
\lineto(67.15685219,23.84315043)
\closepath
}
}
{
\newrgbcolor{curcolor}{0 0 0}
\pscustom[linewidth=1,linecolor=curcolor]
{
\newpath
\moveto(67.15685219,23.84315043)
\lineto(61.49999794,23.84315043)
\lineto(74.22792,16.77208262)
\lineto(67.15685219,29.50000468)
\lineto(67.15685219,23.84315043)
\closepath
}
}
{
\newrgbcolor{curcolor}{0 0 0}
\pscustom[linewidth=1,linecolor=curcolor]
{
\newpath
\moveto(45,44.22798262)
\curveto(45,44.22798262)(67.029312,62.89554262)(70,74.22802262)
\curveto(74.149629,90.05789262)(60.843244,109.96233262)(45,109.22802262)
\curveto(29.178069,108.49469262)(16.929679,89.06495262)(20,74.22802262)
\curveto(22.374055,62.75571262)(45,44.22798262)(45,44.22798262)
\closepath
}
}
{
\newrgbcolor{curcolor}{0 0 0}
\pscustom[linewidth=1,linecolor=curcolor]
{
\newpath
\moveto(70,79.22802262)
\lineto(70,85.22802262)
}
}
{
\newrgbcolor{curcolor}{0 0 0}
\pscustom[linestyle=none,fillstyle=solid,fillcolor=curcolor]
{
\newpath
\moveto(70,75.22802262)
\lineto(74,71.22802262)
\lineto(70,85.22802262)
\lineto(66,71.22802262)
\lineto(70,75.22802262)
\closepath
}
}
{
\newrgbcolor{curcolor}{0 0 0}
\pscustom[linewidth=1,linecolor=curcolor]
{
\newpath
\moveto(70,75.22802262)
\lineto(74,71.22802262)
\lineto(70,85.22802262)
\lineto(66,71.22802262)
\lineto(70,75.22802262)
\closepath
}
}
{
\newrgbcolor{curcolor}{0 0 0}
\pscustom[linestyle=none,fillstyle=solid,fillcolor=curcolor]
{
\newpath
\moveto(55.5,107.2280174)
\curveto(55.5,101.7080174)(51.02,97.2280174)(45.5,97.2280174)
\curveto(39.98,97.2280174)(35.5,101.7080174)(35.5,107.2280174)
\curveto(35.5,112.7480174)(39.98,117.2280174)(45.5,117.2280174)
\curveto(51.02,117.2280174)(55.5,112.7480174)(55.5,107.2280174)
\closepath
}
}
{
\newrgbcolor{curcolor}{0 0 0}
\pscustom[linewidth=2,linecolor=curcolor]
{
\newpath
\moveto(55.5,107.2280174)
\curveto(55.5,101.7080174)(51.02,97.2280174)(45.5,97.2280174)
\curveto(39.98,97.2280174)(35.5,101.7080174)(35.5,107.2280174)
\curveto(35.5,112.7480174)(39.98,117.2280174)(45.5,117.2280174)
\curveto(51.02,117.2280174)(55.5,112.7480174)(55.5,107.2280174)
\closepath
}
}
{
\newrgbcolor{curcolor}{0 0 0}
\pscustom[linewidth=1,linecolor=curcolor]
{
\newpath
\moveto(20,80.22802262)
\lineto(20,74.22802262)
}
}
{
\newrgbcolor{curcolor}{0 0 0}
\pscustom[linestyle=none,fillstyle=solid,fillcolor=curcolor]
{
\newpath
\moveto(20,84.22802262)
\lineto(16,88.22802262)
\lineto(20,74.22802262)
\lineto(24,88.22802262)
\lineto(20,84.22802262)
\closepath
}
}
{
\newrgbcolor{curcolor}{0 0 0}
\pscustom[linewidth=1,linecolor=curcolor]
{
\newpath
\moveto(20,84.22802262)
\lineto(16,88.22802262)
\lineto(20,74.22802262)
\lineto(24,88.22802262)
\lineto(20,84.22802262)
\closepath
}
}
\end{pspicture}&%LaTeX with PSTricks extensions
%%Creator: 0.46
%%Please note this file requires PSTricks extensions
\psset{xunit=.5pt,yunit=.5pt,runit=.5pt}
\begin{pspicture}(52,18)
{
\pscustom[linewidth=2,linecolor=black]
{
\newpath
\moveto(1,9)
\lineto(51,9)
}
}
{
\pscustom[linewidth=2,linecolor=black,fillstyle=solid,fillcolor=black]
{
\newpath
\moveto(31,9)
\lineto(23,1)
\lineto(51,9)
\lineto(23,17)
\lineto(31,9)
\closepath
}
}
\end{pspicture}&%LaTeX with PSTricks extensions
%%Creator: 0.46
%%Please note this file requires PSTricks extensions
\psset{xunit=.5pt,yunit=.5pt,runit=.5pt}
\begin{pspicture}(85,115)
{
\pscustom[linewidth=2,linecolor=black,fillstyle=solid,fillcolor=black]%vertex
{
\newpath
\moveto(55.5,45.50004)
\curveto(55.5,39.98004)(51.02,35.50004)(45.5,35.50004)
\curveto(39.98,35.50004)(35.5,39.98004)(35.5,45.50004)
\curveto(35.5,51.02004)(39.98,55.50004)(45.5,55.50004)
\curveto(51.02,55.50004)(55.5,51.02004)(55.5,45.50004)
\closepath
}
}
{
\pscustom[linewidth=1,linecolor=black]
{
\newpath
\moveto(5.5,5.49998262)
\lineto(45.5,45.49998262)
\lineto(85.5,5.49998262)
}
}
{
\pscustom[linewidth=1,linecolor=black,fillstyle=solid,fillcolor=black]%arrowhead
{
\newpath
\moveto(25.15686219,25.15691481)
\lineto(25.15686219,19.50006056)
\lineto(32.22793,32.22798262)
\lineto(19.50000794,25.15691481)
\lineto(25.15686219,25.15691481)
\closepath
}
}
{
\pscustom[linewidth=1,linecolor=black,fillstyle=solid,fillcolor=black]%arrowhead
{
\newpath
\moveto(67.15685219,23.84315043)
\lineto(61.49999794,23.84315043)
\lineto(74.22792,16.77208262)
\lineto(67.15685219,29.50000468)
\lineto(67.15685219,23.84315043)
\closepath
}
}
{
\pscustom[linewidth=1,linecolor=black]%loop
{
\newpath
\moveto(45,44.22798262)
\curveto(45,44.22798262)(67.029312,62.89554262)(70,74.22802262)
\curveto(74.149629,90.05789262)(60.843244,109.96233262)(45,109.22802262)
\curveto(29.178069,108.49469262)(16.929679,89.06495262)(20,74.22802262)
\curveto(22.374055,62.75571262)(45,44.22798262)(45,44.22798262)
\closepath
}
}
{
\pscustom[linewidth=1,linecolor=black,fillstyle=solid,fillcolor=black]%arrowhead
{
\newpath
\moveto(42.07722123,109.46836996)
\lineto(37.61197079,105.99539739)
\lineto(52,108.22802262)
\lineto(38.60424866,113.93362041)
\lineto(42.07722123,109.46836996)
\closepath
}
}
\end{pspicture}\\
%LaTeX with PSTricks extensions
%%Creator: 0.46
%%Please note this file requires PSTricks extensions
\psset{xunit=.25pt,yunit=.25pt,runit=.25pt}
\begin{pspicture}(440,330)
{
\newgray{lightgrey}{.8}
\newgray{lightishgrey}{.7}
\newgray{grey}{.6}
\newgray{midgrey}{.4}
\newgray{darkgrey}{.3}
}
{
\pscustom[linewidth=3,linecolor=darkgrey,fillstyle=solid,fillcolor=lightishgrey]%1-handle
{
\newpath
\moveto(77.57363832,217.5735773)
\lineto(120.00004093,259.99997991)
\lineto(204.85284074,175.14718009)
\lineto(162.42643814,132.72077749)
\lineto(77.57363832,217.5735773)
\closepath
}
}
{
\pscustom[linewidth=3,linecolor=darkgrey,fillstyle=solid,fillcolor=lightishgrey]%1-handle
{
\newpath
\moveto(332.42642643,217.5735773)
\lineto(290.00002383,259.99997991)
\lineto(205.14722401,175.14718009)
\lineto(247.57362662,132.72077749)
\lineto(332.42642643,217.5735773)
\closepath
}
}
{
\pscustom[linewidth=3,linecolor=darkgrey,fillstyle=solid,fillcolor=lightishgrey]%1-handle
{
\newpath
\moveto(406.22652673,87.04325264)
\lineto(355.9748552,132.96307166)
\lineto(289.99998019,10.00001323)
\lineto(340.25165172,-35.91980579)
\lineto(406.22652673,87.04325264)
\closepath
}
}
{
\pscustom[linewidth=3,linecolor=darkgrey,fillstyle=solid,fillcolor=lightishgrey]%1-handle
{
\newpath
\moveto(3.77346849,87.0432393)
\lineto(54.02514002,132.96305832)
\lineto(120.00001503,9.9999999)
\lineto(69.74834349,-35.91981913)
\lineto(3.77346849,87.0432393)
\closepath
}
}
{
\pscustom[linewidth=3,linecolor=darkgrey,fillstyle=solid,fillcolor=grey]%0-handle
{
\newpath
\moveto(10,39.99998262)
\curveto(60,39.99998262)(90,90.00000262)(90,140.00000262)
\curveto(90,190.00000262)(105,280.00000262)(205,280.00000262)
\curveto(305,280.00000262)(320,190.00000262)(320,140.00000262)
\curveto(320,90.00000262)(350,39.99998262)(400,39.99998262)
\lineto(400,340.00000262)
\lineto(10,340.00000262)
\lineto(10,39.99998262)
\closepath
}
}
{
\pscustom[linewidth=3,linecolor=darkgrey,fillstyle=solid,fillcolor=grey]%0-handle
{
\newpath
\moveto(249.909086,144.90910879)
\curveto(249.909086,120.06910934)(229.74908645,99.90910979)(204.909087,99.90910979)
\curveto(180.06908755,99.90910979)(159.909088,120.06910934)(159.909088,144.90910879)
\curveto(159.909088,169.74910824)(180.06908755,189.90910779)(204.909087,189.90910779)
\curveto(229.74908645,189.90910779)(249.909086,169.74910824)(249.909086,144.90910879)
\closepath
}
}
{
\pscustom[linewidth=6,linecolor=black]
{
\newpath
\moveto(10,39.99998262)
\curveto(16,39.99998262)(22,40.99998262)(28,41.99998262)
}
}
{
\pscustom[linewidth=6,linecolor=black]
{
\newpath
\moveto(400,39.99998262)
\curveto(394,39.99998262)(388,40.99998262)(382,41.99998262)
}
}
{
\pscustom[linewidth=6,linecolor=black]
{
\newpath
\moveto(332,87.00000262)
\lineto(345,9.99998262)
}
}
{
\pscustom[linewidth=6,linecolor=black,linestyle=dashed,dash=24 24]
{
\newpath
\moveto(78,87.00000262)
\lineto(65,9.99998262)
}
}
{
\pscustom[linewidth=6,linecolor=black,linestyle=dashed,dash=24 24]
{
\newpath
\moveto(382,41.99998262)
\lineto(315,9.99998262)
}
}
{
\pscustom[linewidth=6,linecolor=black]
{
\newpath
\moveto(28.000004,41.99998262)
\lineto(95,9.99998262)
}
}
{
\pscustom[linewidth=6,linecolor=black]
{
\newpath
\moveto(192,188.00000262)
\curveto(205,193.00000262)(218,188.00000262)(218,188.00000262)
}
}
{
\pscustom[linewidth=6,linecolor=black]
{
\newpath
\moveto(249,134.00000262)
\curveto(244,114.00000262)(225,100.00000262)(205,100.00000262)
\curveto(184.97502,100.00000262)(166,115.00000262)(161,134.00000262)
}
}
{
\pscustom[linewidth=6,linecolor=black]
{
\newpath
\moveto(313,198.00000262)
\curveto(317,182.00000262)(320,164.00000262)(320,150.00000262)
\curveto(319,130.00000262)(322,109.00000262)(332,87.00000262)
}
}
{
\pscustom[linewidth=6,linecolor=black]
{
\newpath
\moveto(97,198.00000262)
\curveto(93,182.00000262)(90,164.00000262)(90,150.00000262)
\curveto(91,130.00000262)(88,108.99997262)(78,86.99997262)
}
}
{
\pscustom[linewidth=6,linecolor=black]
{
\newpath
\moveto(128,252.00000262)
\curveto(148,270.00000262)(167,280.00000262)(210,280.00000262)
\curveto(260,280.00000262)(282,251.00000262)(282,251.00000262)
}
}
{
\pscustom[linewidth=6,linecolor=black]
{
\newpath
\moveto(96,198.00000262)
\lineto(193,188.00000262)
}
}
{
\pscustom[linewidth=6,linecolor=black,linestyle=dashed,dash=24 24]
{
\newpath
\moveto(314,198.00000262)
\lineto(217,188.00000262)
}
}
{
\pscustom[linewidth=6,linecolor=black]
{
\newpath
\moveto(282,252.00000262)
\curveto(264,203.00000262)(249,134.00000262)(249,134.00000262)
}
}
{
\pscustom[linewidth=6,linecolor=black,linestyle=dashed,dash=24 24]
{
\newpath
\moveto(128,252.00000262)
\curveto(146,203.00000262)(161,134.00000262)(161,134.00000262)
}
}
{
\pscustom[linewidth=4,linecolor=black,linestyle=dashed,dash=8 8]
{
\newpath
\moveto(300,225)
\curveto(300,183.6)(257.44,150)(205,150)
\curveto(152.56,150)(110,183.6)(110,225)
\curveto(110,266.4)(152.56,300)(205,300)
\curveto(257.44,300)(300,266.4)(300,225)
\closepath
}
}
{
\pscustom[linestyle=none,fillstyle=solid,fillcolor=white]
{
\newpath
\moveto(-60.30282593,349.88946533)
\lineto(460.30282593,349.88946533)
\lineto(460.30282593,329.69718552)
\lineto(-60.30282593,329.69718552)
\lineto(-60.30282593,349.88946533)
\closepath
}
}
{
\pscustom[linestyle=none,fillstyle=solid,fillcolor=white]
{
\newpath
\moveto(-9.59327888,340.0725708)
\lineto(10.31336021,340.0725708)
\lineto(10.31336021,-0.10028076)
\lineto(-9.59327888,-0.10028076)
\lineto(-9.59327888,340.0725708)
\closepath
}
}
{
\pscustom[linestyle=none,fillstyle=solid,fillcolor=white]
{
\newpath
\moveto(399.65759277,340.13867188)
\lineto(419.56405449,340.13867188)
\lineto(419.56405449,-10.12255859)
\lineto(399.65759277,-10.12255859)
\lineto(399.65759277,340.13867188)
\closepath
}
}
{
\pscustom[linestyle=none,fillstyle=solid,fillcolor=white]
{
\newpath
\moveto(-50.0941124,10.1015625)
\lineto(460.06015396,10.1015625)
\lineto(460.06015396,-39.74683762)
\lineto(-50.0941124,-39.74683762)
\lineto(-50.0941124,10.1015625)
\closepath
}
}
\end{pspicture}&&%LaTeX with PSTricks extensions
%%Creator: 0.46
%%Please note this file requires PSTricks extensions
\psset{xunit=.25pt,yunit=.25pt,runit=.25pt}
\begin{pspicture}(440,330)
{
\newgray{lightgrey}{.8}
\newgray{lightishgrey}{.7}
\newgray{grey}{.6}
\newgray{midgrey}{.4}
\newgray{darkgrey}{.3}
}
{
\pscustom[linewidth=3,linecolor=darkgrey,fillstyle=solid,fillcolor=lightishgrey]%1-handle
{
\newpath
\moveto(170,290)
\lineto(240,290)
\lineto(240,160)
\lineto(170,160)
\lineto(170,290)
\closepath
}
}
{
\pscustom[linewidth=3,linecolor=darkgrey,fillstyle=solid,fillcolor=lightishgrey]%1-handle
{
\newpath
\moveto(406.22652673,87.04325264)
\lineto(355.9748552,132.96307166)
\lineto(289.99998019,10.00001323)
\lineto(340.25165172,-35.91980579)
\lineto(406.22652673,87.04325264)
\closepath
}
}
{
\pscustom[linewidth=3,linecolor=darkgrey,fillstyle=solid,fillcolor=lightishgrey]%1-handle
{
\newpath
\moveto(3.77346849,87.0432393)
\lineto(54.02514002,132.96305832)
\lineto(120.00001503,9.9999999)
\lineto(69.74834349,-35.91981913)
\lineto(3.77346849,87.0432393)
\closepath
}
}
{
\pscustom[linewidth=3,linecolor=darkgrey,fillstyle=solid,fillcolor=grey]%0-handle
{
\newpath
\moveto(10,39.99998262)
\curveto(60,39.99998262)(90,90.00000262)(90,140.00000262)
\curveto(90,190.00000262)(105,280.00000262)(205,280.00000262)
\curveto(305,280.00000262)(320,190.00000262)(320,140.00000262)
\curveto(320,90.00000262)(350,39.99998262)(400,39.99998262)
\lineto(400,340.00000262)
\lineto(10,340.00000262)
\lineto(10,39.99998262)
\closepath
}
}
{
\pscustom[linewidth=3,linecolor=darkgrey,fillstyle=solid,fillcolor=grey]%0-handle
{
\newpath
\moveto(249.909086,144.90910879)
\curveto(249.909086,120.06910934)(229.74908645,99.90910979)(204.909087,99.90910979)
\curveto(180.06908755,99.90910979)(159.909088,120.06910934)(159.909088,144.90910879)
\curveto(159.909088,169.74910824)(180.06908755,189.90910779)(204.909087,189.90910779)
\curveto(229.74908645,189.90910779)(249.909086,169.74910824)(249.909086,144.90910879)
\closepath
}
}
{
\pscustom[linewidth=6,linecolor=black]
{
\newpath
\moveto(10,39.99998262)
\curveto(16,39.99998262)(22,40.99998262)(28,41.99998262)
}
}
{
\pscustom[linewidth=6,linecolor=black]
{
\newpath
\moveto(400,39.99998262)
\curveto(394,39.99998262)(388,40.99998262)(382,41.99998262)
}
}
{
\pscustom[linewidth=6,linecolor=black]
{
\newpath
\moveto(332,87.00000262)
\lineto(345,9.99998262)
}
}
{
\pscustom[linewidth=6,linecolor=black,linestyle=dashed,dash=24 24]
{
\newpath
\moveto(78,87.00000262)
\lineto(65,9.99998262)
}
}
{
\pscustom[linewidth=6,linecolor=black,linestyle=dashed,dash=24 24]
{
\newpath
\moveto(382,41.99998262)
\lineto(315,9.99998262)
}
}
{
\pscustom[linewidth=6,linecolor=black]
{
\newpath
\moveto(28.000004,41.99998262)
\lineto(95,9.99998262)
}
}
{
\pscustom[linewidth=6,linecolor=black]
{
\newpath
\moveto(171,174.00000262)
\curveto(163,164.00000262)(159,150.00000262)(161,134.00000262)
\curveto(163.40832,114.73345262)(184.97502,100.00000262)(205,100.00000262)
\curveto(225,100.00000262)(244,114.00000262)(249,134.00000262)
\curveto(255,160.00000262)(239,174.00000262)(239,174.00000262)
}
}
{
\pscustom[linewidth=6,linecolor=black]
{
\newpath
\moveto(78,86.99997262)
\curveto(88,108.99997262)(91,130.00000262)(90,150.00000262)
\curveto(90,164.00000262)(93,182.00000262)(97,198.00000262)
\curveto(107,234.00000262)(134,265.00000262)(170,276.00000262)
}
}
{
\pscustom[linewidth=6,linecolor=black]
{
\newpath
\moveto(332,86.99997262)
\curveto(322,108.99997262)(319,130.00000262)(320,150.00000262)
\curveto(320,164.00000262)(317,182.00000262)(313,198.00000262)
\curveto(303,234.00000262)(276,265.00000262)(240,276.00000262)
}
}
{
\pscustom[linewidth=6,linecolor=black]
{
\newpath
\moveto(169,277.00000262)
\lineto(240,173.00000262)
}
}
{
\pscustom[linewidth=6,linecolor=black,linestyle=dashed,dash=24 24]
{
\newpath
\moveto(241,277.00000262)
\lineto(170,173.00000262)
}
}
{
\pscustom[linestyle=none,fillstyle=solid,fillcolor=white]
{
\newpath
\moveto(-9.97011328,349.59635458)
\lineto(409.98181735,350.22191482)
\lineto(409.98181735,329.99042564)
\lineto(-9.97011328,329.3648654)
\lineto(-9.97011328,349.59635458)
\closepath
}
}
{
\pscustom[linestyle=none,fillstyle=solid,fillcolor=white]
{
\newpath
\moveto(-9.59327888,340.0725708)
\lineto(10.31336021,340.0725708)
\lineto(10.31336021,-0.10028076)
\lineto(-9.59327888,-0.10028076)
\lineto(-9.59327888,340.0725708)
\closepath
}
}
{
\pscustom[linestyle=none,fillstyle=solid,fillcolor=white]
{
\newpath
\moveto(399.65759277,340.13867188)
\lineto(419.56405449,340.13867188)
\lineto(419.56405449,-10.12255859)
\lineto(399.65759277,-10.12255859)
\lineto(399.65759277,340.13867188)
\closepath
}
}
{
\pscustom[linestyle=none,fillstyle=solid,fillcolor=white]
{
\newpath
\moveto(-9.96357059,10.12585449)
\lineto(419.97746944,10.12585449)
\lineto(419.97746944,-39.77116776)
\lineto(-9.96357059,-39.77116776)
\lineto(-9.96357059,10.12585449)
\closepath
}
}
\end{pspicture}
\end{array}$
\caption{\label{mgraphpic4}}
\end{figure}
\subsection{Crowell's proof}
%Master document is alexpolypaper.tex.

Crowell calculates $\nrap{L}{t}$ for an alternating link $L$ by a similar but distinct method to Murasugi, again making use of directed spanning subtrees of graphs.

Let $D$ be an alternating diagram for a link $L$. Let $\mathcal{G}$ be the underlying graph of $D$ and $O$ the induced orientation on $\mathcal{G}$.

\begin{definition}
Any edge $e\in\E(\mathcal{G})$ corresponds to an arc $\rho_e$ in $D$ between two crossings. Since $D$ is alternating, one end of $\rho_e$ is part of an undercrossing, and the other end is part of an overcrossing. 
Let $o$ be the orientation of $\mathcal{G}$ that orients each edge $e$ from the overcrossing to the undercrossing.
Near any $v\in\V(\mathcal{G})$, $o$ is as shown in Figure \ref{cgraphpic1}.
\begin{figure}[htbp]
\centering
%LaTeX with PSTricks extensions
%%Creator: 0.46
%%Please note this file requires PSTricks extensions
\psset{xunit=.35pt,yunit=.35pt,runit=.35pt}
\begin{pspicture}(170,170)
{
\newrgbcolor{curcolor}{0 0 0}
\pscustom[linewidth=2,linecolor=black]
{
\newpath
\moveto(170.00002,10.00000262)
\lineto(9.99998,170.00000262)
}
}
{
\newrgbcolor{curcolor}{0 0 0}
\pscustom[linewidth=2,linecolor=black]
{
\newpath
\moveto(10,9.99998262)
\lineto(79.664805,79.66480262)
}
}
{
\newrgbcolor{curcolor}{0 0 0}
\pscustom[linewidth=2,linecolor=black]
{
\newpath
\moveto(170.16667,170.16668262)
\lineto(100.32961,100.32961262)
}
}
{
\pscustom[linewidth=2,linecolor=black,fillstyle=solid,fillcolor=white]%arrowhead
{
\newpath
\moveto(40.09353533,139.90645436)
\lineto(55.53674797,135.77695275)
\lineto(44.22304093,124.46324279)
\lineto(40.09353533,139.90645436)
\closepath
}
}
{
\pscustom[linewidth=2,linecolor=black,fillstyle=solid,fillcolor=white]%arrowhead
{
\newpath
\moveto(139.90646467,40.09355088)
\lineto(124.46325203,44.22305248)
\lineto(135.77695907,55.53676245)
\lineto(139.90646467,40.09355088)
\closepath
}
}
{
\pscustom[linewidth=2,linecolor=black,fillstyle=solid,fillcolor=white]%arrowhead
{
\newpath
\moveto(120.26021527,120.26021794)
\lineto(124.38971733,135.70343045)
\lineto(135.70342696,124.38972308)
\lineto(120.26021527,120.26021794)
\closepath
}
}
{
\pscustom[linewidth=2,linecolor=black,fillstyle=solid,fillcolor=white]%arrowhead
{
\newpath
\moveto(59.90645179,59.90646724)
\lineto(55.77695327,44.46325378)
\lineto(44.46324105,55.77695855)
\lineto(59.90645179,59.90646724)
\closepath
}
}
\end{pspicture}
\caption{\label{cgraphpic1}}
\end{figure}
\end{definition}

\begin{definition}
Define subsets $\mathcal{H},\mathcal{K}\subseteq\E(\mathcal{G})$ by at every vertex putting the incoming edges with respect to $o$ into $\mathcal{H},\mathcal{K}$ as shown in Figure \ref{cgraphpic2}.

Define a map $\alpha\colon \E(\mathcal{G}) \to \{1,-t\}$ by
\[
\alpha(e)=
\begin{cases}
-t & \text{if }e\in\mathcal{H}\\
\hfill 1 & \text{if }e\in\mathcal{K}.\\ 
\end{cases}
\]
\end{definition}
\begin{figure}[htbp]
\centering
%LaTeX with PSTricks extensions
%%Creator: 0.46
%%Please note this file requires PSTricks extensions
\psset{xunit=.35pt,yunit=.35pt,runit=.35pt}
\begin{pspicture}(170,200)
{
\pscustom[linewidth=2,linecolor=black]
{
\newpath
\moveto(170,9.99998262)
\lineto(9.99998,170.00000262)
}
}
{
\pscustom[linewidth=2,linecolor=black]
{
\newpath
\moveto(170.04339,170.04339262)
\lineto(100.00376,100.00375262)
}
}
{
\pscustom[linewidth=2,linecolor=black]
{
\newpath
\moveto(80.03963,80.03965262)
\lineto(10,9.99998262)
}
}
{
\pscustom[linewidth=2,linecolor=black,fillstyle=solid,fillcolor=black]%arrowhead
{
\newpath
\moveto(53.97545703,126.0245384)
\lineto(65.28916553,126.02453953)
\lineto(39.83332,140.16667262)
\lineto(53.97545816,114.7108299)
\lineto(53.97545703,126.0245384)
\closepath
}
}
{
\pscustom[linewidth=2,linecolor=black,fillstyle=solid,fillcolor=black]%arrowhead
{
\newpath
\moveto(133.97544858,46.02456995)
\lineto(145.28915708,46.02456431)
\lineto(119.83332,60.16671262)
\lineto(133.97544294,34.71086145)
\lineto(133.97544858,46.02456995)
\closepath
}
}
{
\pscustom[linewidth=1,linecolor=gray]
{
\newpath
\moveto(35,80.00000262)
\curveto(40,60.00000262)(40,70.00000262)(45,70.00000262)
\curveto(50,70.00000262)(50,60.00000262)(50,60.00000262)
}
}
{
\pscustom[linewidth=0.5,linecolor=darkgray,fillstyle=solid,fillcolor=gray]
{
\newpath
\moveto(50,64.00000262)
\lineto(48,66.00000262)
\lineto(50,59.00000262)
\lineto(52,66.00000262)
\lineto(50,64.00000262)
\closepath
}
}
{
\pscustom[linewidth=1,linecolor=gray]
{
\newpath
\moveto(155,110.00000262)
\curveto(135,115.00000262)(145,115.00000262)(145,120.00000262)
\curveto(145,125.00000262)(135,125.00000262)(135,125.00000262)
}
}
{
\pscustom[linewidth=0.5,linecolor=darkgray,fillstyle=solid,fillcolor=gray]
{
\newpath
\moveto(139,125.00000262)
\lineto(141,127.00000262)
\lineto(134,125.00000262)
\lineto(141,123.00000262)
\lineto(139,125.00000262)
\closepath
}
}
{
\put(15,85){$\mathcal{H}$}
\put(160,100){$\mathcal{K}$}
}
\end{pspicture}
\caption{\label{cgraphpic2}}
\end{figure}

\begin{theorem}[\cite{MR0099665} Theorem 2.12]\label{polyisctrees}
For any $v\in\V(\mathcal{G})$,
\[
\rap{L}{t}=\sum_{\mathcal{T}\in\Tr(\mathcal{G},v)} \left(\prod_{e\in\mathcal{T}} \alpha (e)\right).
\]
\end{theorem}

\begin{lemma}[\cite{MR0099665} 4.7]
A Seifert circle is a cycle with respect to $O$ and $o$ (possibly in opposite directions).

The special Seifert circles of $D$ are exactly the cycles in $\mathcal{G}$ with respect to $o$ that are contained entirely in $\mathcal{H}$ or entirely in $\mathcal{K}$.
\end{lemma}

\begin{remark}[\cite{MR0099665} 4.2]
Neither $\mathcal{H}$ nor $\mathcal{K}$ contains a pair of distinct edges with a common terminal vertex with respect to $o$.
\end{remark}

\begin{definition}
For a directed spanning subtree $\mathcal{T}$ of $\mathcal{G}$ with respect to $o$, define $\chr (\mathcal{T})$ to be the number of edges of $\mathcal{T}$ that lie in $\mathcal{K}$.

Define an \textit{\mmtree{\mathcal{H}}{}} $\mathcal{T}$ to be a directed spanning subtree with $\chr(\mathcal{T})$ minimal among such trees with the same origin as $\mathcal{T}$ (that is, $\mathcal{T}$ contains as many edges of $\mathcal{H}$ as possible).

For $v\in\V(\mathcal{G})$, let $\HTr(\mathcal{G},v)$ be the set of \mmtree{\mathcal{H}}s in $\mathcal{G}$ with origin $v$.

Define a \textit{\mmtree{\mathcal{K}}} and $\KTr(\mathcal{G},v)$ analogously.
\end{definition}

\begin{remark}
By Theorem \ref{polyisctrees}, $\nrap{L}{0}=|\KTr(\mathcal{G},v)|=|\HTr(\mathcal{G},v)|$ for any $v\in\V(\mathcal{G})$.
\end{remark}

Now suppose $D$ is special. Then every Seifert circle is contained in $\mathcal{H}$ or is contained in $\mathcal{K}$. 
Further, since no two edges in $\mathcal{H}$ share a terminal vertex, no two Seifert circles in $\mathcal{H}$ share a vertex. We can therefore collapse each such Seifert circle to a point, giving a planar graph $\graph{H}{D}$. 
Define orientation $o$ on $\graph{H}{D}$ to be that inherited from $o$ on $\mathcal{G}$. Since at each vertex of $\mathcal{G}$ exactly two edges are collapsed, and these are adjacent, incoming and outgoing edges alternate at each vertex of $\graph{H}{D}$.
Figure \ref{cgraphpic3} 
\begin{figure}[htbp]
\centering
\input{pictexfiles/cgraphpic3}
\caption{\label{cgraphpic3}}
\end{figure}
shows this process when $L$ is the knot $7_4$.

\begin{lemma}
Let $A\colon \mathcal{G}\to\graph{H}{D}$ be the map that collapses the Seifert circles in $\mathcal{H}$.
Then, for any $v\in\V(\mathcal{G})$, $A$ induces a bijection $A_v\colon \HTr(\mathcal{G},v)\to\Tr(\graph{H}{D}{},A(v))$. 
\end{lemma}
\begin{proof}
Define a map $B_v\colon\Tr(\graph{H}{D}{},A(v))\to\Tr(\mathcal{G},v)$ as follows.

Let $\mathcal{T}\in\Tr(\graph{H}{D}{},A(v))$. Then, for $e\in\mathcal{K}\subseteq\E(\mathcal{G})$, let $e\in B_v(\mathcal{T})$ if and only if $A(e)\in\mathcal{T}$. 

Consider a Seifert circle $C$ in $\mathcal{G}$ contained in $\mathcal{H}$. If $v$ lies on $C$, then no edge of $B_v(\mathcal{T})\cap\mathcal{K}$ has its terminal vertex on $C$. Let $f_C$ be the edge of $C$ whose terminal vertex is $v$. If $v$ instead does not lie on $C$, then $B_v(\mathcal{T})\cap\mathcal{K}$ contains exactly one edge $e_C$ whose terminal vertex lies on $C$. In this case, let $f_C$ be the edge of $C$ that has the same terminal vertex as $e_C$.
In either case, let $B_v(\mathcal{T})\cap C = C\setminus f_C$ (see Figure \ref{cgraphpic4}).
\begin{figure}[htbp]
\centering
%LaTeX with PSTricks extensions
%%Creator: 0.46
%%Please note this file requires PSTricks extensions
\psset{xunit=.5pt,yunit=.5pt,runit=.5pt}
\begin{pspicture}(250,250)
{
\pscustom[linewidth=1,linecolor=lightgray]
{
\newpath
\moveto(139.56218,14.99998002)
\lineto(119.56218,49.99998002)
\lineto(99.56218,14.99998002)
}
}
{
\pscustom[linewidth=1,linecolor=lightgray,fillstyle=solid,fillcolor=white]%arrow
{
\newpath
\moveto(111.56218007,35.99999745)
\lineto(110.36306456,26.13397169)
\lineto(102.41839337,32.10544906)
\lineto(111.56218007,35.99999745)
\closepath
}
}
{
\pscustom[linewidth=1,linecolor=black,linestyle=dashed,dash=4 4]%dotted line
{
\newpath
\moveto(119.56218,49.99998002)
\lineto(139.56218,14.99998002)
}
}
{
\pscustom[linewidth=1,linecolor=lightgray,fillstyle=solid,fillcolor=white]%arrow
{
\newpath
\moveto(132.56217996,26.99999742)
\lineto(123.39896356,31.26794673)
\lineto(131.67672428,37.06955025)
\lineto(132.56217996,26.99999742)
\closepath
}
}
{
\pscustom[linewidth=1,linecolor=lightgray]
{
\newpath
\moveto(19.10789,69.87336002)
\lineto(59.10789,84.87336002)
\lineto(19.10789,99.87336002)
}
}
{
\pscustom[linewidth=1,linecolor=black,linestyle=dashed,dash=4 4]%dotted line
{
\newpath
\moveto(59.10789,84.87336002)
\lineto(19.10789,69.87336002)
}
}
{
\pscustom[linewidth=1,linecolor=lightgray,fillstyle=solid,fillcolor=white]%arrow
{
\newpath
\moveto(32.10818009,74.87338007)
\lineto(37.56246985,82.74678086)
\lineto(41.65389006,74.08652697)
\lineto(32.10818009,74.87338007)
\closepath
}
}
{
\pscustom[linewidth=1,linecolor=lightgray,fillstyle=solid,fillcolor=white]%arrow
{
\newpath
\moveto(39.56218,91.99998026)
\lineto(30.01647005,91.21312681)
\lineto(34.10788995,99.87338084)
\lineto(39.56218,91.99998026)
\closepath
}
}
{
\pscustom[linewidth=1,linecolor=lightgray]
{
\newpath
\moveto(19.38271,169.79430002)
\lineto(59.38271,154.79430002)
\lineto(19.38271,139.79430002)
}
}
{
\pscustom[linewidth=1,linecolor=lightgray,fillstyle=solid,fillcolor=white]%arrow
{
\newpath
\moveto(45.56217983,159.99996267)
\lineto(36.09803628,157.6795063)
\lineto(38.82053389,167.03592323)
\lineto(45.56217983,159.99996267)
\closepath
}
}
{
\pscustom[linewidth=1,linecolor=black,linestyle=dashed,dash=4 4]%dotted line
{
\newpath
\moveto(59.82053,155.03596002)
\lineto(19.82053,140.03596002)
}
}
{
\pscustom[linewidth=1,linecolor=lightgray,fillstyle=solid,fillcolor=white]%arrow
{
\newpath
\moveto(29.82053005,144.035963)
\lineto(37.07180845,153.09815934)
\lineto(41.2942615,142.28726986)
\lineto(29.82053005,144.035963)
\closepath
}
}
{
\pscustom[linewidth=1,linecolor=lightgray]
{
\newpath
\moveto(89.56218,240.00000002)
\lineto(119.56218,190.00000002)
\lineto(149.56218,240.00000002)
}
}
{
\pscustom[linewidth=1,linecolor=black,linestyle=dashed,dash=4 4]%dotted line
{
\newpath
\moveto(89.56218,240.00000002)
\lineto(119.56218,190.00000002)
}
}
{
\pscustom[linewidth=1,linecolor=lightgray,fillstyle=solid,fillcolor=white]%arrow
{
\newpath
\moveto(99.23205269,224.03589555)
\lineto(109.56217998,219.99999743)
\lineto(100.90192603,213.07179382)
\lineto(99.23205269,224.03589555)
\closepath
}
}
{
\pscustom[linewidth=1,linecolor=lightgray]
{
\newpath
\moveto(220,70.24166002)
\lineto(180,85.24166002)
\lineto(220,100.24166002)
}
}
{
\pscustom[linewidth=1,linecolor=lightgray,fillstyle=solid,fillcolor=white]%arrow
{
\newpath
\moveto(193.82053017,80.03599737)
\lineto(203.28467372,82.35645373)
\lineto(200.56217611,73.0000368)
\lineto(193.82053017,80.03599737)
\closepath
}
}
{
\pscustom[linewidth=1,linecolor=black,linestyle=dashed,dash=4 4]%dotted line
{
\newpath
\moveto(179.56218,85.00000002)
\lineto(219.56218,100.00000002)
}
}
{
\pscustom[linewidth=1,linecolor=lightgray,fillstyle=solid,fillcolor=white]%arrow
{
\newpath
\moveto(209.56217995,95.99999704)
\lineto(202.31090155,86.9378007)
\lineto(198.0884485,97.74869017)
\lineto(209.56217995,95.99999704)
\closepath
}
}
{
\pscustom[linewidth=1,linecolor=lightgray]
{
\newpath
\moveto(219.56218,170.00000002)
\lineto(179.56218,155.00000002)
\lineto(219.56218,140.00000002)
}
}
{
\pscustom[linewidth=1,linecolor=black,linestyle=dashed,dash=4 4]%dotted line
{
\newpath
\moveto(179.56218,155.00000002)
\lineto(219.56218,170.00000002)
}
}
{
\pscustom[linewidth=1,linecolor=lightgray,fillstyle=solid,fillcolor=white]%arrow
{
\newpath
\moveto(206.56189001,164.99999736)
\lineto(201.10760025,157.12659658)
\lineto(197.01618004,165.78685046)
\lineto(206.56189001,164.99999736)
\closepath
}
}
{
\pscustom[linewidth=1,linecolor=lightgray,fillstyle=solid,fillcolor=white]%arrow
{
\newpath
\moveto(199.1078899,147.87339798)
\lineto(208.65359985,148.66025143)
\lineto(204.56217995,139.99999739)
\lineto(199.1078899,147.87339798)
\closepath
}
}
{
\pscustom[linewidth=1,linecolor=gray,fillstyle=solid,fillcolor=white]%polygon
{
\newpath
\moveto(119.55915296,50.78996302)
\lineto(59.55154339,85.52656648)
\lineto(59.55154153,154.99977662)
\lineto(119.55914924,189.7363833)
\lineto(179.56675881,154.99977984)
\lineto(179.56676067,85.5265697)
\lineto(119.55915296,50.78996302)
\closepath
}
}
{
\pscustom[linewidth=3,linecolor=black]%T
{
\newpath
\moveto(149.56218,240.00000002)
\lineto(119.56218,190.00000002)
\lineto(179.56218,155.00000002)
\lineto(179.56218,85.00000002)
\lineto(119.56218,49.99998002)
\lineto(59.56218,85.00000002)
\lineto(59.56218,155.00000002)
}
}
{
\pscustom[linewidth=2,linecolor=black,fillstyle=solid,fillcolor=white]%arrow
{
\newpath
\moveto(154.56217979,169.99999909)
\lineto(139.80112514,169.42612008)
\lineto(146.69380007,182.73604508)
\lineto(154.56217979,169.99999909)
\closepath
}
}
{
\pscustom[linewidth=2,linecolor=black,fillstyle=solid,fillcolor=white]%arrow
{
\newpath
\moveto(134.56218,232.36217957)
\lineto(148.43801679,225.64903333)
\lineto(135.68634321,216.98877929)
\lineto(134.56218,232.36217957)
\closepath
}
}
{
\pscustom[linewidth=2,linecolor=black,fillstyle=solid,fillcolor=white]%arrow
{
\newpath
\moveto(179.5621802,107.36218002)
\lineto(172.94763601,118.81890591)
\lineto(186.17672378,118.81890627)
\lineto(179.5621802,107.36218002)
\closepath
}
}
{
\pscustom[linewidth=2,linecolor=black,fillstyle=solid,fillcolor=white]%arrow
{
\newpath
\moveto(144.56218021,64.99999909)
\lineto(159.32323486,64.42612008)
\lineto(152.43055993,77.73604508)
\lineto(144.56218021,64.99999909)
\closepath
}
}
{
\pscustom[linewidth=2,linecolor=black,fillstyle=solid,fillcolor=white]%arrow
{
\newpath
\moveto(81.69380021,72.73601095)
\lineto(96.45485486,73.30988996)
\lineto(89.56217993,59.99996495)
\lineto(81.69380021,72.73601095)
\closepath
}
}
{
\pscustom[linewidth=2,linecolor=black,fillstyle=solid,fillcolor=white]%arrow
{
\newpath
\moveto(59.5621802,119.99998002)
\lineto(52.94763601,108.54325412)
\lineto(66.17672378,108.54325377)
\lineto(59.5621802,119.99998002)
\closepath
}
}
{
\pscustom[linewidth=2,linecolor=gray,fillstyle=solid,fillcolor=white]%arrow
{
\newpath
\moveto(94.56217979,175.00001095)
\lineto(79.80112514,175.57388996)
\lineto(86.69380007,162.26396495)
\lineto(94.56217979,175.00001095)
\closepath
}
}
{
\put(115,115){$C$}
\put(150,210){$e_C$}
\put(60,185){$f_C$}
\put(150,45){$B_v(\mathcal{T})$}
}
\end{pspicture}
\caption{\label{cgraphpic4}}
\end{figure}

Then $A(B_v(\mathcal{T}))=\mathcal{T}$. Thus, since $\mathcal{T}$ contains no circuits, $B_v(\mathcal{T})$ contains no circuits.
Let $w\in\V(\mathcal{G})$. Since $\mathcal{T}$ contains a directed path from $A(v)$ to $A(w)$, it is clear that $B_v(\mathcal{T})$ contains a directed path from $v$ to $w$. 
Hence $B_v(\mathcal{T})\in\Tr(\mathcal{G},v)$.

We can now see that $B_v(\mathcal{T})\in\HTr(\mathcal{G},v)$. Thus if $\mathcal{T}'\in\HTr(\mathcal{G},v)$ and $C$ is a Seifert circle in $\mathcal{H}$ then $\mathcal{T}'$ contains all but one edge of $C$, as in Figure \ref{cgraphpic4}. Therefore, at most one edge of $\mathcal{T}'\cap\mathcal{K}$ has its terminal vertex on $C$, with no such edge if $v$ lies on $C$. This means that $A(\mathcal{T}')$ does not contain any circuits, and no edge of $A(\mathcal{T}')$ has terminal vertex $A(v)$. It is now clear that we can define $A_v$ by $A_v(\mathcal{T})=A(\mathcal{T})$. Knowing this, we see that $A_v$ and $B_v$ are mutual inverses.
\end{proof}

As before, we can construct a product sutured manifold $\nmfld{H}{\mathcal{G}}$ from a digraph $(\mathcal{G},\mathcal{O})$.

\begin{definition}
Let $\mathcal{G}'$ be a planar digraph in which incoming and outgoing edges alternate at each vertex.
The boundary of any region $r$ of $\sphere\setminus\mathcal{G}'$ is a cycle.
Define a \textit{$\mathcal{K}$--circle} of $\mathcal{G}'$ to be any such cycle that is oriented clockwise around $r$.
\end{definition}

\begin{construction}
$N=\nmfld{H}{\mathcal{G}}$ has a 0--handle at each vertex of $\mathcal{G}$, and a 1--handle running along each edge of $\mathcal{G}$.

Attach a 2--handle $\disc ^2\times\intvl$ for each $\mathcal{K}$--circle of $\mathcal{G}$. If $r$ is a region of $\sphere\setminus\mathcal{G}$ whose boundary $\partial r$ is a $\mathcal{K}$--circle, the boundary of the union of the 0--handles and the 1--handles of $N$ meets $r$ in a simple closed curve. Attach the 2--handle along this curve.

For a 1--handle $V_1$, let $V_1\cap s=V_1\cap\partial N\cap\sphere$. This section of $s$ is oriented in the same direction as the edge of $\mathcal{G}$ that $V_1$ runs along.

Let $V_0$ be a 0--handle. Then $\partial V_0\cap\sphere$ consists of adjacent simple arcs $\rho^1_1,\rho^2_1,\cdots,\rho^1_n,\rho^2_n$ for some $n\in\mathbb{N}$, ordered clockwise around $V_0$, where $\rho^1_i$ is properly embedded in $N$ and $\rho^2_i\subset \partial N$ for each $i$. For $1\leq m\leq n$, join the midpoint $\rho^2_m(\frac{1}{2})$  of $\rho^2_m$ to the far endpoint of $\rho^1_m$ by a simple arc running over $V_0$, and to the far endpoint of $\rho^1_{m+1}$ (where $\rho^1_{n+1}=\rho^1_1$) by a simple arc running under $V_0$, as shown in Figure \ref{cgraphpic5}.
\begin{figure}[htbp]
\centering
%LaTeX with PSTricks extensions
%%Creator: 0.46
%%Please note this file requires PSTricks extensions
\psset{xunit=.45pt,yunit=.45pt,runit=.45pt}
\begin{pspicture}(380,215)
{
\newgray{lightgrey}{.8}
\newgray{lightishgrey}{.7}
\newgray{grey}{.6}
\newgray{midgrey}{.4}
\newgray{darkgrey}{.3}
}
{
\pscustom[linestyle=none,fillstyle=solid,fillcolor=lightgrey]%2-handle
{
\newpath
\moveto(392.63008,158.96379262)
\lineto(262.63008,138.96379262)
\lineto(242.63008,88.96379262)
\lineto(372.63008,18.96378262)
\lineto(392.63008,158.96379262)
\closepath
}
}
{
\pscustom[linestyle=none,fillstyle=solid,fillcolor=lightgrey]%2-handle
{
\newpath
\moveto(5.3781741,155.74827262)
\lineto(135.37817,135.74827262)
\lineto(155.37817,85.74827262)
\lineto(25.378174,15.74828262)
\lineto(5.3781741,155.74827262)
\closepath
}
}
{
\pscustom[linewidth=1.5,linecolor=darkgrey,fillstyle=solid,fillcolor=lightishgrey]%1-handle
{
\newpath
\moveto(8.81618944,165.24605509)
\lineto(137.63011961,142.53269496)
\lineto(134.19212682,123.03485933)
\lineto(5.37819666,145.74821947)
\lineto(8.81618944,165.24605509)
\closepath
}
}
{
\pscustom[linewidth=1.5,linecolor=darkgrey,fillstyle=solid,fillcolor=lightishgrey]%1-handle
{
\newpath
\moveto(25.37817982,25.74810151)
\lineto(138.65524301,91.1486578)
\lineto(148.55455415,74.00255146)
\lineto(35.27749096,8.60199516)
\lineto(25.37817982,25.74810151)
\closepath
}
}
{
\pscustom[linewidth=1.5,linecolor=darkgrey,fillstyle=solid,fillcolor=lightishgrey]%1-handle
{
\newpath
\moveto(389.19216511,168.46151793)
\lineto(260.37823495,145.7481578)
\lineto(263.81622774,126.25032217)
\lineto(392.6301579,148.96368231)
\lineto(389.19216511,168.46151793)
\closepath
}
}
{
\pscustom[linewidth=1.5,linecolor=darkgrey,fillstyle=solid,fillcolor=lightishgrey]%1-handle
{
\newpath
\moveto(367.63013108,23.96358538)
\lineto(254.35306789,89.36414167)
\lineto(244.45375675,72.21803532)
\lineto(357.73081994,6.81747903)
\lineto(367.63013108,23.96358538)
\closepath
}
}
{
\pscustom[linewidth=1.5,linecolor=darkgrey,fillstyle=solid,fillcolor=grey]%0-handle
{
\newpath
\moveto(254.97935221,189.11210628)
\curveto(284.4361926,156.21225591)(281.64186565,105.60392718)(248.74201528,76.14708679)
\curveto(215.84216491,46.6902464)(165.23383618,49.48457335)(135.77699579,82.38442372)
\curveto(107.89017567,113.53074482)(108.51193452,160.28445243)(137.21715111,190.67817868)
}
}
{
\pscustom[linewidth=3,linecolor=black]
{
\newpath
\moveto(115.37817,145.74827262)
\lineto(35.378174,160.74827262)
}
}
{
\pscustom[linewidth=3,linecolor=black]
{
\newpath
\moveto(274.37817,147.74827262)
\lineto(365.37817,165.74827262)
}
}
{
\pscustom[linewidth=3,linecolor=darkgrey]
{
\newpath
\moveto(30.378174,5.74828262)
\lineto(146.37817,72.74827262)
}
}
{
\pscustom[linewidth=3,linecolor=darkgrey]
{
\newpath
\moveto(360.37817,5.74828262)
\lineto(244.37817,72.74827262)
}
}
{
\pscustom[linewidth=3,linecolor=darkgrey]
{
\newpath
\moveto(231.76498561,186.35206329)
\curveto(241.4969918,182.34844956)(250.25992705,180.19885917)(260.73989987,179.2443767)
}
}
{
\pscustom[linewidth=3,linecolor=darkgrey,linestyle=dashed,dash=12 12]
{
\newpath
\moveto(158.37817439,186.74827329)
\curveto(148.6461682,182.74465956)(139.88323295,180.59506917)(129.40326013,179.6405867)
}
}
{
\pscustom[linewidth=0.5,linecolor=white,fillstyle=solid,fillcolor=white]
{
\newpath
\moveto(34.99196243,194.35137939)
\lineto(375.87086868,194.35137939)
\lineto(375.87086868,185.3070097)
\lineto(34.99196243,185.3070097)
\lineto(34.99196243,194.35137939)
\closepath
}
}
{
\pscustom[linewidth=0.5,linecolor=white,fillstyle=solid,fillcolor=white]
{
\newpath
\moveto(0,209.56793213)
\lineto(35.41101858,209.56793213)
\lineto(35.41101858,3.97117615)
\lineto(0,3.97117615)
\lineto(0,209.56793213)
\closepath
}
}
{
\pscustom[linewidth=0.5,linecolor=white,fillstyle=solid,fillcolor=white]
{
\newpath
\moveto(356.98187256,206.24230957)
\lineto(397.96148682,206.24230957)
\lineto(397.96148682,0.72357178)
\lineto(356.98187256,0.72357178)
\lineto(356.98187256,206.24230957)
\closepath
}
}
{
\pscustom[linewidth=3,linecolor=darkgrey,linestyle=dashed,dash=12 12]
{
\newpath
\moveto(262.38476102,179.99999991)
\curveto(235.03558194,149.78638877)(228.61108165,108.89264703)(245.37816684,71.74826995)
}
}
{
\pscustom[linewidth=3,linecolor=black]
{
\newpath
\moveto(197.92716589,59.12866185)
\curveto(201.91154214,104.66754068)(234.42611758,140.44295289)(279.37816597,148.74825747)
}
}
{
\pscustom[linewidth=3,linecolor=darkgrey]
{
\newpath
\moveto(128.37157898,179.99999991)
\curveto(155.72075806,149.78638877)(162.14525835,108.89264703)(145.37817316,71.74826995)
}
}
{
\pscustom[linewidth=3,linecolor=black,linestyle=dashed,dash=12 12]
{
\newpath
\moveto(196.89579805,55.33038217)
\curveto(193.25404932,101.18033572)(160.60704015,137.39180826)(115.37817403,145.74825747)
}
}
{
\pscustom[linewidth=2.5,linecolor=black,linestyle=dashed,dash=3 4]
{
\newpath
\moveto(274.4447779,147.9298838)
\curveto(278.97710569,118.51308614)(268.16186613,91.1567686)(244.73919612,72.79189699)
}
}
{
\pscustom[linewidth=2.5,linecolor=black,linestyle=dashed,dash=3 4]
{
\newpath
\moveto(145.74215104,73.00779698)
\curveto(122.93455187,91.05160323)(112.22812803,117.52946116)(116.08402598,146.35475518)
}
}
{
\pscustom[linewidth=2.5,linecolor=black,linestyle=dashed,dash=5 4]
{
\newpath
\moveto(244.96168516,72.96697923)
\curveto(215.50418721,49.70161065)(175.40308832,49.66624501)(145.90460021,72.87961959)
}
}
{
\pscustom[linewidth=2,linecolor=black,fillstyle=solid,fillcolor=black]
{
\newpath
\moveto(199.37158,55)
\curveto(199.37158,53.62)(198.25158,52.5)(196.87158,52.5)
\curveto(195.49158,52.5)(194.37158,53.62)(194.37158,55)
\curveto(194.37158,56.38)(195.49158,57.5)(196.87158,57.5)
\curveto(198.25158,57.5)(199.37158,56.38)(199.37158,55)
\closepath
}
}
{
\pscustom[linewidth=2,linecolor=black,fillstyle=solid,fillcolor=black]
{
\newpath
\moveto(118.37158,146)
\curveto(118.37158,144.62)(117.25158,143.5)(115.87158,143.5)
\curveto(114.49158,143.5)(113.37158,144.62)(113.37158,146)
\curveto(113.37158,147.38)(114.49158,148.5)(115.87158,148.5)
\curveto(117.25158,148.5)(118.37158,147.38)(118.37158,146)
\closepath
}
}
{
\pscustom[linewidth=2,linecolor=black,fillstyle=solid,fillcolor=black]
{
\newpath
\moveto(276.87158,147.5)
\curveto(276.87158,146.12)(275.75158,145)(274.37158,145)
\curveto(272.99158,145)(271.87158,146.12)(271.87158,147.5)
\curveto(271.87158,148.88)(272.99158,150)(274.37158,150)
\curveto(275.75158,150)(276.87158,148.88)(276.87158,147.5)
\closepath
}
}
{
\pscustom[linewidth=2,linecolor=black,fillstyle=solid,fillcolor=black]
{
\newpath
\moveto(147.87158,72.5)
\curveto(147.87158,71.12)(146.75158,70)(145.37158,70)
\curveto(143.99158,70)(142.87158,71.12)(142.87158,72.5)
\curveto(142.87158,73.88)(143.99158,75)(145.37158,75)
\curveto(146.75158,75)(147.87158,73.88)(147.87158,72.5)
\closepath
}
}
{
\pscustom[linewidth=2,linecolor=black,fillstyle=solid,fillcolor=black]
{
\newpath
\moveto(247.37158,73)
\curveto(247.37158,71.62)(246.25158,70.5)(244.87158,70.5)
\curveto(243.49158,70.5)(242.37158,71.62)(242.37158,73)
\curveto(242.37158,74.38)(243.49158,75.5)(244.87158,75.5)
\curveto(246.25158,75.5)(247.37158,74.38)(247.37158,73)
\closepath
}
}
{
\put(185,35){\scriptsize$\rho^2_m(\frac{1}{2})$}
\put(150,45){\scriptsize$\rho^2_m$}
\put(75,100){\scriptsize$\rho^1_{m+1}$}
\put(275,100){\scriptsize$\rho^1_m$}
}
\end{pspicture}
%LaTeX with PSTricks extensions
%%Creator: 0.46
%%Please note this file requires PSTricks extensions
\psset{xunit=.45pt,yunit=.45pt,runit=.45pt}
\begin{pspicture}(380,215)
{
\newgray{lightgrey}{.8}
\newgray{lightishgrey}{.7}
\newgray{grey}{.6}
\newgray{midgrey}{.4}
\newgray{darkgrey}{.3}
}
{
\pscustom[linestyle=none,fillstyle=solid,fillcolor=lightgrey]%2-handle
{
\newpath
\moveto(394.2585,158.96379262)
\lineto(264.2585,138.96379262)
\lineto(244.2585,88.96379262)
\lineto(374.2585,18.96378262)
\lineto(394.2585,158.96379262)
\closepath
}
}
{
\pscustom[linestyle=none,fillstyle=solid,fillcolor=lightgrey]%2-handle
{
\newpath
\moveto(7.006594,155.74827262)
\lineto(137.00659,135.74827262)
\lineto(157.00659,85.74827262)
\lineto(27.006594,15.74828262)
\lineto(7.006594,155.74827262)
\closepath
}
}
{
\pscustom[linewidth=1.5,linecolor=darkgrey,fillstyle=solid,fillcolor=lightishgrey]%1-handle
{
\newpath
\moveto(10.4445979,165.24606105)
\lineto(139.25852806,142.53270091)
\lineto(135.82053528,123.03486529)
\lineto(7.00660512,145.74822542)
\lineto(10.4445979,165.24606105)
\closepath
}
}
{
\pscustom[linewidth=1.5,linecolor=darkgrey,fillstyle=solid,fillcolor=lightishgrey]%1-handle
{
\newpath
\moveto(27.00659072,25.74816801)
\lineto(140.28365391,91.1487243)
\lineto(150.18296505,74.00261795)
\lineto(36.90590186,8.60206166)
\lineto(27.00659072,25.74816801)
\closepath
}
}
{
\pscustom[linewidth=1.5,linecolor=darkgrey,fillstyle=solid,fillcolor=lightishgrey]%1-handle
{
\newpath
\moveto(390.82056297,168.46157209)
\lineto(262.00663281,145.74821195)
\lineto(265.4446256,126.25037633)
\lineto(394.25855576,148.96373646)
\lineto(390.82056297,168.46157209)
\closepath
}
}
{
\pscustom[linewidth=1.5,linecolor=darkgrey,fillstyle=solid,fillcolor=lightishgrey]%1-handle
{
\newpath
\moveto(369.25855016,23.96365511)
\lineto(255.98148697,89.36421141)
\lineto(246.08217583,72.21810506)
\lineto(359.35923902,6.81754877)
\lineto(369.25855016,23.96365511)
\closepath
}
}
{
\pscustom[linewidth=1.5,linecolor=darkgrey,fillstyle=solid,fillcolor=grey]%0-handle
{
\newpath
\moveto(251.23727099,194.56181508)
\curveto(283.70235064,164.62648138)(285.75547778,113.98266765)(255.82014408,81.51758801)
\curveto(225.88481038,49.05250836)(175.24099665,46.99938122)(142.77591701,76.93471492)
\curveto(110.31083736,106.87004862)(108.25771022,157.51386235)(138.19304392,189.97894199)
\curveto(140.09825703,192.04515899)(141.7973177,193.71234784)(143.89923835,195.57809682)
}
}
{
\pscustom[linewidth=3,linecolor=black]
{
\newpath
\moveto(117.00659,145.74827262)
\lineto(37.006594,160.74827262)
}
}
{
\pscustom[linewidth=3,linecolor=black]
{
\newpath
\moveto(276.00659,147.74827262)
\lineto(367.00659,165.74827262)
}
}
{
\pscustom[linewidth=3,linecolor=black]
{
\newpath
\moveto(32.006594,5.74828262)
\lineto(148.00659,72.74827262)
}
}
{
\pscustom[linewidth=3,linecolor=black]
{
\newpath
\moveto(362.00659,5.74828262)
\lineto(246.00659,72.74827262)
}
}
{
\pscustom[linewidth=3,linecolor=black]
{
\newpath
\moveto(233.39340561,186.35206329)
\curveto(243.1254118,182.34844956)(251.88834705,180.19885917)(262.36831987,179.2443767)
}
}
{
\pscustom[linewidth=3,linecolor=black,linestyle=dashed,dash=12 12]
{
\newpath
\moveto(160.00659439,186.74827329)
\curveto(150.2745882,182.74465956)(141.51165295,180.59506917)(131.03168013,179.6405867)
}
}
{
\pscustom[linewidth=0.5,linecolor=white,fillstyle=solid,fillcolor=white]
{
\newpath
\moveto(36.69642258,199.43835449)
\lineto(377.42325974,199.43835449)
\lineto(377.42325974,185.3830471)
\lineto(36.69642258,185.3830471)
\lineto(36.69642258,199.43835449)
\closepath
}
}
{
\pscustom[linewidth=0.5,linecolor=white,fillstyle=solid,fillcolor=white]
{
\newpath
\moveto(0.4785507,209.56060791)
\lineto(36.52145079,209.56060791)
\lineto(36.52145079,3.97857666)
\lineto(0.4785507,3.97857666)
\lineto(0.4785507,209.56060791)
\closepath
}
}
{
\pscustom[linewidth=0.5,linecolor=white,fillstyle=solid,fillcolor=white]
{
\newpath
\moveto(356.98187256,206.24230957)
\lineto(397.961483,206.24230957)
\lineto(397.961483,0.72355652)
\lineto(356.98187256,0.72355652)
\lineto(356.98187256,206.24230957)
\closepath
}
}
{
\pscustom[linewidth=3,linecolor=black,linestyle=dashed,dash=12 12]
{
\newpath
\moveto(198.52421805,55.33038217)
\curveto(194.88246932,101.18033572)(162.23546015,137.39180826)(117.00659403,145.74825747)
}
}
{
\pscustom[linewidth=3,linecolor=black]
{
\newpath
\moveto(199.55558589,59.12866185)
\curveto(203.53996214,104.66754068)(236.05453758,140.44295289)(281.00658597,148.74825747)
}
}
{
\pscustom[linewidth=3,linecolor=black,linestyle=dashed,dash=12 12]
{
\newpath
\moveto(264.01318102,179.99999991)
\curveto(236.66400194,149.78638877)(230.23950165,108.89264703)(247.00658684,71.74826995)
}
}
{
\pscustom[linewidth=3,linecolor=black]
{
\newpath
\moveto(129.99999898,179.99999991)
\curveto(157.34917806,149.78638877)(163.77367835,108.89264703)(147.00659316,71.74826995)
}
}
{
\pscustom[linewidth=2,linecolor=black,fillstyle=solid,fillcolor=black]%arrowhead
{
\newpath
\moveto(74.38838649,153.92232532)
\lineto(64.974812,147.646609)
\lineto(94,150.00000262)
\lineto(68.11267016,163.33589981)
\lineto(74.38838649,153.92232532)
\closepath
}
}
{
\pscustom[linewidth=2,linecolor=black,fillstyle=solid,fillcolor=black]%arrowhead
{
\newpath
\moveto(308.27557802,36.92256211)
\lineto(319.21677743,39.80182511)
\lineto(291,46.99998262)
\lineto(311.15484102,25.9813627)
\lineto(308.27557802,36.92256211)
\closepath}
}
{
\pscustom[linewidth=2,linecolor=black,fillstyle=solid,fillcolor=black]%arrowhead
{
\newpath
\moveto(88.36486284,38.92276138)
\lineto(91.34169647,49.83781803)
\lineto(71,28.99998262)
\lineto(99.27991949,35.94592775)
\lineto(88.36486284,38.92276138)
\closepath
}
}
{
\pscustom[linewidth=2,linecolor=black,fillstyle=solid,fillcolor=black]%arrowhead
{
\newpath
\moveto(308.38838649,154.07767991)
\lineto(302.11267016,144.66410543)
\lineto(328,158.00000262)
\lineto(298.974812,160.35339624)
\lineto(308.38838649,154.07767991)
\closepath
}
}
\end{pspicture}
\caption{\label{cgraphpic5}}
\end{figure}

It is clear that $N$ is now a sutured manifold.
\end{construction}

\begin{remark}
We may similarly define $\graph{K}{D}$ and $\nmfld{K}{\mathcal{G}}$.
\end{remark}

%------------------------------

\section{Digraphs constructed from link diagrams}\label{infgraphssec}
\subsection{Bounding valency}
%Master document is alexpolypaper.tex.

In both Theorem \ref{incompthm} and Theorem \ref{finitenessthm}, we wish to use a bound on $\nrap{L}{0}$ to control the possibilities for the homogeneous link $L$. 
In Section \ref{alexpolysec} we have established the following.

\begin{lemma}
Let $L$ be a link with a reduced, special, alternating diagram $D$. Then 
$\nrap{L}{0}=|\Tr(\graph{M}{D},v)|=|\Tr(\graph{H}{D},v)|=|\Tr(\graph{K}{D},v)|$, where in each case $v$ may be any vertex of the relevant digraph.
\end{lemma}

Thus we now turn our attention to controlling a digraph using a bound on the number of directed spanning subtrees it contains.

\begin{lemma}
Let $\mathcal{G}$ be a digraph with no loops,  and fix $v_0 \in \V(\mathcal{G})$. Suppose there is a directed spanning subtree $\mathcal{T}$ of $\mathcal{G}$ with origin $v_0$. Let $w$ be any leaf of $\mathcal{T}$ and let $n$ be the in-degree of $w$ in $\mathcal{G}$. Then $\mathcal{G}$ has at least $n$ directed spanning subtrees with origin $v_0$.
\end{lemma}
\begin{proof}
Let $e$ be the edge of $\mathcal{T}$ with terminal vertex $w$.
By repeated use of Lemma \ref{deletecontractlemma}, $|\Tr(\mathcal{G},v_0)|\geq |\Tr(\mathcal{G}/(\mathcal{T}\setminus e),v_0)|$.
But $\mathcal{G}/(\mathcal{T}\setminus e)$ has two vertices and $w$ still has in-degree $n$. Since $\mathcal{G}$ contained no loops, $\mathcal{G}/(\mathcal{T}\setminus e)$ has no loops with endpoints at $w$. Thus $|\Tr(\mathcal{G}/(\mathcal{T}\setminus e),v_0)|\geq n$.
\end{proof}

\begin{definition}
A vertex $v\in\V(\mathcal{G})$ of a graph $\mathcal{G}$ is a \textit{cut vertex} if $\mathcal{G}[\V(\mathcal{G})\setminus\{v\}]$ is disconnected.
\end{definition}

\begin{definition}\label{primedefn}
Define a planar digraph $\mathcal{G}$ to be \textit{prime} if none of the following hold.
\begin{itemize}
  \item $\mathcal{G}$ contains a loop.
  \item $\mathcal{G}$ has a cut vertex.
  \item There is a simple closed curve $\rho$ in $\sphere$ disjoint from the vertices of $\mathcal{G}$ meeting the edges of $\mathcal{G}$ at exactly two points and with at least one vertex of $\mathcal{G}$ on each side of $\rho$.
\end{itemize}
\end{definition}

\begin{proposition}\label{leaforiginprop}
Let $\mathcal{G}$ be a prime, planar digraph such that, at every vertex, incoming and outgoing edges alternate. Let $w\in\V(\mathcal{G})$. Then there is a vertex $v_w\in\V(\mathcal{G})\setminus\{w\}$ and a directed spanning subtree $\mathcal{T}$ of $\mathcal{G}$ with origin $v_w$ such that $w$ is a leaf of $\mathcal{T}$.
\end{proposition}
\begin{proof}
Assume otherwise. 
For $v\in\V(\mathcal{G})\setminus\{w\}$, let $A(v)$ be the set of vertices $v'\in\V(\mathcal{G})\setminus\{w\}$ such that there is a directed path in $\mathcal{G}$ from $v$ to $v'$ that does not pass through $w$,
and let $B(v)=\V(\mathcal{G})\setminus(\{w\}\cup A(v))$.

Suppose that $B(v_0)=\emptyset$ for some $v_0\in\V(\mathcal{G})\setminus\{w\}$. Then for every $v\in\V(\mathcal{G})\setminus\{w\}$ there is a directed path $\rho(v)$ in $\mathcal{G}$ from $v_0$ to $v$ that does not pass through $w$. Take the union of these paths, and discard edges as necessary to give a tree $\mathcal{T}'$ that includes all vertices in $\V(\mathcal{G})\setminus\{w\}$. Pick any edge $e$ with terminal vertex $w$. Then $\mathcal{T}=\mathcal{T}'\cup e$ is a directed spanning subtree of $\mathcal{G}$ with origin $v_0$ of which $w$ is a leaf. This contradicts the assumption that no such $\mathcal{T}$ exists. Thus, for every $v\in\V(\mathcal{G})\setminus\{w\}$, the set $B(v)$ is non-empty, as is $A(v)$.

Choose $v_0\neq w$. By the definition of $B(v_0)$, no edge of $\mathcal{G}$ runs from a vertex in $A(v_0)$ to a vertex in $B(v_0)$. Note that $\mathcal{G}[A(v_0)]$ is connected. On the other hand, $\mathcal{G}[B(v_0)]$ may be disconnected, but every vertex of $B(v_0)$ lies on a path that begins at $w$ and does not meet $A(v_0)$. Thus $\mathcal{G}$ has the form shown in Figure \ref{finitevalencepic1}, where each arrow in the picture may denote multiple edges.
\begin{figure}[htbp]
\centering
%LaTeX with PSTricks extensions
%%Creator: inkscape 0.46
%%Please note this file requires PSTricks extensions

\psset{xunit=.35pt,yunit=.35pt,runit=.35pt}
\begin{pspicture}(390,170)

{
\pscustom[linewidth=1,linecolor=black,fillstyle=solid,fillcolor=white]%left box
{
\newpath
\moveto(0.53815818,117.13598633)
\lineto(119.46184134,117.13598633)
\lineto(119.46184134,28.21230316)
\lineto(0.53815818,28.21230316)
\lineto(0.53815818,117.13598633)
\closepath
}
}
{
\pscustom[linewidth=1,linecolor=black,fillstyle=solid,fillcolor=white]%right box
{
\newpath
\moveto(280.53814697,117.13598633)
\lineto(399.46183014,117.13598633)
\lineto(399.46183014,28.21230316)
\lineto(280.53814697,28.21230316)
\lineto(280.53814697,117.13598633)
\closepath
}
}
{
\pscustom[linewidth=1,linecolor=black,fillstyle=solid,fillcolor=black]%circle
{
\newpath
\moveto(208.999999,87.67415046)
\curveto(208.999999,82.69566053)(204.95948543,78.65514696)(199.9809955,78.65514696)
\curveto(195.00250557,78.65514696)(190.961992,82.69566053)(190.961992,87.67415046)
\curveto(190.961992,92.65264039)(195.00250557,96.69315396)(199.9809955,96.69315396)
\curveto(204.95948543,96.69315396)(208.999999,92.65264039)(208.999999,87.67415046)
\closepath
}
}
{
\pscustom[linewidth=1.02066743,linecolor=black]%curve
{
\newpath
\moveto(279.48966,52.16378262)
\curveto(200,20.71028262)(120.51034,52.16378262)(120.51034,52.16378262)
}
}
{
\pscustom[linewidth=1,fillstyle=solid,fillcolor=black]%arrowhead
{
\newpath
\moveto(130.00102909,48.40838362)
\lineto(135.29946433,50.70249966)
\lineto(120.51034,52.16378262)
\lineto(132.29514513,43.10994839)
\lineto(130.00102909,48.40838362)
\closepath
}
}
{
\pscustom[linewidth=1,linecolor=black]%curve
{
\newpath
\moveto(120,92.67415262)
\curveto(155,102.67415262)(190,92.67415262)(190,92.67415262)
}
}
{
\pscustom[linewidth=1,fillstyle=solid,fillcolor=black]%arrowhead
{
\newpath
\moveto(180.38476052,95.4213639)
\lineto(175.43978022,92.67415262)
\lineto(190,92.67415262)
\lineto(177.63754924,100.3663442)
\lineto(180.38476052,95.4213639)
\closepath
}
}
{
\pscustom[linewidth=1,linecolor=black]%curve
{
\newpath
\moveto(208,92.67415262)
\curveto(243,102.67415262)(278,92.67415262)(278,92.67415262)
}
}
{
\pscustom[linewidth=1,fillstyle=solid,fillcolor=black]%arrowhead
{
\newpath
\moveto(268.38476052,95.4213639)
\lineto(263.43978022,92.67415262)
\lineto(278,92.67415262)
\lineto(265.63754924,100.3663442)
\lineto(268.38476052,95.4213639)
\closepath
}
}
{
\pscustom[linewidth=1,linecolor=black]%curve
{
\newpath
\moveto(192,82.67415262)
\curveto(157,72.67415262)(122,82.67415262)(122,82.67415262)
}
}
{
\pscustom[linewidth=1,fillstyle=solid,fillcolor=black]%arrowhead
{
\newpath
\moveto(131.61523948,79.92694134)
\lineto(136.56021978,82.67415262)
\lineto(122,82.67415262)
\lineto(134.36245076,74.98196104)
\lineto(131.61523948,79.92694134)
\closepath
}
}
{
\pscustom[linewidth=1,linecolor=black]%curve
{
\newpath
\moveto(281,82.67415262)
\curveto(246,72.67415262)(211,82.67415262)(211,82.67415262)
}
}
{
\pscustom[linewidth=1,fillstyle=solid,fillcolor=black]%arrowhead
{
\newpath
\moveto(220.61523948,79.92694134)
\lineto(225.56021978,82.67415262)
\lineto(211,82.67415262)
\lineto(223.36245076,74.98196104)
\lineto(220.61523948,79.92694134)
\closepath
}
}
{
\put(30,140){$A(v_0)$}%labels
\put(300,140){$B(v_0)$}
\put(190,110){$w$}
\put(190,10){$e$}
}
\end{pspicture}
\caption{\label{finitevalencepic1}}
\end{figure}
Since the boundary of each region of $\sphere \setminus \mathcal{G}$ is a cycle, $e$ denotes at most one edge of $\mathcal{G}$. As $\mathcal{G}$ is prime, $e$ denotes at least one edge. Call this edge $e(v_0)$. Similarly define $e(v)$ for each $v\in\V(\mathcal{G})\setminus\{w\}$.

For $v,v'\in\V(\mathcal{G})\setminus\{w\}$, if $v'\in A(v)$ then $A(v')\subseteq A(v)$. This inclusion of sets gives a partial order on $\V(\mathcal{G})\setminus\{w\}$. 
Choose $v_+$ to be maximal with respect to this ordering. Now let $v_b$ be the initial vertex of $e(v_+)$, and $v_a$ the terminal vertex. 
Then $v_+, v_a\in A(v_+)$ and $v_b\in B(v_+)$. It follows that $A(v_a)\subseteq A(v_+)$, so $v_b \in B(v_a)$. In addition, $v_a\in A(v_a)$, and $v_a\in A(v_b)$, so $A(v_a)\subseteq A(v_b)$. If $v_+\in A(v_a)$ then $A(v_+)\subseteq A(v_b)$. Since $v_b\in A(v_b)\setminus A(v_+)$, this contradicts maximality of $v_+$. Thus $v_+ \in B(v_a)$.
In summary, $A(v_a)\cap B(v_+)=\emptyset$ while $A(v_a)\cap A(v_+)$, $B(v_a)\cap A(v_+)$ and $B(v_a)\cap B(v_+)$ are non-empty. This gives $\mathcal{G}$ the structure shown in Figure \ref{finitevalencepic2}, where each box is non-empty.
\begin{figure}[htbp]
\centering
%LaTeX with PSTricks extensions
%%Creator: inkscape 0.46
%%Please note this file requires PSTricks extensions
\psset{xunit=.35pt,yunit=.35pt,runit=.35pt}
\begin{pspicture}(520,320)

%----------------------------------------------------------------------------------------do I want grey boxes?
%{
%\pscustom[linewidth=1,linecolor=lightgray,fillstyle=solid,fillcolor=lightgray]%graybox
%{
%\newpath
%\moveto(0,280)
%\lineto(510,280)
%\lineto(510,170)
%\lineto(0,170)
%\lineto(0,280)
%\closepath
%}
%}

{
\pscustom[linewidth=1,linecolor=black,fillstyle=solid,fillcolor=white]%box
{
\newpath
\moveto(100,270)
\lineto(220,270)
\lineto(220,180)
\lineto(100,180)
\lineto(100,270)
\closepath
}
}
{
\pscustom[linewidth=1,linecolor=black,fillstyle=solid,fillcolor=white]%box
{
\newpath
\moveto(380,100)
\lineto(500,100)
\lineto(500,10)
\lineto(380,10)
\lineto(380,100)
\closepath
}
}
{
\pscustom[linewidth=1,linecolor=black,fillstyle=solid,fillcolor=white]%box
{
\newpath
\moveto(380,270)
\lineto(500,270)
\lineto(500,180)
\lineto(380,180)
\lineto(380,270)
\closepath
}
}
{
\pscustom[linewidth=1,linecolor=black,fillstyle=solid,fillcolor=black]%circle
{
\newpath
\moveto(309.018999,140.00000046)
\curveto(309.018999,135.02151053)(304.97848543,130.98099696)(299.9999955,130.98099696)
\curveto(295.02150557,130.98099696)(290.980992,135.02151053)(290.980992,140.00000046)
\curveto(290.980992,144.97849039)(295.02150557,149.01900396)(299.9999955,149.01900396)
\curveto(304.97848543,149.01900396)(309.018999,144.97849039)(309.018999,140.00000046)
\closepath
}
}
{
\pscustom[linewidth=1,linecolor=black,linestyle=dashed,dash=4 4]%dashed curve
{
\newpath
\moveto(220.51034,230.51033262)
\curveto(300,261.96383262)(379.48966,230.51033262)(379.48966,230.51033262)
}
}
{
\pscustom[linewidth=1,linecolor=black,fillstyle=solid,fillcolor=black]%arrowhead
{
\newpath
\moveto(369.99587828,234.26695534)
\lineto(364.6957165,231.97209175)
\lineto(379.48966,230.51033262)
\lineto(367.70101468,239.56711712)
\lineto(369.99587828,234.26695534)
\closepath
}
}
{
\pscustom[linewidth=1,linecolor=black,linestyle=dashed,dash=4 4]%dashed curve
{
\newpath
\moveto(379.48966,214.48966262)
\curveto(300,183.03616262)(220.51034,214.48966262)(220.51034,214.48966262)
}
}
{
\pscustom[linewidth=1,linecolor=black,fillstyle=solid,fillcolor=black]%arrowhead
{
\newpath
\moveto(230.00102909,210.73426362)
\lineto(235.29946433,213.02837966)
\lineto(220.51034,214.48966262)
\lineto(232.29514513,205.43582839)
\lineto(230.00102909,210.73426362)
\closepath
}
}
{
\pscustom[linewidth=1,linecolor=black,linestyle=dashed,dash=4 4]%dashed curve
{
\newpath
\moveto(450.58829,179.41170262)
\curveto(478.68881,140.00000262)(450.58829,100.58829262)(450.58829,100.58829262)
}
}
{
\pscustom[linewidth=1,linecolor=black,fillstyle=solid,fillcolor=black]%arrowhead
{
\newpath
\moveto(457.41891148,110.16841686)
\lineto(456.31911038,116.73271515)
\lineto(450.58829,100.58829262)
\lineto(463.98320977,111.26821797)
\lineto(457.41891148,110.16841686)
\closepath
}
}
{
\pscustom[linewidth=1,linecolor=black,linestyle=dashed,dash=4 4]%dashed curve
{
\newpath
\moveto(434.40833,100.59842262)
\curveto(405.32515,139.99999262)(434.40833,179.40158262)(434.40833,179.40158262)
}
}
{
\pscustom[linewidth=1,linecolor=black,fillstyle=solid,fillcolor=black]%arrowhead
{
\newpath
\moveto(427.30075635,169.77231571)
\lineto(428.30943366,163.07757949)
\lineto(434.40833,179.40158262)
\lineto(420.60602013,168.76363841)
\lineto(427.30075635,169.77231571)
\closepath
}
}
{
\pscustom[linewidth=1,linecolor=black,linestyle=dashed,dash=4 4]%dashed curve
{
\newpath
\moveto(165.49678,179.50322262)
\curveto(245.12709,70.41398262)(379.50323,60.49678262)(379.50323,60.49678262)
}
}
{
\pscustom[linewidth=1,linecolor=black,fillstyle=solid,fillcolor=black]%arrowhead
{
\newpath
\moveto(369.59469802,61.22805004)
\lineto(365.33877826,57.55714421)
\lineto(379.50323,60.49678262)
\lineto(365.92379219,65.4839698)
\lineto(369.59469802,61.22805004)
\closepath
}
}
{
\pscustom[linewidth=1,linecolor=black]%curve
{
\newpath
\moveto(140.50059,179.49940262)
\curveto(230.12515,48.64618262)(379.49941,38.58054262)(379.49941,38.58054262)
}
}
{
\pscustom[linewidth=1,linecolor=black,fillstyle=solid,fillcolor=black]%arrowhead
{
\newpath
\moveto(146.15811572,171.23932906)
\lineto(151.72515543,170.19830993)
\lineto(140.50059,179.49940262)
\lineto(145.11709659,165.67228936)
\lineto(146.15811572,171.23932906)
\closepath
}
}
{
\pscustom[linewidth=1,linecolor=black]%curve
{
\newpath
\moveto(220.49059,189.50941262)
\curveto(269.83647,179.75471262)(294.50941,150.49059262)(294.50941,150.49059262)
}
}
{
\pscustom[linewidth=1,linecolor=black,fillstyle=solid,fillcolor=black]%arrowhead
{
\newpath
\moveto(288.18487646,157.99200543)
\lineto(282.65449792,158.46275715)
\lineto(294.50941,150.49059262)
\lineto(288.65562818,163.52238397)
\lineto(288.18487646,157.99200543)
\closepath
}
}
{
\pscustom[linewidth=1,linecolor=black]%curve
{
\newpath
\moveto(379.50941,90.49059262)
\curveto(330.16353,100.24529262)(305.49059,129.50941262)(305.49059,129.50941262)
}
}
{
\pscustom[linewidth=1,linecolor=black,fillstyle=solid,fillcolor=black]%arrowhead
{
\newpath
\moveto(311.81512354,122.0079998)
\lineto(317.34550208,121.53724809)
\lineto(305.49059,129.50941262)
\lineto(311.34437182,116.47762126)
\lineto(311.81512354,122.0079998)
\closepath
}
}
{
\pscustom[linewidth=1,linecolor=black]%curve
{
\newpath
\moveto(210.49079,179.50921262)
\curveto(245.06135,140.49079262)(289.50921,140.49079262)(289.50921,140.49079262)
}
}
{
\pscustom[linewidth=1,linecolor=black,fillstyle=solid,fillcolor=black]%arrowhead
{
\newpath
\moveto(217.00018773,172.16231356)
\lineto(222.54270644,171.82731303)
\lineto(210.49079,179.50921262)
\lineto(216.6651872,166.61979485)
\lineto(217.00018773,172.16231356)
\closepath
}
}
{
\pscustom[linewidth=1,linecolor=black]%curve
{
\newpath
\moveto(389.50921,100.49079262)
\curveto(354.93865,139.50921262)(310.49079,139.50921262)(310.49079,139.50921262)
}
}
{
\pscustom[linewidth=1,linecolor=black,fillstyle=solid,fillcolor=black]%arrowhead
{
\newpath
\moveto(382.99981227,107.83769167)
\lineto(377.45729356,108.1726922)
\lineto(389.50921,100.49079262)
\lineto(383.3348128,113.38021038)
\lineto(382.99981227,107.83769167)
\closepath
}
}
{
\pscustom[linewidth=1,linecolor=black]%curve
{
\newpath
\moveto(300.48817,150.48817262)
\curveto(325,180.00000262)(379.51183,189.08846262)(379.51183,189.08846262)
}
}
{
\pscustom[linewidth=1,linecolor=black,fillstyle=solid,fillcolor=black]%arrowhead
{
\newpath
\moveto(369.88140963,187.48283517)
\lineto(366.67149246,182.98841605)
\lineto(379.51183,189.08846262)
\lineto(365.3869905,190.69275234)
\lineto(369.88140963,187.48283517)
\closepath
}
}%%%
{
\pscustom[linewidth=1,linecolor=black]%curve
{
\newpath
\moveto(389.50292,179.52816262)
\curveto(355,145.00000262)(310.49708,144.50025262)(310.49708,144.50025262)
}
}
{
\pscustom[linewidth=1,linecolor=black,fillstyle=solid,fillcolor=black]%arrowhead
{
\newpath
\moveto(320.43808358,144.61188611)
\lineto(324.36983162,148.63294094)
\lineto(310.49708,144.50025262)
\lineto(324.45913842,140.68013808)
\lineto(320.43808358,144.61188611)
\closepath
}
}
\put(270,115){$w$}
\put(120,290){$A(v_a)$}
\put(400,290){$B(v_a)$}
\put(0,220){$A(v_+)$}
\put(0,40){$B(v_+)$}
\end{pspicture}
\caption{\label{finitevalencepic2}}
\end{figure}

Recall that, for any $v\in\V(\mathcal{G})\setminus\{w\}$, no edge starts in $A(v)$ and ends in $B(v)$, while at most one edge starts in $B(v)$ and ends in $A(v)$. Since $e(v_+)$ starts in $B(v_a)\cap B(v_+)$ and ends in $A(v_a)\cap A(v_+)$, we see that all the dashed arrows in Figure \ref{finitevalencepic2} are excluded. This contradicts that $\mathcal{G}$ is prime.
\end{proof}

\begin{corollary}\label{boundedvalencecor}
Let $\mathcal{G}$ and $w$ be as in Proposition \ref{leaforiginprop}. Let $n$ be the in-degree of $w$ in $\mathcal{G}$. Then there exists $v\in\V(\mathcal{G})\setminus\{w\}$ such that $|\Tr(\mathcal{G},v)|\geq n$.
In particular, if $\mathcal{G}\in\{\graph{M}{D},\graph{H}{D},\graph{K}{D}\}$ for some diagram $D$ of a link $L$ then $\nrap{L}{0}\geq n$.
\end{corollary}
\subsection{Digraph properties}
%Master document is alexpolypaper.tex.

\begin{definition}
A diagram $D$ of a link $L$ is \textit{twist-reduced} if, given a simple closed curve $\rho$ in $\sphere$ that passes through two crossings of $D$ transversely and otherwise lies in $\sphere\setminus D$, the section of $L$ on one side of $\rho$ consists of a line of bigons, like that shown in Figure \ref{lgbijectionpic1}.
\begin{figure}[htbp]
\centering
%LaTeX with PSTricks extensions
%%Creator: 0.46
%%Please note this file requires PSTricks extensions
\psset{xunit=.35pt,yunit=.35pt,runit=.35pt}
\begin{pspicture}(700,275)
{
\pscustom[linewidth=2,linecolor=black]
{
\newpath
\moveto(470,60.00000262)
\curveto(470,60.00000262)(470,60.00000262)(475,55.00000262)
\curveto(477.37171,52.62829262)(481.6459,49.99998262)(485,49.99998262)
\curveto(500,49.99998262)(520,49.99998262)(535,49.99998262)
\curveto(538.3541,49.99998262)(542.62829,47.37168262)(545,44.99998262)
\curveto(555,34.99998262)(565,24.99998262)(575,14.99998262)
\curveto(577.37171,12.62828262)(581.6459,9.99998262)(585,9.99998262)
\curveto(610,9.99998262)(645,9.99998262)(670,9.99998262)
\curveto(674.24264,9.99998262)(680,15.75738262)(680,19.99998262)
\curveto(680,90.00000262)(680,195.00000262)(680,270.00000262)
\curveto(680,274.24264262)(674.24264,280.00000262)(670,280.00000262)
\curveto(645,280.00000262)(610,280.00000262)(585,280.00000262)
\curveto(581.6459,280.00000262)(577.37171,277.37171262)(575,275.00000262)
\curveto(570,270.00000262)(570,270.00000262)(570,270.00000262)
}
}
{
\pscustom[linewidth=2,linecolor=black]
{
\newpath
\moveto(450,230.00000262)
\curveto(450,230.00000262)(450,230.00000262)(445,235.00000262)
\curveto(442.62829,237.37171262)(440,241.64590262)(440,245.00000262)
\curveto(440,255.00000262)(440,260.00000262)(440,270.00000262)
\curveto(440,274.24264262)(445.75736,280.00000262)(450,280.00000262)
\curveto(475,280.00000262)(510,280.00000262)(535,280.00000262)
\curveto(538.3541,280.00000262)(542.62829,277.37171262)(545,275.00000262)
\curveto(555,265.00000262)(565,255.00000262)(575,245.00000262)
\curveto(577.37171,242.62829262)(581.6459,240.00000262)(585,240.00000262)
\curveto(600,240.00000262)(615,240.00000262)(630,240.00000262)
\curveto(634.24264,240.00000262)(640,234.24264262)(640,230.00000262)
\curveto(640,180.00000262)(640,110.00000262)(640,60.00000262)
\curveto(640,55.75736262)(634.24264,49.99998262)(630,49.99998262)
\curveto(620,49.99998262)(600,49.99998262)(585,49.99998262)
\curveto(581.6459,49.99998262)(577.37171,47.37168262)(575,44.99998262)
\curveto(570,39.99998262)(570,39.99998262)(570,39.99998262)
}
}
{
\pscustom[linewidth=2,linecolor=black]
{
\newpath
\moveto(450,180.00000262)
\curveto(450,180.00000262)(441.62359,186.08031262)(440,190.00000262)
\curveto(438.85195,192.77164262)(438.65836,197.31672262)(440,200.00000262)
\curveto(445,210.00000262)(455,220.00000262)(475,235.00000262)
\curveto(477.68328,237.01246262)(481.6459,240.00000262)(485,240.00000262)
\curveto(500,240.00000262)(520,240.00000262)(535,240.00000262)
\curveto(538.3541,240.00000262)(542.62829,242.62829262)(545,245.00000262)
\curveto(550,250.00000262)(550,250.00000262)(550,250.00000262)
}
}
{
\pscustom[linewidth=2,linecolor=black]
{
\newpath
\moveto(470,110.00000262)
\curveto(470,110.00000262)(478.37641,103.91969262)(480,100.00000262)
\curveto(481.14805,97.22836262)(481.14805,92.77164262)(480,90.00000262)
\curveto(474.31744,76.28109262)(455,65.00000262)(445,55.00000262)
\curveto(442.62827,52.62827262)(440,48.35408262)(440,44.99998262)
\curveto(440,34.99998262)(440,24.99998262)(440,19.99998262)
\curveto(440,15.75738262)(445.75736,9.99998262)(450,9.99998262)
\curveto(475,9.99998262)(510,9.99998262)(535,9.99998262)
\curveto(538.3541,9.99998262)(542.6283,12.62828262)(545,14.99998262)
\curveto(550,19.99998262)(550,19.99998262)(550,19.99998262)
}
}
{
\pscustom[linewidth=2,linecolor=black]
{
\newpath
\moveto(450,80.00000262)
\curveto(450,80.00000262)(441.62359,86.08031262)(440,90.00000262)
\curveto(438.85195,92.77164262)(438.85195,97.22836262)(440,100.00000262)
\curveto(446.49435,115.67876262)(473.50565,124.32124262)(480,140.00000262)
\curveto(481.14805,142.77164262)(481.14805,147.22836262)(480,150.00000262)
\curveto(478.37641,153.91969262)(470,160.00000262)(470,160.00000262)
}
}
{
\pscustom[linewidth=2,linecolor=black,fillstyle=solid,fillcolor=lightgray]%circle
{
\newpath
\moveto(699.99999,149.99998)
\curveto(699.99999,127.91998)(682.07999,109.99998)(659.99999,109.99998)
\curveto(637.91999,109.99998)(619.99999,127.91998)(619.99999,149.99998)
\curveto(619.99999,172.07998)(637.91999,189.99998)(659.99999,189.99998)
\curveto(682.07999,189.99998)(699.99999,172.07998)(699.99999,149.99998)
\closepath
}
}
{
\pscustom[linewidth=4,linecolor=black]
{
\newpath
\moveto(310,150.00000262)
\lineto(400,150.00000262)
}
}
{
\pscustom[linewidth=4,linecolor=black,fillstyle=solid,fillcolor=black]%arrowhead
{
\newpath
\moveto(360,150.00000262)
\lineto(344,134.00000262)
\lineto(400,150.00000262)
\lineto(344,166.00000262)
\lineto(360,150.00000262)
\closepath
}
}
{
\pscustom[linewidth=2,linecolor=black]
{
\newpath
\moveto(160,39.99998262)
\lineto(165,44.99998262)
\curveto(165,44.99998262)(170,49.99998262)(175,49.99998262)
\curveto(188.5,49.99998262)(210,49.99998262)(220,49.99998262)
\curveto(224.24264,49.99998262)(230,55.75736262)(230,60.00000262)
\curveto(230,65.00000262)(230,225.00000262)(230,230.00000262)
\curveto(230,234.24264262)(224.24264,240.00000262)(220,240.00000262)
\curveto(205,240.00000262)(190,240.00000262)(175,240.00000262)
\curveto(171.6459,240.00000262)(167.37171,242.62829262)(165,245.00000262)
\curveto(155,255.00000262)(145,265.00000262)(135,275.00000262)
\curveto(132.62829,277.37171262)(128.3541,280.00000262)(125,280.00000262)
\curveto(100,280.00000262)(60,280.00000262)(40,280.00000262)
\curveto(35.757359,280.00000262)(30,274.24264262)(30,270.00000262)
\curveto(30,195.00000262)(30,85.00000262)(30,19.99998262)
\curveto(30,15.75738262)(35.757359,9.99998262)(40,9.99998262)
\curveto(65,9.99998262)(100,9.99998262)(125,9.99998262)
\curveto(128.3541,9.99998262)(132.6283,12.62828262)(135,14.99998262)
\curveto(140,19.99998262)(140,19.99998262)(140,19.99998262)
}
}
{
\pscustom[linewidth=2,linecolor=black]
{
\newpath
\moveto(160,270.00000262)
\curveto(160,270.00000262)(160,270.00000262)(165,275.00000262)
\curveto(168.53553,278.53553262)(171.6459,280.00000262)(175,280.00000262)
\curveto(200,280.00000262)(235,280.00000262)(260,280.00000262)
\curveto(264.24264,280.00000262)(270,274.24264262)(270,270.00000262)
\curveto(270,200.00000262)(270,90.00000262)(270,19.99998262)
\curveto(270,15.75738262)(264.24264,9.99998262)(260,9.99998262)
\curveto(235,9.99998262)(200,9.99998262)(175,9.99998262)
\curveto(171.6459,9.99998262)(167.37171,12.62828262)(165,14.99998262)
\curveto(155,24.99998262)(145,34.99998262)(135,44.99998262)
\curveto(132.62829,47.37168262)(128.3541,49.99998262)(125,49.99998262)
\curveto(110,49.99998262)(95,49.99998262)(80,49.99998262)
\curveto(75.757355,49.99998262)(70,55.75736262)(70,60.00000262)
\curveto(70,105.00000262)(70,180.00000262)(70,230.00000262)
\curveto(70,234.24264262)(75.757359,240.00000262)(80,240.00000262)
\curveto(95,240.00000262)(110,240.00000262)(125,240.00000262)
\curveto(128.3541,240.00000262)(132.62829,242.62829262)(135,245.00000262)
\curveto(140,250.00000262)(140,250.00000262)(140,250.00000262)
}
}
{
\pscustom[linewidth=2,linecolor=black,fillstyle=solid,fillcolor=lightgray]%circle
{
\newpath
\moveto(290,150)
\curveto(290,127.92)(272.08,110)(250,110)
\curveto(227.92,110)(210,127.92)(210,150)
\curveto(210,172.08)(227.92,190)(250,190)
\curveto(272.08,190)(290,172.08)(290,150)
\closepath
}
}
{
\pscustom[linewidth=2,linecolor=black,fillstyle=solid,fillcolor=lightgray]%circle
{
\newpath
\moveto(90,150)
\curveto(90,127.92)(72.08,110)(50,110)
\curveto(27.92,110)(10,127.92)(10,150)
\curveto(10,172.08)(27.92,190)(50,190)
\curveto(72.08,190)(90,172.08)(90,150)
\closepath
}
}
{
\pscustom[linewidth=2,linecolor=black]
{
\newpath
\moveto(450,130.00000262)
\curveto(450,130.00000262)(441.62359,136.08031262)(440,140.00000262)
\curveto(438.85195,142.77164262)(438.85195,147.22836262)(440,150.00000262)
\curveto(446.49435,165.67876262)(473.50565,174.32124262)(480,190.00000262)
\curveto(481.14805,192.77164262)(481.14805,197.22836262)(480,200.00000262)
\curveto(478.37641,203.91969262)(470,210.00000262)(470,210.00000262)
}
}
{
\put(645,140){$B$}
\put(235,140){$B$}
\put(35,140){$A$}
}
\end{pspicture}
\caption{\label{lgbijectionpic1}}
\end{figure}

If $D$ is special, we say $D$ is \textit{black-twist-reduced} if this holds whenever $\rho$ is contained in the black regions (see \cite{MR2018964} p215).
\end{definition}

\begin{remark}
Any link diagram can be made twist-reduced by a finite sequence of flypes. Given a special, alternating, twist-reduced diagram $D$, the diagram $D'$ formed by removing white bigons from $D$ (equivalently, removing loops that bound discs from $\graph{M}{D}$) may not be twist-reduced, but will be black-twist-reduced. These changes preserve the property of being special and alternating.
\end{remark}

\begin{definition}
Let $\Lambda$ be the set of link diagrams $D$ with the following properties.
\begin{itemize}
	\item[(L1)] $D$ is connected.
	\item[(L2)] $D$ is special, alternating and positive.
	\item[(L3)] $D$ has no nugatory crossings.
	\item[(L4)] $D$ is prime. 
	\item[(L5)] $D$ is black-twist-reduced.
	\item[(L6)] $D$ has no white bigons.
\end{itemize}
\end{definition}

\begin{definition}
Let $\Gamma$ be the set of digraphs $\mathcal{G}$ with the following properties.
\begin{itemize}
	\item[(G1)] $\mathcal{G}$ is connected.
	\item[(G2)] $\mathcal{G}$ is planar.
	\item[(G3)] $\mathcal{G}$ is prime (see Definition \ref{primedefn}).
	\item[(G4)] At every vertex of $\mathcal{G}$, incoming and outgoing edges alternate.
\end{itemize}
\end{definition}

\begin{remark}
(G3) and (G4) together imply the following.
\begin{itemize}
	\item[(G5)] No vertex of $\mathcal{G}$ has in-degree 1.
\end{itemize}
\end{remark}

\begin{lemma}
$\graphm{M}$ is a bijection from $\Lambda$ to $\Gamma$.
\end{lemma}
\begin{proof}
Let $D\in\Lambda$. 
Recall that $\graph{M}{D}$ consists of a vertex in every white region, and an edge for every crossing, pointing from the undotted side to the dotted side.
By (L1) and (L2), $\graph{M}{D}$ is a well-defined digraph with properties (G1), (G2) and (G4). 
$\graph{M}{D}$ has no loops by (L3), and no cut vertices by (L1) and (L4). By (L5) and (L6), no simple closed curve meets $\E(\graph{M}{D}{})$ twice and separates $\V(\graph{M}{D}{})$. Thus $\graph{M}{D}$ is prime.
Hence $\graph{M}{D}{}\in\Gamma$.

Now let $\mathcal{G}\in\Gamma$. By (G2) and (G4), $\nmfld{M}{\mathcal{G}}$ is defined. Let $D_{\mathcal{G}}$ be the diagram constructed from $\nmfld{M}{\mathcal{G}}$. We choose the black regions of $\sphere\setminus D_{\mathcal{G}}$ to be those that correspond to 0--handles in $\nmfld{M}{\mathcal{G}}$.
It is clear from the construction of $D_{\mathcal{G}}$ that (L1) and (L2) hold. $D_{\mathcal{G}}$ has no nugatory crossings because there are no loops in $\mathcal{G}$, and no obvious decomposition as a connected sum because $\mathcal{G}$ has  no cut vertex.
Suppose either (L5) or (L6) does not hold. This means there is a simple closed curve $\rho\subset\sphere$ that meets $D_{\mathcal{G}}$ at exactly two of the crossings of $D_{\mathcal{G}}$ and is otherwise contained in the black regions. Then $\rho$ meets the edges of $\mathcal{G}$ twice and separates $\V(\mathcal{G})$. This contradicts (G3).
Hence $D_{\mathcal{G}}\in\Lambda$.

It is now easy to check that, when restricted to $\Lambda$ and $\Gamma$ respectively, these two constructions are mutual inverses. 
\end{proof}

\begin{lemma}
$\graphm{H}$ is a bijection from $\Lambda$ to $\Gamma$.
\end{lemma}
\begin{proof}
Let $D\in\Lambda$. Recall that $\graph{H}{D}$ is obtained by taking the underlying graph $\mathcal{G}$ of $D$, with orientations $O$ and $o$, and collapsing all edges of $\mathcal{H}$. Properties (L1) and (L2) ensure $\graph{H}{D}$ is well defined with properties (G1), (G2) and (G4).
Suppose $\graph{H}{D}$ contains a loop $e$, and consider the copy of $e$ in $\mathcal{G}$. By (L3) there are no loops in $\mathcal{G}$, so $e$ must have both its endpoints on a single Seifert circle contained in $\mathcal{H}$. Since $D$ is special, we can use $e$ to construct a simple closed curve $\rho$ that meets $D$ twice at crossings (the endpoints of $e$) and otherwise is contained in the black regions of $D$. This is impossible since (L5) and (L6) hold. Thus $\graph{H}{D}$ contains no loops.
If $\graph{H}{D}$ is not prime, this means there is a simple closed curve $\rho'$ crossing the edges of $\mathcal{G}$ twice and dividing the vertices, contradicting (L4).
Therefore, $\graph{H}{D}\in\Gamma$.

Conversely, let $\mathcal{G}\in\Gamma$. By (G2) and (G4), we can construct $\nmfld{H}{\mathcal{G}}$. Let $D_{\mathcal{G}}$ be the diagram given by $\nmfld{H}{\mathcal{G}}$, where the black regions are those that correspond to a 0--handle or a 2--handle. $D_{\mathcal{G}}$ is connected, by (G1). By construction, (L2) holds and no black region meets itself. (G3) means no white region meets itself. Thus (L3) holds.
Any decomposition of $D_{\mathcal{G}}$ as a connected sum must come either from a similar decomposition of $\mathcal{G}$ or from a cut vertex in $\mathcal{G}$. Neither is possible since $\mathcal{G}$ is prime, so (L4) holds.
Finally, suppose there is a simple closed curve $\rho$ that meets $D_{\mathcal{G}}$ at two crossings and otherwise lies in the black regions. This gives a simple closed curve $\rho'$ that meets $\mathcal{G}$ exactly once at a vertex, again contradicting that $\mathcal{G}$ is prime. Thus (L5) and (L6) hold. Hence $D_{\mathcal{G}}\in\Lambda$.

When restricted to $\Lambda$ and $\Gamma$ respectively, these two constructions are mutual inverses.
\end{proof}

\begin{corollary}\label{boundedsizeregionscor}
Let $\mathcal{G}\in\Gamma$. Let $v\in\V(\mathcal{G})$ and let $n=|\Tr(\mathcal{G},v)|$. Then the length of the boundary of any region $r$ of $\sphere\setminus\mathcal{G}$ is at most $n$.
\end{corollary}
\begin{proof}
Suppose otherwise. Let $r$ be a region of $\sphere\setminus\mathcal{G}$ with $m$ sides, for some $m>n$.
Let $D^{\mathcal{M}}=(\graphm{M})^{-1}(\mathcal{G})$. There is a Seifert circle $C$ in $D^{\mathcal{M}}$ corresponding to $r$, which also has $m$ sides.
We may suppose that $C\subseteq\mathcal{H}$ (otherwise, replace $\mathcal{H}$ with $\mathcal{K}$ in the following argument). Let $v_C$ be the vertex of $\graph{H}{D^{\mathcal{M}}}$ corresponding to $C$. Then $v_C$ has in-degree $m$.
By Corollary \ref{boundedvalencecor}, $\nrap{L}{0}\geq m$, where $L$ is the link with diagram $D^{\mathcal{M}}$. This contradicts that $\nrap{L}{0}=|\Tr(\graph{M}{D^{\mathcal{M}}},v)|=|\Tr(\mathcal{G},v)|=n<m$.
\end{proof}
\subsection{Infinite digraphs}
%Master document is alexpolypaper.tex.

\begin{lemma}
Let $\Phi$ be an infinite set of planar, pointed digraphs (each with a fixed embedding into $\sphere$) such that every $(\mathcal{G},v)\in\Phi$ has valence bounded above by $n_1\in\mathbb{N}$. Then there is a sequence of distinct elements $(\mathcal{G}_i,v_i)_{i=1}^\infty$ such that $\ball{\mathcal{G}_{m_1}}{v_{m_1}}{m}$ and $\ball{\mathcal{G}_{m_2}}{v_{m_2}}{m}$ are the same up to ambient isotopy of $\sphere$ whenever $m_1,m_2\geq m$.
\end{lemma}
\begin{proof}
For $m\in\mathbb{N}$, let $\Theta_m$ be the set of all embeddings of planar, pointed digraphs up to ambient isotopy of $\sphere$ with valence bounded above by $n_1$ and radius at most $m$. Then $|\Theta_m|<\infty$ for all $m$. Let $\Theta=\bigcup_{i=1}^\infty \Theta_i$. Choose an enumeration $\theta\colon\Theta\to\mathbb{N}$ such that $\theta(\Theta_m)=\{1,\cdots ,|\Theta_m|\}$ for all $m$.

Let $\Psi_m=\{1,\cdots,|\Theta_m|\}$ for all $m$, and $\Psi=\prod_{i=1}^\infty \Psi_i$. Then $\Psi$ is compact. Define $\phi\colon\Phi\to\Psi$ by $\phi(\mathcal{G},v)=\theta\big(\ball{\mathcal{G}}{v}{i}\big)_{i=1}^\infty$. 
If $\phi(\mathcal{G}_1,v_1)=\phi(\mathcal{G}_2,v_2)$ then $\mathcal{G}_1,\mathcal{G}_2\in\Theta_m$ for some $m\in\mathbb{N}$, and $\mathcal{G}_1=\ball{\mathcal{G}_1}{v_1}{m}{}=\ball{\mathcal{G}_2}{v_2}{m}{}=\mathcal{G}_2$.  
Thus $\phi$ is injective. In particular, $\phi(\Phi)$ is infinite. 

Thus there is a sequence of distinct elements $\big( (J_i^j )_{i=1}^\infty \big)_{j=1}^\infty\subseteq\phi(\Phi)$ converging to an element $(J_i^\infty)_{i=1}^{\infty}\in\Psi$. 
For $m\in\mathbb{N}$, $\Psi_m$ is discrete and $J_m^j\to J_m^\infty$ as $j\to\infty$, so there exists $i_m\in\mathbb{N}$ such that $J_m^k=J_m^\infty\in\Psi_m$ for all $k\geq i_m$. By passing to a subsequence of $(J_i^j)$ we may assume that $i_m=m$ for all $m$. Then $\big( \phi^{-1}(J_i^j) \big)_{j=1}^\infty$ gives the required sequence in $\Phi$.
\end{proof}

\begin{definition}
Given a sequence $(\mathcal{G}_i,v_i)_{i=1}^\infty$ as above, define a (not necessarily finite) pointed digraph $(\mathcal{G}_\infty,v_\infty)$ by $(\mathcal{G}_\infty,v_\infty)=\bigcup_{i=1}^\infty \ball{\mathcal{G}_i}{v_i}{i}$. Note that, since $\ball{\mathcal{G}_m}{v_m}{m}{}=\ball{\mathcal{G}_k}{v_k}{m}$ for $k\geq m$, this is well-defined.
\end{definition}

\begin{lemma}
\begin{itemize}
	\item[(1)] $\mathcal{G}_\infty$ is infinite.
	\item[(2)] $\mathcal{G}_\infty$ has valence bounded above by $n_1$.
	\item[(3)] There is an injection $\mathcal{G}_{\infty}\to\sphere$ that is an embedding on any finite subgraph of $\mathcal{G}_\infty$. In particular, given a simple closed curve in $\mathcal{G}_\infty$, there is a well-defined notion of which `side' of this curve any other point of $\mathcal{G}_\infty$ lies. 
\end{itemize} 
\end{lemma}
\begin{proof}
(3) For $n\in\mathbb{N}$, let $f_n\colon \mathcal{G}_n\to\sphere$ be the embedding of $\mathcal{G}_n$ into $\sphere$.
Given $n\geq 2$, there is an isotopy $H_n\colon \sphere\times \intvl\to\sphere$ from the identity on $\sphere$ to a map $h_n\colon \sphere\to\sphere$ with $h_n(f_n(x))=f_{n-1}(x)$ for all $x\in\ball{\mathcal{G}_{n-1}}{v_{n-1}}{n-1}$. Define the inclusion $f_{\infty}\colon\mathcal{G}_{\infty}\to\sphere$ by $f_{\infty}(x)=h_2(h_3(\cdots h_n(f_n(x)\cdots)$ for all $x\in\ball{\mathcal{G}_{\infty}}{v_{\infty}}{n}$ for any $n\geq 2$.
\end{proof}

\begin{lemma}
Suppose that, for every $m\in\mathbb{N}$, the boundary of any region $r$ of $\sphere\setminus\mathcal{G}_m$ is a cycle in $\mathcal{G}_m$ with length at most $n_2$. Then $\mathcal{G}_\infty$ is \ocon{\mathcal{O}}.
\end{lemma}
\begin{proof}
Let $u,v\in\V(\mathcal{G}_\infty)$. Then $u,v\in\ball{\mathcal{G}_\infty}{v_\infty}{m}$ for some $m\in\mathbb{N}$. As $\mathcal{G}_m$ is connected, there is an (unoriented) path from $u$ to $v$ in $\ball{\mathcal{G}_\infty}{v_\infty}{m}{}=\ball{\mathcal{G}_{m+n_2}}{v_{m+n_2}}{m}$. This can be altered to a directed path from $u$ to $v$ in ${\ball{\mathcal{G}_{m+n_2}}{v_{m+n_2}}{m+n_2}{}}\subset \mathcal{G}_{\infty}$.
\end{proof}

\begin{definition}
Let $\mathcal{G}$ be a (possibly infinite) digraph.
Given sets $A,B\subseteq \V(\mathcal{G})$, define an \textit{$(A,B)$--path} $\rho$ to be a simple directed path with respect to $\mathcal{O}$ that begins at a vertex of $A$ and ends at a vertex of $B$ without meeting any vertex of $A\cup B$ in its interior. If $\rho$ contains exactly one edge, call it an \textit{$(A,B)$--edge}.
\end{definition}

\begin{proposition}[Menger's Theorem]
Let $\mathcal{G}$ be a (possibly infinite) digraph. Let $A\subseteq\V(\mathcal{G})$ and $v_B\in\V(\mathcal{G})\setminus A$. Then either there are edge-disjoint $(A,v_B)$--paths $\rho_1,\rho_2$ or there is a partition of $\V(\mathcal{G})$ into sets $\tilde{A}$ and $\tilde{B}$ such that $A\subseteq\tilde{A}$, $v_B\in\tilde{B}$, and there is at most one $(\tilde{A},\tilde{B})$--edge.
\end{proposition}
\begin{proof}
Let $A_1=\{v\in\V(\mathcal{G}):\textrm{there is an $(A,v)$--path in $(\mathcal{G},\mathcal{O})$}\}$ and let $B_1=\V(\mathcal{G})\setminus A_1$. Then $A\subseteq A_1$. If $v_B\in B_1$ then this gives the required partition.

Suppose $v_B\in A_1$. Choose an $(A,v_B)$--path $\rho$. Define $\mathcal{O}_\rho$ to be the orientation of $\mathcal{G}$ given by reversing the direction of $\mathcal{O}$ on every edge of $\rho$. Let $A_2=\{v\in\V(\mathcal{G}):\textrm{there is an $(A,v)$--path in $(\mathcal{G},\mathcal{O}_\rho)$}\}$. 

Assume $v_B\in A_2$. Let $\rho '$ be an $(A,v_B)$--path in $(\mathcal{G},\mathcal{O}_\rho)$, and let $B={\{e\in\E(\mathcal{G}):e\in\rho \textrm{ XOR } e\in\rho '\}}$.
Consider the edges in $B$ with respect to $\mathcal{O}$. Exactly two such edges start in $A$ while no edge ends in $A$, and two edges end at $v_B$. At any vertex $v\notin A\cup\{v_B\}$, the number of edges in $B$ ending at $v$ is equal to the number starting at $v$. Thus, since $B$ is finite, there are two edge-disjoint $(A,v_B)$--paths in $B$. 

If instead $v_B\notin A_2$, let $B_2=\V(\mathcal{G})\setminus A_2$. By definition, there are no $(A_2,B_2)$--edges with respect to $\mathcal{O}_\rho$. In particular, no edge of $\rho$ is a $(B_2,A_2)$--edge with respect to $\mathcal{O}$. This means $\rho$ contains exactly one $(A_2,B_2)$--edge with respect to $\mathcal{O}$. Hence $\tilde{A}=A_2$ and $\tilde{B}=B_2$ are as required.
\end{proof}

\begin{lemma}\label{twoabedgeslemma}
Suppose that $\mathcal{G}_m\in\Gamma$ and the boundary of any region $r$ of $\sphere\setminus\mathcal{G}_m$ is a cycle in $\mathcal{G}_m$ with length at most $n_2$ for each $m\in\mathbb{N}$. Then, for any partition of $\V(\mathcal{G}_\infty)$ into non-empty sets $A$ and $B$, there are at least two $(A,B)$--edges. 
\end{lemma}
\begin{proof}
$\mathcal{G}_\infty$ is \ocon{\mathcal{O}}, so there is at least one $(B,A)$--edge. Let $e$ be a $(B,A)$--edge. Choose $m_1\in\mathbb{N}$ such that $e\in\ball{\mathcal{G}_\infty}{v_\infty}{m_1}$ and set $m_2=m_1+n_2$.
Let $r_1$ and $r_2$ be the regions of $\sphere\setminus\mathcal{G}_{m_2}$ adjacent to $e$. The boundaries of these regions are both contained in $\ball{\mathcal{G}_{m_2}}{v_{m_2}}{m_2}$. Let $A_{m_2}=A\cap\ball{\mathcal{G}_{m_2}}{v_{m_2}}{m_2}$ and $B_{m_2}=B\cap\ball{\mathcal{G}_{m_2}}{v_{m_2}}{m_2}$. Then the boundaries of $r_1$ and $r_2$ contain $(A_{m_2},B_{m_2})$--edges $e_1$ and $e_2$ respectively. Since $\mathcal{G}_{m_2}$ is prime, $e_1\neq e_2$. Thus $e_1$ and $e_2$ are distinct $(A,B)$--edges in $\mathcal{G}_\infty$.
\end{proof}

\begin{corollary}\label{spearpaircor}
If $(\mathcal{G}_i)_{i=1}^\infty$ is as in Lemma \ref{twoabedgeslemma} and $A$ is a non-empty, proper subset of $\V(\mathcal{G}_\infty)$, there is a vertex $v_B\notin A$ with two $(A,v_B)$--paths that do not meet except at their endpoints.
\end{corollary}
\begin{proof}
Choose $v\notin A$. There are edge-disjoint $(A,v)$--paths $\rho_1$ and $\rho_2$. Take $v_B$ to be the first common vertex of $\rho_1$ and $\rho_2$ not in $A$ as measured along $\rho_1$.
\end{proof}

\begin{definition}
Call distinct directed paths $\rho_1$ and $\rho_2$ with the same endpoints a \textit{spear-pair} if, for some $m\in\mathbb{N}\cup\{0\}$, they run together for the first $m$ vertices, and after that point meet only at their final vertex (see Figure \ref{infgraphpic1}).
\begin{figure}[htbp]
\centering
%LaTeX with PSTricks extensions
%%Creator: 0.46
%%Please note this file requires PSTricks extensions
\psset{xunit=.75pt,yunit=.75pt,runit=.75pt}
\begin{pspicture}(200,80)
{
\pscustom[linewidth=1,linecolor=black]
{
\newpath
\moveto(30,29.99998262)
\lineto(100,29.99998262)
}
}
{
\pscustom[linewidth=1,linecolor=black,fillstyle=solid,fillcolor=black]
{
\newpath
\moveto(60,29.99998262)
\lineto(56,25.99998262)
\lineto(70,29.99998262)
\lineto(56,33.99998262)
\lineto(60,29.99998262)
\closepath
}
}
{
\pscustom[linewidth=1]
{
\newpath
\moveto(100,29.99998262)
\curveto(110,54.00000262)(147,49.99998262)(147,49.99998262)
}
}
{
\pscustom[linewidth=1,linecolor=black,fillstyle=solid,fillcolor=black]%arrowhead
{
\newpath
\moveto(137.0579301,51.07480636)
\lineto(132.65117264,47.5279079)
\lineto(147,49.99998262)
\lineto(133.51103163,55.48156382)
\lineto(137.0579301,51.07480636)
\closepath
}
}
{
\pscustom[linewidth=1]
{
\newpath
\moveto(147,49.99998262)
\curveto(170,46.99998262)(175,29.99998262)(175,29.99998262)
}
}
{
\pscustom[linewidth=1]
{
\newpath
\moveto(100,29.99998262)
\curveto(110,5.99998262)(147,9.99998262)(147,9.99998262)
}
}
{
\pscustom[linewidth=1,linecolor=black,fillstyle=solid,fillcolor=black]%arrowhead
{
\newpath
\moveto(137.05792952,8.92516419)
\lineto(133.51102871,4.51840863)
\lineto(147,9.99998262)
\lineto(132.65117396,12.47206501)
\lineto(137.05792952,8.92516419)
\closepath
}
}
{
\pscustom[linewidth=1]
{
\newpath
\moveto(147,9.99998262)
\curveto(170,12.99998262)(175,29.99998262)(175,29.99998262)
}
}
{
\put(90,50){$\rho_1$}
\put(90,5){$\rho_2$}
}
\end{pspicture}
\caption{\label{infgraphpic1}}
\end{figure}
\end{definition}

\begin{proposition}
Let $(\mathcal{G}_i)_{i=1}^\infty$ be as in Lemma \ref{twoabedgeslemma}. Let $n\in\mathbb{N}$. Then there is a set $\mathcal{F}_n\subseteq\V(\mathcal{G}_\infty)$ such that there are at least $2^n$ (finite) \spantree{\mathcal{F}_n}s of $\mathcal{G}_{\infty}$ with origin $v_\infty$.
\end{proposition}
\begin{proof}
We inductively construct, for $0\leq m \leq n$, a vertex $w_m\in\V(\mathcal{G}_\infty)$, a set $A_m\subset \V(\mathcal{G}_\infty)$ and  $(v_\infty,w_m)$--paths $\rho^1_m,\rho^2_m$. We also construct an unoriented (and possibly not simple) closed curve $\sigma_m$ in $\mathcal{G}_\infty$ and an open disc $\mathbb{D}_m\subset\sphere$. These are all chosen with the following properties.
\begin{itemize}
	\item[(1)$_m$] $B_m=\V(\mathcal{G}_\infty)\setminus A_m$ is infinite. 
	\item[(2)$_m$] $\rho^1_m,\rho^2_m\subseteq \mathcal{G}_\infty[A_m]$.	
	\item[(3)$_m$] If $m>0$ then $\rho^1_m,\rho^2_m$ are a spear-pair.	
	\item[(4)$_m$] $\sigma_m\subseteq\bigcup_{i=0}^m (\rho^1_i \cup \rho^2_i)$.	
	\item[(5)$_m$] $\mathbb{D}_m$ is a connected component of $\sphere\setminus\sigma_m$.
	\item[(6)$_m$] $B_m\subset\mathbb{D}_m$ and $A_m\cap\mathbb{D}_m=\emptyset$.
	\item[(7)$_m$] If $m>0$ then $w_m\in B_{m-1}$.
	\item[(8)$_m$] Let $\mathcal{F}_m=\{w_0,\cdots,w_m\}$. For each $(J_i)_{i=1}^m\in\{1,2\}^m$ there is a \spantree{\mathcal{F}_m} $\mathcal{T}$ of $\mathcal{G}_\infty$ with origin $v_\infty$ such that $\mathcal{T}\subseteq \mathcal{G}_{\infty}[A_m]$ and, for $1\leq i \leq m$, the edge of $\mathcal{T}$ ending at $w_i$ is the last edge of $\rho_i^{J_i}$.
\end{itemize}

Let $w_0=A_0=\rho^1_0=\rho^2_0=\sigma_0=v_\infty$ and $\mathbb{D}_0=\sphere\setminus v_\infty$. Then (1)$_0$--(8)$_0$ hold.

Suppose $w_m$, $A_m$, $\rho^1_m$, $\rho^2_m$, $\sigma_m$ and $\mathbb{D}_m$ have been defined for some $m<n$. Then by Corollary \ref{spearpaircor} there exists $w_{m+1}\in B_m$ and $(A_m,w_{m+1})$-paths $\rho^3_{m+1},\rho^4_{m+1}$ that do not meet except at their endpoints. Since $w_{m+1}\in\mathbb{D}_m$, both $\rho^3_{m+1}$ and $\rho^4_{m+1}$ must have their initial vertex on $\sigma_m$. 
By $(4)_m$, these can be extended using directed paths in $\bigcup_{i=0}^m (\rho^1_i \cup \rho^2_i)$ to a spear-pair of $(v_\infty,w_{m+1})$--paths $\rho^1_{m+1},\rho^2_{m+1}$.

Now $\rho^3_{m+1}\cup\rho^4_{m+1}$ forms a finite length simple closed curve or arc contained in $\mathcal{G}_\infty$ with both endpoints on the boundary of $\mathbb{D}_m$. Thus $\mathbb{D}_m\setminus (\rho^3_{m+1}\cup\rho^4_{m+1})$ consists of two discs $\mathbb{D}_{m+1}^1$ and $\mathbb{D}_{m+1}^2$. Further, $B_{m+1}^1=\mathbb{D}_{m+1}^1\cap\V(\mathcal{G}_\infty)$ and $B_{m+1}^2=\mathbb{D}_{m+1}^2\cap\V(\mathcal{G}_\infty)$ are well defined. At least one of $B_{m+1}^1,B_{m+1}^2$ is infinite. Suppose $B_{m+1}^1$ is infinite. Let $B_{m+1}=B_{m+1}^1$ (that is, let $A_{m+1}=\V(\mathcal{G})\setminus B^1_{m+1}$) and let $\mathbb{D}_{m+1}=\mathbb{D}^1_{m+1}$. Finally, define $\sigma_{m+1}$ as the boundary of $\mathbb{D}_{m+1}$. Note that $\sigma_{m+1}$ consists of $\rho^3_{m+1}\cup\rho^4_{m+1}$, together with all, one section or none of $\sigma_m$. Figure \ref{infgraphpic2} shows a specific example.
\begin{figure}[htbp]
\centering
%LaTeX with PSTricks extensions
%%Creator: inkscape 0.46
%%Please note this file requires PSTricks extensions
\psset{xunit=.5pt,yunit=.5pt,runit=.5pt}
\begin{pspicture}(290,290)
{
\pscustom[linestyle=none,fillstyle=solid,fillcolor=lightgray]
{
\newpath
\moveto(20,270)
\lineto(270,270)
\lineto(270,20)
\lineto(20,20)
\lineto(20,270)
\closepath
}
}
{
\pscustom[linewidth=0.1,linecolor=white,fillstyle=solid,fillcolor=white]
{
\newpath
\moveto(45,145.00000262)
\lineto(45,145.00000262)
\curveto(48,151.00000262)(67.569369,157.71291262)(71,160.00000262)
\curveto(74,162.00000262)(105,172.00000262)(125,171.00000262)
\curveto(126,191.99999262)(132,208.00000262)(146,226.00000262)
\curveto(153,235.00000262)(166,247.00000262)(181,250.00000262)
\lineto(195,250.00000262)
\curveto(232,233.00000262)(243,206.00000262)(246,200.00000262)
\curveto(255,183.00000262)(256,162.00000262)(255,156.00000262)
\curveto(252,148.00000262)(254,117.00000262)(211,98.00000262)
\curveto(196,94.00000262)(181,90.00000262)(165,88.00000262)
\curveto(150,88.00000262)(134,90.00000262)(121,98.00000262)
\curveto(128,92.00000262)(132,84.00000262)(135,67.00000262)
\curveto(135,60.00000262)(135,50.99998262)(126,36.99998262)
\lineto(114,36.99998262)
\curveto(107,48.99998262)(106,53.00000262)(105,66.00000262)
\curveto(106,75.00000262)(108.75,84.00000262)(112,90.00000262)
\curveto(115.25,96.00000262)(119,99.00000262)(119,98.00000262)
\lineto(119,99.00000262)
\lineto(103,111.00000262)
\lineto(90,124.00000262)
\curveto(77,127.00000262)(63,133.00000262)(50,140.00000262)
\lineto(45,145.00000262)
\closepath
}
}
{
\pscustom[linewidth=1,linecolor=black]
{
\newpath
\moveto(50,150.00000262)
\curveto(105,180.00000262)(145,170.00000262)(145,170.00000262)
}
}
{
\pscustom[linewidth=1,linecolor=black,fillstyle=solid,fillcolor=black]
{
\newpath
\moveto(135.298575,172.42535887)
\lineto(130.4478625,169.51493137)
\lineto(145,170.00000262)
\lineto(132.3881475,177.27607137)
\lineto(135.298575,172.42535887)
\closepath
}
}
{
\pscustom[linewidth=1,linecolor=black]
{
\newpath
\moveto(145,170.00000262)
\curveto(195,165.00000262)(220,150.00000262)(220,150.00000262)
}
}
{
\pscustom[linewidth=1,linecolor=black]
{
\newpath
\moveto(50,140.00000262)
\curveto(105,110.00000262)(145,120.00000262)(145,120.00000262)
}
}
{
\pscustom[linewidth=1,linecolor=black,fillstyle=solid,fillcolor=black]
{
\newpath
\moveto(135.298575,117.57464637)
\lineto(132.3881475,112.72393387)
\lineto(145,120.00000262)
\lineto(130.4478625,120.48507387)
\lineto(135.298575,117.57464637)
\closepath
}
}
{
\pscustom[linewidth=1,linecolor=black]
{
\newpath
\moveto(145,120.00000262)
\curveto(195,125.00000262)(220,140.00000262)(220,140.00000262)
}
}
{
\pscustom[linewidth=1,linecolor=black]
{
\newpath
\moveto(125,171.00000262)
\curveto(125,206.00000262)(150,230.00000262)(150,230.00000262)
}
}
{
\pscustom[linewidth=1,linecolor=black,fillstyle=solid,fillcolor=black]
{
\newpath
\moveto(142.78612679,223.07468434)
\lineto(142.67070482,217.41900774)
\lineto(150,230.00000262)
\lineto(137.13045019,223.19010631)
\lineto(142.78612679,223.07468434)
\closepath
}
}
{
\pscustom[linewidth=1,linecolor=black]
{
\newpath
\moveto(145,225.00000262)
\curveto(160,245.00000262)(180,250.00000262)(180,250.00000262)
}
}
{
\pscustom[linewidth=1,linecolor=black]
{
\newpath
\moveto(255,165.00000262)
\curveto(250,230.00000262)(195,250.00000262)(195,250.00000262)
}
}
{
\pscustom[linewidth=1,linecolor=black]
{
\newpath
\moveto(105,65.00000262)
\curveto(105,49.99998262)(115,34.99998262)(115,34.99998262)
}
}
{
\pscustom[linewidth=1,linecolor=black]
{
\newpath
\moveto(120,99.00000262)
\curveto(135,89.00004262)(135,60.00010262)(135,60.00010262)
}
}
{
\pscustom[linewidth=1,linecolor=black,fillstyle=solid,fillcolor=black]
{
\newpath
\moveto(135,70.00010262)
\lineto(131,74.00010262)
\lineto(135,60.00010262)
\lineto(139,74.00010262)
\lineto(135,70.00010262)
\closepath
}
}
{
\pscustom[linewidth=1,linecolor=black]
{
\newpath
\moveto(135,65.00010262)
\curveto(135,49.99998262)(125,34.99998262)(125,34.99998262)
}
}
{
\pscustom[linewidth=1,linecolor=black,fillstyle=solid,fillcolor=black]
{
\newpath
\moveto(233,145.00000077)
\curveto(232.99999997,140.96695519)(229.8552505,137.6937588)(225.98046994,137.69375883)
\curveto(222.10568938,137.69375886)(218.96093997,140.96695531)(218.96094,145.00000089)
\curveto(218.96094003,149.03304648)(222.1056895,152.30624286)(225.98047006,152.30624283)
\curveto(229.85123642,152.3062428)(232.99369178,149.04075242)(232.99999071,145.01189033)
}
}
{
\pscustom[linewidth=1,linecolor=black,fillstyle=solid,fillcolor=black]
{
\newpath
\moveto(195,250.00000077)
\curveto(194.99999997,245.96695519)(191.8552505,242.6937588)(187.98046994,242.69375883)
\curveto(184.10568938,242.69375886)(180.96093997,245.96695531)(180.96094,250.00000089)
\curveto(180.96094003,254.03304648)(184.1056895,257.30624286)(187.98047006,257.30624283)
\curveto(191.85123642,257.3062428)(194.99369178,254.04075242)(194.99999071,250.01189033)
}
}
{
\pscustom[linewidth=1,linecolor=black,fillstyle=solid,fillcolor=black]
{
\newpath
\moveto(127.01953,32.30624077)
\curveto(127.01952997,28.27319519)(123.8747805,24.9999988)(119.99999994,24.99999883)
\curveto(116.12521938,24.99999886)(112.98046997,28.27319531)(112.98047,32.30624089)
\curveto(112.98047003,36.33928648)(116.1252195,39.61248286)(120.00000006,39.61248283)
\curveto(123.87076642,39.6124828)(127.01322178,36.34699242)(127.01952071,32.31813033)
}
}
{
\pscustom[linewidth=1,linecolor=black,fillstyle=solid,fillcolor=black]
{
\newpath
\moveto(51.000009,144.98811077)
\curveto(51.00000897,140.95506519)(47.8552595,137.6818688)(43.98047894,137.68186883)
\curveto(40.10569838,137.68186886)(36.96094897,140.95506531)(36.960949,144.98811089)
\curveto(36.96094903,149.02115648)(40.1056985,152.29435286)(43.98047906,152.29435283)
\curveto(47.85124542,152.2943528)(50.99370078,149.02886242)(50.99999971,145.00000033)
}
}
{
\pscustom[linewidth=1,linecolor=black]
{
\newpath
\moveto(90,124.00000262)
\curveto(135,74.00000262)(173.87811,88.46953262)(200,95.00000262)
\curveto(220,100.00000262)(255,115.00000262)(255,165.00000262)
}
}
{
\pscustom[linewidth=1,linecolor=black,fillstyle=solid,fillcolor=black]
{
\newpath
\moveto(255,155.00000262)
\lineto(259,151.00000262)
\lineto(255,165.00000262)
\lineto(251,151.00000262)
\lineto(255,155.00000262)
\closepath
}
}
{
\pscustom[linewidth=1,linecolor=black]
{
\newpath
\moveto(120,99.00000262)
\curveto(105,89.00000262)(105,60.00000262)(105,60.00000262)
}
}
{
\pscustom[linewidth=1,linecolor=black,fillstyle=solid,fillcolor=black]
{
\newpath
\moveto(105,70.00000262)
\lineto(101,74.00000262)
\lineto(105,60.00000262)
\lineto(109,74.00000262)
\lineto(105,70.00000262)
\closepath
}
}
{
\put(30,160){$w_0$}
\put(220,160){$w_1$}
\put(200,250){$w_2$}
\put(135,30){$w_3$}
\put(95,150){$\rho^1_1$}
\put(140,130){$\rho^2_1$}
\put(145,210){$\rho^1_2$}
\put(215,200){$\rho^2_2$}
\put(75,60){$\rho^1_3$}
\put(140,60){$\rho^2_3$}
\put(60,220){$\mathbb{D}_3$}
}
\end{pspicture}
\caption{\label{infgraphpic2}}
\end{figure}

Properties (1)$_{m+1}$--(7)$_{m+1}$ now hold. It remains to check (8)$_{m+1}$. Note that since the final edges of $\rho_i^1$ and $\rho_i^2$ differ for $1\leq i \leq m+1$, (8)$_{m+1}$ implies there are at least $2^{m+1}$ \spantree{\mathcal{F}_{m+1}}s with origin $v_\infty$.

Let $(J_i)_{i=1}^{m+1}\in\{1,2\}^{m+1}$. If $m+1=1$, let $\mathcal{T}=\rho_1^{J_1}$. If $m+1>1$, choose $\mathcal{T}'$ with origin $v_\infty$  such that $\mathcal{T}'\subseteq \mathcal{G}_{\infty}[A_m]$ and, for $1\leq i \leq m$, the edge of $\mathcal{T}'$ ending at $w_i$ is the last edge of $\rho_i^{J_i}$. Since $w_{m+1}\in B_m$, it does not lie on $\mathcal{T}'$. Let $v'$ be the last vertex of $\rho_{m+1}^{J_{m+1}}$ that meets $\mathcal{T}'$, and let $\rho '$ be the section of $\rho_{m+1}^{J_{m+1}}$ from $v'$ to $w_{m+1}$. Then $\mathcal{T}=\mathcal{T}'\cup\rho '$ has the required properties.
\end{proof}

\begin{theorem}\label{gammafinitethm}
The set $\Phi_n={\{\mathcal{G}\in\Gamma:\exists v\in\V(\mathcal{G}),|\Tr(\mathcal{G},v)|\leq n\}}$ is finite for each $n\in\mathbb{N}$.
\end{theorem}
\begin{proof}
Fix $n\in\mathbb{N}$. For each $\mathcal{G}\in\Phi_n$, fix a vertex $v\in\V(\mathcal{G})$ such that $|\Tr(\mathcal{G},v)|\leq n$ and fix an embedding of $\mathcal{G}$ into $\sphere$. By Corollary \ref{boundedvalencecor} each $\mathcal{G}\in\Phi_n$ has valence bounded above by $2n$, and by Corollary \ref{boundedsizeregionscor} the boundary of any region $r$ of $\sphere\setminus\mathcal{G}$ is a cycle in $\mathcal{G}$ with length at most $n$. 

Suppose, for a contradiction, that $\Phi_n$ is infinite. Then there is a sequence $(\mathcal{G}_i,v_i)$ as above, from which we can define an infinite pointed digraph $\mathcal{G}_\infty$. There is a set $\mathcal{F}\subseteq\V(\mathcal{G}_\infty)$ and a sequence $(\mathcal{T}_i)_{i=1}^{n+1}$ of finite \spantree{\mathcal{F}}s with origin $v_\infty$. Choose $m\in\mathbb{N}$ such that $\mathcal{T}_i\subseteq\ball{\mathcal{G}_\infty}{v_\infty}{m}{}=\ball{\mathcal{G}_m}{v_m}{m}$ for $1\leq i \leq n+1$. Since $\mathcal{G}_m$ is \ocon{\mathcal{O}}, Lemma \ref{extendtreeslemma} implies that $n+1\leq |\Tr(\ball{\mathcal{G}_m}{v_m}{m}{},v_m)|\leq |\Tr(\mathcal{G}_m,v_m)|\leq n$.
\end{proof}

\begin{remark}
Theorem \ref{finitenessthm} now follows. The details of the proof are contained in the proofs of Theorem \ref{juhaszthmspecial} and Theorem \ref{fulljuhaszproof}, and so are omitted here.
For the prime case, see the proof of Corollary \ref{fibrednesscor}.
\end{remark}

%------------------------------

\section{Proof of Juhasz' theorem}\label{juhaszsec}
\subsection{Proof}
%Master document is alexpolypaper.tex.

By Theorem \ref{gammafinitethm} we now know that the set $\Phi_3$ is finite. In order to prove Theorem \ref{finitenessthm}, we calculate this set explicitly.

\begin{theorem}
Let $\mathcal{G}\in\Gamma$. Suppose that $|\Tr(\mathcal{G},v)|<4$ for any $v\in\V(\mathcal{G})$. Then, up to reflection, $\mathcal{G}$ is one of the digraphs $\mathcal{G}_{\alpha},\mathcal{G}_{\beta},\mathcal{G}_{\gamma},\mathcal{G}_{\delta}$ shown in Figure \ref{threetreegraphspic1}.
\begin{figure}[htbp]
\centering
%LaTeX with PSTricks extensions
%%Creator: 0.46
%%Please note this file requires PSTricks extensions
\psset{xunit=.5pt,yunit=.5pt,runit=.5pt}
\begin{pspicture}(480,90)
{
\pscustom[linewidth=1,linecolor=black,fillstyle=solid,fillcolor=black]%circle
{
\newpath
\moveto(269.5,39.99989)
\curveto(269.5,34.47989)(265.02,29.99989)(259.5,29.99989)
\curveto(253.98,29.99989)(249.5,34.47989)(249.5,39.99989)
\curveto(249.5,45.51989)(253.98,49.99989)(259.5,49.99989)
\curveto(265.02,49.99989)(269.5,45.51989)(269.5,39.99989)
\closepath
}
}
{
\pscustom[linewidth=1,linecolor=black,fillstyle=solid,fillcolor=black]%circle
{
\newpath
\moveto(410.5,44.99999)
\curveto(410.5,39.47999)(406.02,34.99999)(400.5,34.99999)
\curveto(394.98,34.99999)(390.5,39.47999)(390.5,44.99999)
\curveto(390.5,50.51999)(394.98,54.99999)(400.5,54.99999)
\curveto(406.02,54.99999)(410.5,50.51999)(410.5,44.99999)
\closepath
}
}
{
\pscustom[linewidth=1,linecolor=black,fillstyle=solid,fillcolor=black]%circle
{
\newpath
\moveto(475.5,74.999991)
\curveto(475.5,69.479991)(471.02,64.999991)(465.5,64.999991)
\curveto(459.98,64.999991)(455.5,69.479991)(455.5,74.999991)
\curveto(455.5,80.519991)(459.98,84.999991)(465.5,84.999991)
\curveto(471.02,84.999991)(475.5,80.519991)(475.5,74.999991)
\closepath
}
}
{
\pscustom[linewidth=1,linecolor=black,fillstyle=solid,fillcolor=black]%circle
{
\newpath
\moveto(349.5,39.99989)
\curveto(349.5,34.47989)(345.02,29.99989)(339.5,29.99989)
\curveto(333.98,29.99989)(329.5,34.47989)(329.5,39.99989)
\curveto(329.5,45.51989)(333.98,49.99989)(339.5,49.99989)
\curveto(345.02,49.99989)(349.5,45.51989)(349.5,39.99989)
\closepath
}
}
{
\pscustom[linewidth=1,linecolor=black,fillstyle=solid,fillcolor=black]%circle
{
\newpath
\moveto(65.5,40)
\curveto(65.5,34.48)(61.02,30)(55.5,30)
\curveto(49.98,30)(45.5,34.48)(45.5,40)
\curveto(45.5,45.52)(49.98,50)(55.5,50)
\curveto(61.02,50)(65.5,45.52)(65.5,40)
\closepath
}
}
{
\pscustom[linewidth=1,linecolor=black,fillstyle=solid,fillcolor=black]%circle
{
\newpath
\moveto(128.5,39.99996)
\curveto(128.5,34.47996)(124.02,29.99996)(118.5,29.99996)
\curveto(112.98,29.99996)(108.5,34.47996)(108.5,39.99996)
\curveto(108.5,45.51996)(112.98,49.99996)(118.5,49.99996)
\curveto(124.02,49.99996)(128.5,45.51996)(128.5,39.99996)
\closepath
}
}
{
\pscustom[linewidth=1,linecolor=black,fillstyle=solid,fillcolor=black]%circle
{
\newpath
\moveto(208.5,39.99996)
\curveto(208.5,34.47996)(204.02,29.99996)(198.5,29.99996)
\curveto(192.98,29.99996)(188.5,34.47996)(188.5,39.99996)
\curveto(188.5,45.51996)(192.98,49.99996)(198.5,49.99996)
\curveto(204.02,49.99996)(208.5,45.51996)(208.5,39.99996)
\closepath
}
}
{
\pscustom[linewidth=1,linecolor=black,fillstyle=solid,fillcolor=black]%circle
{
\newpath
\moveto(475.5,15)
\curveto(475.5,9.48)(471.02,5)(465.5,5)
\curveto(459.98,5)(455.5,9.48)(455.5,15)
\curveto(455.5,20.52)(459.98,25)(465.5,25)
\curveto(471.02,25)(475.5,20.52)(475.5,15)
\closepath
}
}
{
\pscustom[linewidth=1,linecolor=black]
{
\newpath
\moveto(118.5,49.99998262)
\curveto(163.5,69.99996262)(198.5,49.99998262)(198.5,49.99998262)
}
}
{
\pscustom[linewidth=1,linecolor=black,fillstyle=solid,fillcolor=black]%arrowhead
{
\newpath
\moveto(127.63811699,54.06136389)
\lineto(129.66881128,59.34116319)
\lineto(118.5,49.99998262)
\lineto(132.9179163,52.0306696)
\lineto(127.63811699,54.06136389)
\closepath
}
}
{
\pscustom[linewidth=1,linecolor=black]
{
\newpath
\moveto(118.5,29.99988262)
\curveto(163.28244,9.90318262)(196.94678,29.30958262)(198.5,29.99988262)
}
}
{
\pscustom[linewidth=1,linecolor=black,fillstyle=solid,fillcolor=black]%arrowhead
{
\newpath
\moveto(189.36184084,25.93859622)
\lineto(187.33109174,20.65881799)
\lineto(198.5,29.99988262)
\lineto(184.08206262,27.96934532)
\lineto(189.36184084,25.93859622)
\closepath
}
}
{
\pscustom[linewidth=1,linecolor=black]
{
\newpath
\moveto(126.5,44.99988262)
\curveto(156.5,54.99996262)(188.5,44.99988262)(188.5,44.99988262)
\lineto(188.5,44.99988262)
}
}
{
\pscustom[linewidth=1,linecolor=black,fillstyle=solid,fillcolor=black]%arrowhead
{
\newpath
\moveto(178.95520701,47.98265429)
\lineto(173.94418115,45.35784576)
\lineto(188.5,44.99988262)
\lineto(176.33039849,52.99368015)
\lineto(178.95520701,47.98265429)
\closepath
}
}
{
\pscustom[linewidth=1,linecolor=black]
{
\newpath
\moveto(128.5,34.99988262)
\curveto(158.5,24.99988262)(190.5,35.99988262)(190.5,35.99988262)
\lineto(190.5,35.99988262)
}
}
{
\pscustom[linewidth=1,linecolor=black,fillstyle=solid,fillcolor=black]%arrowhead
{
\newpath
\moveto(137.98683298,31.83760496)
\lineto(143.04647724,34.36742709)
\lineto(128.5,34.99988262)
\lineto(140.51665511,26.7779607)
\lineto(137.98683298,31.83760496)
\closepath
}
}
{
\pscustom[linewidth=1,linecolor=black]
{
\newpath
\moveto(268.5,42.99988262)
\curveto(298.5,44.99988262)(329.5,42.99988262)(329.5,42.99988262)
}
}
{
\pscustom[linewidth=1,linecolor=black,fillstyle=solid,fillcolor=black]%arrowhead
{
\newpath
\moveto(319.52074691,43.6437054)
\lineto(315.27151656,39.90953327)
\lineto(329.5,42.99988262)
\lineto(315.78657479,47.89293575)
\lineto(319.52074691,43.6437054)
\closepath
}
}
{
\pscustom[linewidth=1,linecolor=black]
{
\newpath
\moveto(270.5,36.99988262)
\curveto(300.5,34.99988262)(329.5,36.99988262)(329.5,36.99988262)
}
}
{
\pscustom[linewidth=1,linecolor=black,fillstyle=solid,fillcolor=black]%arrowhead
{
\newpath
\moveto(280.47785158,36.33469251)
\lineto(284.73506825,40.0597571)
\lineto(270.5,36.99988262)
\lineto(284.20291617,32.07747584)
\lineto(280.47785158,36.33469251)
\closepath
}
}
{
\pscustom[linewidth=1,linecolor=black]
{
\newpath
\moveto(264.5,49.99988262)
\curveto(299.5,69.99989262)(335.5,48.99988262)(335.5,48.99988262)
}
}
{
\pscustom[linewidth=1,linecolor=black,fillstyle=solid,fillcolor=black]%arrowhead
{
\newpath
\moveto(273.18243035,54.96127387)
\lineto(274.67084599,60.41880251)
\lineto(264.5,49.99988262)
\lineto(278.639959,53.47285823)
\lineto(273.18243035,54.96127387)
\closepath
}
}
{
\pscustom[linewidth=1,linecolor=black]
{
\newpath
\moveto(263.5,30.99988262)
\curveto(297.66197,10.61078262)(329.76897,27.80338262)(334.5,29.99988262)
}
}
{
\pscustom[linewidth=1,linecolor=black,fillstyle=solid,fillcolor=black]%arrowhead
{
\newpath
\moveto(325.42987397,25.78884785)
\lineto(323.48623747,20.47638353)
\lineto(334.5,29.99988262)
\lineto(320.11740965,27.73248435)
\lineto(325.42987397,25.78884785)
\closepath
}
}
{
\pscustom[linewidth=1,linecolor=black]
{
\newpath
\moveto(258.5,49.99988262)
\curveto(301.72299,93.68930262)(338.33875,51.99128262)(339.5,50.99988262)
}
}
{
\pscustom[linewidth=1,linecolor=black,fillstyle=solid,fillcolor=black]%arrowhead
{
\newpath
\moveto(331.8946429,57.49284368)
\lineto(326.25531563,57.04788526)
\lineto(339.5,50.99988262)
\lineto(331.44968448,63.13217094)
\lineto(331.8946429,57.49284368)
\closepath
}
}
{
\pscustom[linewidth=1,linecolor=black]
{
\newpath
\moveto(260.5,28.99988262)
\curveto(305.85806,-13.14841738)(339.53074,29.81548262)(339.5,29.99988262)
}
}
{
\pscustom[linewidth=1,linecolor=black,fillstyle=solid,fillcolor=black]%arrowhead
{
\newpath
\moveto(267.82551036,22.1927615)
\lineto(273.47856295,22.40011719)
\lineto(260.5,28.99988262)
\lineto(268.03286605,16.5397089)
\lineto(267.82551036,22.1927615)
\closepath
}
}
{
\pscustom[linewidth=1,linecolor=black]
{
\newpath
\moveto(404.5,53.99998262)
\curveto(429.0697,78.92828262)(454.64667,79.85778262)(455.5,80.00000262)
}
}
{
\pscustom[linewidth=1,linecolor=black,fillstyle=solid,fillcolor=black]%arrowhead
{
\newpath
\moveto(445.63605764,78.35603149)
\lineto(442.34806914,73.75286609)
\lineto(455.5,80.00000262)
\lineto(441.03289224,81.64401998)
\lineto(445.63605764,78.35603149)
\closepath
}
}
{
\pscustom[linewidth=1,linecolor=black]
{
\newpath
\moveto(469.5,66.00000262)
\curveto(479.6397,46.55878262)(470.57509,25.30038262)(470.5,24.99998262)
}
}
{
\pscustom[linewidth=1,linecolor=black,fillstyle=solid,fillcolor=black]%arrowhead
{
\newpath
\moveto(472.92505229,34.7014836)
\lineto(470.01447282,39.55210491)
\lineto(470.5,24.99998262)
\lineto(477.77567361,37.61206308)
\lineto(472.92505229,34.7014836)
\closepath
}
}
{
\pscustom[linewidth=1,linecolor=black]
{
\newpath
\moveto(460.5,65.00000262)
\curveto(450.89205,36.17898262)(461.50697,23.84318262)(461.5,23.99998262)
}
}
{
\pscustom[linewidth=1,linecolor=black,fillstyle=solid,fillcolor=black]%arrowhead
{
\newpath
\moveto(457.33744288,55.51326279)
\lineto(459.86711597,50.45354402)
\lineto(460.5,65.00000262)
\lineto(452.27772411,52.98358971)
\lineto(457.33744288,55.51326279)
\closepath
}
}
{
\pscustom[linewidth=1,linecolor=black]
{
\newpath
\moveto(456.5,71.00000262)
\curveto(437.35658,51.12238262)(411.17218,50.09608262)(410.5,49.99998262)
}
}
{
\pscustom[linewidth=1,linecolor=black,fillstyle=solid,fillcolor=black]%arrowhead
{
\newpath
\moveto(420.39934172,51.41526831)
\lineto(423.79296412,55.94111928)
\lineto(410.5,49.99998262)
\lineto(424.92519268,48.0216459)
\lineto(420.39934172,51.41526831)
\closepath
}
}
{
\pscustom[linewidth=1,linecolor=black]
{
\newpath
\moveto(405.5,34.99998262)
\curveto(430.5,9.99998262)(456.5,9.99998262)(456.5,9.99998262)
}
}
{
\pscustom[linewidth=1,linecolor=black,fillstyle=solid,fillcolor=black]%arrowhead
{
\newpath
\moveto(412.57106781,27.92891481)
\lineto(418.22792206,27.92891481)
\lineto(405.5,34.99998262)
\lineto(412.57106781,22.27206056)
\lineto(412.57106781,27.92891481)
\closepath
}
}
{
\pscustom[linewidth=1,linecolor=black]
{
\newpath
\moveto(409.5,39.99998262)
\curveto(439.28423,35.16178262)(454.88242,21.46318262)(455.5,20.99998262)
}
}
{
\pscustom[linewidth=1,linecolor=black,fillstyle=solid,fillcolor=black]%arrowhead
{
\newpath
\moveto(447.50009327,27.00010697)
\lineto(441.90008083,26.20019402)
\lineto(455.5,20.99998262)
\lineto(446.70018032,32.60011941)
\lineto(447.50009327,27.00010697)
\closepath
}
}
{
\put(25,65){$\mathcal{G}_{\alpha}$}
\put(90,65){$\mathcal{G}_{\beta}$}
\put(230,65){$\mathcal{G}_{\gamma}$}
\put(375,65){$\mathcal{G}_{\delta}$}
}
\end{pspicture}
\caption{\label{threetreegraphspic1}}
\end{figure}
\end{theorem}

\begin{remark}
This result for $|\Tr(\mathcal{G},v)|\in\{1,2\}$ is already known (\cite{MR2001624} Theorem 3(5)).
\end{remark}

\begin{proof}
First suppose $\sphere\setminus\mathcal{G}$ is a single region $r$. Then $\mathcal{G}$ is a tree. Since $\mathcal{G}$ is finite but has no 1--valent vertices, $\mathcal{G}=\mathcal{G}_{\alpha}$.

Assume there are at least two regions of $\sphere\setminus\mathcal{G}$. Let $r$ be one such region. Then $\partial r$ is a topological circle, as otherwise $r$ would meet itself at a vertex of $\mathcal{G}$, contradicting that $\mathcal{G}$ is prime.
Note that every region of $\sphere\setminus\mathcal{G}$ has at least two sides. Suppose every region has exactly two sides. Then by considering the Euler characteristic $\chi$ of $\sphere$ and of $\mathcal{G}$ we find that $|\V(\mathcal{G})|=2$. Since $|\Tr(\mathcal{G},v)|<4$ for a vertex $v$, we see that either $\mathcal{G}=\mathcal{G}_{\beta}$ or $\mathcal{G}=\mathcal{G}_{\gamma}$.

This leave the case where $\sphere\setminus\mathcal{G}$ has a region $r_0$ with at least 3 sides. Since $\partial r_0$ is a circle, considering $\chi$ shows there is a second region $r_1$ with at least 3 sides.

We now consider whether certain digraphs can occur as subgraphs of $\mathcal{G}$. Note that, by (G4), the boundary of every region of $\sphere\setminus\mathcal{G}$ is a cycle, so $\mathcal{G}$ is \ocon{\mathcal{O}}. Thus Lemma \ref{extendtreeslemma} applies.

\begin{lemma}\label{nothreetwolemma}
If the graph shown in Figure \ref{threetreegraphspic2} is a subgraph of $\mathcal{G}$, where $v$ and $w$ may coincide or $v'$ and $w'$ may coincide, but not both, then $|\Tr(\mathcal{G},v)|\geq 4$. Here a dashed line denotes a directed simple path.
\begin{figure}[htbp]
\centering
%LaTeX with PSTricks extensions
%%Creator: 0.46
%%Please note this file requires PSTricks extensions
\psset{xunit=.35pt,yunit=.35pt,runit=.35pt}
\begin{pspicture}(210,220)
{
\pscustom[linewidth=6,linecolor=black,fillstyle=solid,fillcolor=black]%circle
{
\newpath
\moveto(50,30)
\curveto(50,24.48)(45.52,20)(40,20)
\curveto(34.48,20)(30,24.48)(30,30)
\curveto(30,35.52)(34.48,40)(40,40)
\curveto(45.52,40)(50,35.52)(50,30)
\closepath
}
}
{
\pscustom[linewidth=6,linecolor=black,fillstyle=solid,fillcolor=black]%circle
{
\newpath
\moveto(50,190)
\curveto(50,184.48)(45.52,180)(40,180)
\curveto(34.48,180)(30,184.48)(30,190)
\curveto(30,195.52)(34.48,200)(40,200)
\curveto(45.52,200)(50,195.52)(50,190)
\closepath
}
}
{
\pscustom[linewidth=6,linecolor=black,fillstyle=solid,fillcolor=black]%circle
{
\newpath
\moveto(179.999999,30)
\curveto(179.999999,24.48)(175.519999,20)(169.999999,20)
\curveto(164.479999,20)(159.999999,24.48)(159.999999,30)
\curveto(159.999999,35.52)(164.479999,40)(169.999999,40)
\curveto(175.519999,40)(179.999999,35.52)(179.999999,30)
\closepath
}
}
{
\pscustom[linewidth=6,linecolor=black,fillstyle=solid,fillcolor=black]%circle
{
\newpath
\moveto(179.999999,190)
\curveto(179.999999,184.48)(175.519999,180)(169.999999,180)
\curveto(164.479999,180)(159.999999,184.48)(159.999999,190)
\curveto(159.999999,195.52)(164.479999,200)(169.999999,200)
\curveto(175.519999,200)(179.999999,195.52)(179.999999,190)
\closepath
}
}
{
\pscustom[linewidth=2,linecolor=black,linestyle=dashed,dash=6 6]
{
\newpath
\moveto(40,29.99998262)
\lineto(40,180.00000262)
\lineto(40,190.00000262)
}
}
{
\pscustom[linewidth=2,linecolor=black,linestyle=dashed,dash=6 6]
{
\newpath
\moveto(170,190.00000262)
\lineto(170,29.99998262)
}
}
{
\pscustom[linewidth=4,linecolor=black,fillstyle=solid,fillcolor=black]%arrowhead
{
\newpath
\moveto(40,105.00000262)
\lineto(46,99.00000262)
\lineto(40,120.00000262)
\lineto(34,99.00000262)
\lineto(40,105.00000262)
\closepath
}
}
{
\pscustom[linewidth=4,linecolor=black,fillstyle=solid,fillcolor=black]%arrowhead
{
\newpath
\moveto(170,115.00000262)
\lineto(164,121.00000262)
\lineto(170,100.00000262)
\lineto(176,121.00000262)
\lineto(170,115.00000262)
\closepath
}
}
{
\pscustom[linewidth=2,linecolor=black]
{
\newpath
\moveto(40,190.00000262)
\lineto(170,190.00000262)
}
}
{
\pscustom[linewidth=2,linecolor=black]
{
\newpath
\moveto(40,190.00000262)
\curveto(90,220.00000262)(120,220.00000262)(170,190.00000262)
}
}
{
\pscustom[linewidth=2,linecolor=black]
{
\newpath
\moveto(40,190.00000262)
\curveto(91.11352,158.73716262)(120,160.00000262)(170,190.00000262)
}
}
{
\pscustom[linewidth=2,linecolor=black]
{
\newpath
\moveto(40,29.99998262)
\curveto(90,60.00000262)(120,60.00000262)(170,29.99998262)
}
}
{
\pscustom[linewidth=2,linecolor=black]
{
\newpath
\moveto(40,29.99998262)
\curveto(90,-0.00001738)(120,-0.00001738)(170,29.99998262)
}
}
{
\pscustom[linewidth=4,linecolor=black,fillstyle=solid,fillcolor=black]%arrowhead
{
\newpath
\moveto(99,212.00000262)
\lineto(93,206.00000262)
\lineto(114,212.00000262)
\lineto(93,218.00000262)
\lineto(99,212.00000262)
\closepath
}
}
{
\pscustom[linewidth=4,linecolor=black,fillstyle=solid,fillcolor=black]%arrowhead
{
\newpath
\moveto(99,168.00000262)
\lineto(93,162.00000262)
\lineto(114,168.00000262)
\lineto(93,174.00000262)
\lineto(99,168.00000262)
\closepath
}
}
{
\pscustom[linewidth=4,linecolor=black,fillstyle=solid,fillcolor=black]%arrowhead
{
\newpath
\moveto(108,190.00000262)
\lineto(114,196.00000262)
\lineto(93,190.00000262)
\lineto(114,184.00000262)
\lineto(108,190.00000262)
\closepath
}
}
{
\pscustom[linewidth=4,linecolor=black,fillstyle=solid,fillcolor=black]%arrowhead
{
\newpath
\moveto(108,50.99998262)
\lineto(114,56.99998262)
\lineto(93,50.99998262)
\lineto(114,44.99998262)
\lineto(108,50.99998262)
\closepath
}
}
{
\pscustom[linewidth=4,linecolor=black,fillstyle=solid,fillcolor=black]%arrowhead
{
\newpath
\moveto(99,7.99998262)
\lineto(93,1.99998262)
\lineto(114,7.99998262)
\lineto(93,13.99998262)
\lineto(99,7.99998262)
\closepath
}
}
{
\put(0,185){$w$}
\put(5,25){$v$}
\put(190,25){$v'$}
\put(190,185){$w'$}
}
\end{pspicture}
\caption{\label{threetreegraphspic2}}
\end{figure}
\end{lemma}

\begin{lemma}\label{notwotwotwolemma}
If the graph shown in Figure \ref{threetreegraphspic3} is a subgraph of $\mathcal{G}$, then $|\Tr(\mathcal{G},v)|\geq 4$.
\begin{figure}[htbp]
\centering
%LaTeX with PSTricks extensions
%%Creator: 0.46
%%Please note this file requires PSTricks extensions
\psset{xunit=.35pt,yunit=.35pt,runit=.35pt}
\begin{pspicture}(410,125)
{
\pscustom[linewidth=8,linecolor=black,fillstyle=solid,fillcolor=black]%circle
{
\newpath
\moveto(30,94.999997)
\curveto(30,89.479997)(25.52,84.999997)(20,84.999997)
\curveto(14.48,84.999997)(10,89.479997)(10,94.999997)
\curveto(10,100.519997)(14.48,104.999997)(20,104.999997)
\curveto(25.52,104.999997)(30,100.519997)(30,94.999997)
\closepath
}
}
{
\pscustom[linewidth=8,linecolor=black,fillstyle=solid,fillcolor=black]%circle
{
\newpath
\moveto(159.999999,94.999997)
\curveto(159.999999,89.479997)(155.519999,84.999997)(149.999999,84.999997)
\curveto(144.479999,84.999997)(139.999999,89.479997)(139.999999,94.999997)
\curveto(139.999999,100.519997)(144.479999,104.999997)(149.999999,104.999997)
\curveto(155.519999,104.999997)(159.999999,100.519997)(159.999999,94.999997)
\closepath
}
}
{
\pscustom[linewidth=8,linecolor=black,fillstyle=solid,fillcolor=black]%circle
{
\newpath
\moveto(287,94.999997)
\curveto(287,89.479997)(282.52,84.999997)(277,84.999997)
\curveto(271.48,84.999997)(267,89.479997)(267,94.999997)
\curveto(267,100.519997)(271.48,104.999997)(277,104.999997)
\curveto(282.52,104.999997)(287,100.519997)(287,94.999997)
\closepath
}
}
{
\pscustom[linewidth=8,linecolor=black,fillstyle=solid,fillcolor=black]%circle
{
\newpath
\moveto(415,94.999993)
\curveto(415,89.479993)(410.52,84.999993)(405,84.999993)
\curveto(399.48,84.999993)(395,89.479993)(395,94.999993)
\curveto(395,100.519993)(399.48,104.999993)(405,104.999993)
\curveto(410.52,104.999993)(415,100.519993)(415,94.999993)
\closepath
}
}
{
\pscustom[linewidth=2,linecolor=black]
{
\newpath
\moveto(20,95.00000262)
\curveto(70,125.00000262)(100,125.00000262)(150,95.00000262)
}
}
{
\pscustom[linewidth=2,linecolor=black]
{
\newpath
\moveto(20,95.00000262)
\curveto(71.11352,63.73716262)(100,65.00000262)(150,95.00000262)
}
}
{
\pscustom[linewidth=2,linecolor=black]
{
\newpath
\moveto(147,95.00000262)
\curveto(197,125.00000262)(227,125.00000262)(277,95.00000262)
}
}
{
\pscustom[linewidth=2,linecolor=black]
{
\newpath
\moveto(147,95.00000262)
\curveto(198.11352,63.73716262)(227,65.00000262)(277,95.00000262)
}
}
{
\pscustom[linewidth=2,linecolor=black]
{
\newpath
\moveto(275,95.00000262)
\curveto(325,125.00000262)(355,125.00000262)(405,95.00000262)
}
}
{
\pscustom[linewidth=2,linecolor=black]
{
\newpath
\moveto(275,95.00000262)
\curveto(326.11352,63.73716262)(355,65.00000262)(405,95.00000262)
}
}
{
\pscustom[linewidth=2,linecolor=black,linestyle=dashed,dash=6 6]
{
\newpath
\moveto(410,95.00000262)
\curveto(390,-5.00001738)(40,-5.00001738)(20,95.00000262)
}
}
{
\pscustom[linewidth=4,linecolor=black,fillstyle=solid,fillcolor=black]%arrowhead
{
\newpath
\moveto(215,117.00000262)
\lineto(221,123.00000262)
\lineto(200,117.00000262)
\lineto(221,111.00000262)
\lineto(215,117.00000262)
\closepath
}
}
{
\pscustom[linewidth=4,linecolor=black,fillstyle=solid,fillcolor=black]%arrowhead
{
\newpath
\moveto(206,73.00000262)
\lineto(200,67.00000262)
\lineto(221,73.00000262)
\lineto(200,79.00000262)
\lineto(206,73.00000262)
\closepath
}
}
{
\pscustom[linewidth=4,linecolor=black,fillstyle=solid,fillcolor=black]%arrowhead
{
\newpath
\moveto(343,117.00000262)
\lineto(349,123.00000262)
\lineto(328,117.00000262)
\lineto(349,111.00000262)
\lineto(343,117.00000262)
\closepath
}
}
{
\pscustom[linewidth=4,linecolor=black,fillstyle=solid,fillcolor=black]%arrowhead
{
\newpath
\moveto(334,73.00000262)
\lineto(328,67.00000262)
\lineto(349,73.00000262)
\lineto(328,79.00000262)
\lineto(334,73.00000262)
\closepath
}
}
{
\pscustom[linewidth=4,linecolor=black,fillstyle=solid,fillcolor=black]%arrowhead
{
\newpath
\moveto(215,19.99998262)
\lineto(221,25.99998262)
\lineto(200,19.99998262)
\lineto(221,13.99998262)
\lineto(215,19.99998262)
\closepath
}
}
{
\pscustom[linewidth=4,linecolor=black,fillstyle=solid,fillcolor=black]%arrowhead
{
\newpath
\moveto(88,117.00000262)
\lineto(94,123.00000262)
\lineto(73,117.00000262)
\lineto(94,111.00000262)
\lineto(88,117.00000262)
\closepath
}
}
{
\pscustom[linewidth=4,linecolor=black,fillstyle=solid,fillcolor=black]%arrowhead
{
\newpath
\moveto(79,73.00000262)
\lineto(73,67.00000262)
\lineto(94,73.00000262)
\lineto(73,79.00000262)
\lineto(79,73.00000262)
\closepath
}
}
{
\put(425,90){$v$}
}
\end{pspicture}
\caption{\label{threetreegraphspic3}}
\end{figure}
\end{lemma}

\begin{lemma}
Suppose exactly two regions of $\sphere\setminus\mathcal{G}$ have 3 or more sides. Then, up to reflection, $\mathcal{G}=\mathcal{G}_{\delta}$.
\end{lemma}
\begin{proof}
Collapse each bigon region of $\sphere\setminus\mathcal{G}$ to a line, giving a graph $\mathcal{G}'$. Since this only leaves two regions, $\mathcal{G}'$ is a topological circle. Number the vertices $\V(\mathcal{G})=\V(\mathcal{G}')$ as $v_1,\cdots,v_n$ around this circle. Then $n\geq 3$. For $1\leq i\leq n$, let $J_i$ be the number of (unoriented) edges in $\mathcal{G}$ joining $v_i$ to $v_{i+1}$ (where $v_{n+1}=v_1$).
By (G4) and (G5), $v_i+v_{i+1}$ is even and at least 4 for each $i$. 
Suppose $v_1\geq 3$. Then, by Lemma \ref{nothreetwolemma}, $v_i=1$ for all $i>1$. Thus $v_2+v_3=2$, contradicting that $v_2+v_3>2$. Hence $v_i\leq 2$ for all $i$. By Lemma \ref{notwotwotwolemma}, this means $n=3$.
\end{proof}

It now suffices to show that at most two regions of $\sphere\setminus\mathcal{G}$ have three or more sides. We therefore assume otherwise.

\begin{claim}
Up to relabelling the regions of $\sphere\setminus\mathcal{G}$, $r_0$ shares an edge with $r_1$.
\end{claim}
\begin{proof}
Suppose otherwise. Then $r_0$ meets a bigon along every edge. By Lemma \ref{notwotwotwolemma}, $\partial r_0$ is a triangle. Since at least three regions of $\sphere\setminus\mathcal{G}$ have three or more sides and $\mathcal{G}$ is prime, $\mathcal{G}$ contains the graph shown in Figure \ref{threetreegraphspic4}a and a directed path disjoint from this graph connecting two distinct vertices of $\partial r_0$. It must therefore contain one of the graphs in Figure \ref{threetreegraphspic4}b. In either case, $|\Tr(\mathcal{G},v)|\geq 4$.
\begin{figure}[htbp]
\centering
(a)
%LaTeX with PSTricks extensions
%%Creator: 0.46
%%Please note this file requires PSTricks extensions
\psset{xunit=.35pt,yunit=.35pt,runit=.35pt}
\begin{pspicture}(200,130)
{
\pscustom[linewidth=2,linecolor=black]
{
\newpath
\moveto(170,85)
\curveto(170,49.12)(138.64,20)(100,20)
\curveto(61.36,20)(30,49.12)(30,85)
\curveto(30,120.88)(61.36,150)(100,150)
\curveto(138.64,150)(170,120.88)(170,85)
\closepath
}
}
{
\pscustom[linewidth=2,linecolor=black]
{
\newpath
\moveto(160.71951381,51.65655243)
\curveto(122.33504487,70.89123799)(99.81127877,107.55254662)(100.00107481,150.48626952)
}
}
{
\pscustom[linewidth=2,linecolor=black]
{
\newpath
\moveto(100,150)
\curveto(100,107.3720946)(77.66811729,71.0605325)(39.62453823,51.82971322)
}
}
{
\pscustom[linewidth=2,linecolor=black]
{
\newpath
\moveto(40.37408272,51.43781687)
\curveto(77.49581254,75.38274836)(122.15644523,75.44945042)(159.34953431,51.61551046)
}
}
{
\pscustom[linewidth=4,linecolor=black,fillstyle=solid,fillcolor=black]%arrowhead
{
\newpath
\moveto(96.13606076,70.64399249)
\lineto(91.53288912,67.35601274)
\lineto(106,69.00000262)
\lineto(92.84808102,75.24716413)
\lineto(96.13606076,70.64399249)
\closepath
}
}
{
\pscustom[linewidth=4,linecolor=black,fillstyle=solid,fillcolor=black]%arrowhead
{
\newpath
\moveto(100,20.99998262)
\lineto(104,24.99998262)
\lineto(90,20.99998262)
\lineto(104,16.99998262)
\lineto(100,20.99998262)
\closepath
}
}
{
\pscustom[linewidth=4,linecolor=black,fillstyle=solid,fillcolor=black]%arrowhead
{
\newpath
\moveto(35.45299804,111.67949967)
\lineto(36.56239843,106.13249771)
\lineto(41,120.00000262)
\lineto(29.90599608,110.57009928)
\lineto(35.45299804,111.67949967)
\closepath
}
}
{
\pscustom[linewidth=4,linecolor=black,fillstyle=solid,fillcolor=black]%arrowhead
{
\newpath
\moveto(160.52786405,117.94427453)
\lineto(155.1613009,119.73312891)
\lineto(165,109.00000262)
\lineto(162.31671843,123.31083767)
\lineto(160.52786405,117.94427453)
\closepath
}
}
{
\pscustom[linewidth=4,linecolor=black,fillstyle=solid,fillcolor=black]%arrowhead
{
\newpath
\moveto(88.93919299,102.19145292)
\lineto(86.83829006,107.44371023)
\lineto(85,93.00000262)
\lineto(94.1914503,104.29235584)
\lineto(88.93919299,102.19145292)
\closepath
}
}
{
\pscustom[linewidth=4,linecolor=black,fillstyle=solid,fillcolor=black]%arrowhead
{
\newpath
\moveto(117.54700196,90.67949967)
\lineto(123.09400392,89.57009928)
\lineto(112,99.00000262)
\lineto(116.43760157,85.13249771)
\lineto(117.54700196,90.67949967)
\closepath
}
}
{
\pscustom[linewidth=12,linecolor=black,fillstyle=solid,fillcolor=black]%circle
{
\newpath
\moveto(45,50)
\curveto(45,47.24)(42.76,45)(40,45)
\curveto(37.24,45)(35,47.24)(35,50)
\curveto(35,52.76)(37.24,55)(40,55)
\curveto(42.76,55)(45,52.76)(45,50)
\closepath
}
}
{
\pscustom[linewidth=12,linecolor=black,fillstyle=solid,fillcolor=black]%circle
{
\newpath
\moveto(105,150)
\curveto(105,147.24)(102.76,145)(100,145)
\curveto(97.24,145)(95,147.24)(95,150)
\curveto(95,152.76)(97.24,155)(100,155)
\curveto(102.76,155)(105,152.76)(105,150)
\closepath
}
}
{
\pscustom[linewidth=12,linecolor=black,fillstyle=solid,fillcolor=black]%circle
{
\newpath
\moveto(165,50)
\curveto(165,47.24)(162.76,45)(160,45)
\curveto(157.24,45)(155,47.24)(155,50)
\curveto(155,52.76)(157.24,55)(160,55)
\curveto(162.76,55)(165,52.76)(165,50)
\closepath
}
}
\end{pspicture}
\quad\quad
(b)\;\;\;
%LaTeX with PSTricks extensions
%%Creator: 0.46
%%Please note this file requires PSTricks extensions
\psset{xunit=.35pt,yunit=.35pt,runit=.35pt}
\begin{pspicture}(400,150)
{
\pscustom[linewidth=2,linecolor=black]%circle
{
\newpath
\moveto(145.594004,73.500035)
\curveto(145.594004,37.620035)(114.234004,8.500035)(75.594004,8.500035)
\curveto(36.954004,8.500035)(5.594004,37.620035)(5.594004,73.500035)
\curveto(5.594004,109.380035)(36.954004,138.500035)(75.594004,138.500035)
\curveto(114.234004,138.500035)(145.594004,109.380035)(145.594004,73.500035)
\closepath
}
}
{
\pscustom[linewidth=2,linecolor=black]%circle
{
\newpath
\moveto(136.31351781,40.15658643)
\curveto(97.92904887,59.39127199)(75.40528277,96.05258062)(75.59507881,138.98630352)
}
}
{
\pscustom[linewidth=2,linecolor=black]%circle
{
\newpath
\moveto(75.594,138.500034)
\curveto(75.594,95.8721286)(53.26211729,59.5605665)(15.21853823,40.32974722)
}
}
{
\pscustom[linewidth=2,linecolor=black]%circle
{
\newpath
\moveto(15.96808672,39.93785687)
\curveto(53.08981654,63.88278836)(97.75044923,63.94949042)(134.94353831,40.11555046)
}
}
{
\pscustom[linewidth=4,linecolor=black,fillstyle=solid,fillcolor=black]%arrowhead
{
\newpath
\moveto(71.73006476,59.14403249)
\lineto(67.12689312,55.85605274)
\lineto(81.594004,57.50004262)
\lineto(68.44208502,63.74720413)
\lineto(71.73006476,59.14403249)
\closepath
}
}
{
\pscustom[linewidth=4,linecolor=black,fillstyle=solid,fillcolor=black]%arrowhead
{
\newpath
\moveto(75.594004,9.49998262)
\lineto(79.594004,13.49998262)
\lineto(65.594004,9.49998262)
\lineto(79.594004,5.49998262)
\lineto(75.594004,9.49998262)
\closepath
}
}
{
\pscustom[linewidth=4,linecolor=black,fillstyle=solid,fillcolor=black]%arrowhead
{
\newpath
\moveto(11.04700204,100.17953967)
\lineto(12.15640243,94.63253771)
\lineto(16.594004,108.50004262)
\lineto(5.50000008,99.07013928)
\lineto(11.04700204,100.17953967)
\closepath
}
}
{
\pscustom[linewidth=4,linecolor=black,fillstyle=solid,fillcolor=black]%arrowhead
{
\newpath
\moveto(136.12186405,106.44431453)
\lineto(130.7553009,108.23316891)
\lineto(140.594,97.50004262)
\lineto(137.91071843,111.81087767)
\lineto(136.12186405,106.44431453)
\closepath
}
}
{
\pscustom[linewidth=4,linecolor=black,fillstyle=solid,fillcolor=black]%arrowhead
{
\newpath
\moveto(64.53319699,90.69149292)
\lineto(62.43229406,95.94375023)
\lineto(60.594004,81.50004262)
\lineto(69.7854543,92.79239584)
\lineto(64.53319699,90.69149292)
\closepath
}
}
{
\pscustom[linewidth=4,linecolor=black,fillstyle=solid,fillcolor=black]%arrowhead
{
\newpath
\moveto(93.14100596,79.17953967)
\lineto(98.68800792,78.07013928)
\lineto(87.594004,87.50004262)
\lineto(92.03160557,73.63253771)
\lineto(93.14100596,79.17953967)
\closepath
}
}
{
\pscustom[linewidth=2,linecolor=black,linestyle=dashed,dash=6 6]
{
\newpath
\moveto(75.67849764,138.21447928)
\curveto(98.43070099,161.90383848)(136.12240419,162.66442898)(159.8117637,139.91222591)
\curveto(183.50112322,117.16002285)(184.26171372,79.46832015)(161.50951036,55.77896094)
\curveto(153.94434672,47.90219068)(145.71165528,42.91758803)(135.2256213,39.86499843)
}
}
{
\pscustom[linewidth=4,linecolor=black,fillstyle=solid,fillcolor=black]%arrowhead
{
\newpath
\moveto(167.12186405,132.44431453)
\lineto(161.7553009,134.23316891)
\lineto(171.594,123.50004262)
\lineto(168.91071843,137.81087767)
\lineto(167.12186405,132.44431453)
\closepath
}
}
{
\pscustom[linewidth=2,linecolor=black]%circle
{
\newpath
\moveto(379.87449,73.500065)
\curveto(379.87449,37.620065)(348.51449,8.500065)(309.87449,8.500065)
\curveto(271.23449,8.500065)(239.87449,37.620065)(239.87449,73.500065)
\curveto(239.87449,109.380065)(271.23449,138.500065)(309.87449,138.500065)
\curveto(348.51449,138.500065)(379.87449,109.380065)(379.87449,73.500065)
\closepath
}
}
{
\pscustom[linewidth=2,linecolor=black]%circle
{
\newpath
\moveto(370.59400381,40.15661743)
\curveto(332.20953487,59.39130299)(309.68576877,96.05261162)(309.87556481,138.98633452)
}
}
{
\pscustom[linewidth=2,linecolor=black]%circle
{
\newpath
\moveto(309.874494,138.500065)
\curveto(309.874494,95.8721596)(287.54261129,59.5605975)(249.49903223,40.32977822)
}
}
{
\pscustom[linewidth=2,linecolor=black]%circle
{
\newpath
\moveto(250.24857272,39.93787687)
\curveto(287.37030254,63.88280836)(332.03093523,63.94951042)(369.22402431,40.11557046)
}
}
{
\pscustom[linewidth=4,linecolor=black,fillstyle=solid,fillcolor=black]%arrowhead
{
\newpath
\moveto(306.01055076,59.14403249)
\lineto(301.40737912,55.85605274)
\lineto(315.87449,57.50004262)
\lineto(302.72257102,63.74720413)
\lineto(306.01055076,59.14403249)
\closepath
}
}
{
\pscustom[linewidth=4,linecolor=black,fillstyle=solid,fillcolor=black]%arrowhead
{
\newpath
\moveto(309.87449,9.49998262)
\lineto(313.87449,13.49998262)
\lineto(299.87449,9.49998262)
\lineto(313.87449,5.49998262)
\lineto(309.87449,9.49998262)
\closepath
}
}
{
\pscustom[linewidth=4,linecolor=black,fillstyle=solid,fillcolor=black]%arrowhead
{
\newpath
\moveto(245.32750084,100.17955114)
\lineto(246.43690976,94.63255088)
\lineto(250.87449,108.50006262)
\lineto(239.78050058,99.07014221)
\lineto(245.32750084,100.17955114)
\closepath
}
}
{
\pscustom[linewidth=4,linecolor=black,fillstyle=solid,fillcolor=black]%arrowhead
{
\newpath
\moveto(370.40235405,106.44431453)
\lineto(365.0357909,108.23316891)
\lineto(374.87449,97.50004262)
\lineto(372.19120843,111.81087767)
\lineto(370.40235405,106.44431453)
\closepath
}
}
{
\pscustom[linewidth=4,linecolor=black,fillstyle=solid,fillcolor=black]%arrowhead
{
\newpath
\moveto(298.81368299,90.69149292)
\lineto(296.71278006,95.94375023)
\lineto(294.87449,81.50004262)
\lineto(304.0659403,92.79239584)
\lineto(298.81368299,90.69149292)
\closepath
}
}
{
\pscustom[linewidth=4,linecolor=black,fillstyle=solid,fillcolor=black]%arrowhead
{
\newpath
\moveto(327.42149196,79.17953967)
\lineto(332.96849392,78.07013928)
\lineto(321.87449,87.50004262)
\lineto(326.31209157,73.63253771)
\lineto(327.42149196,79.17953967)
\closepath
}
}
{
\pscustom[linewidth=2,linecolor=black,linestyle=dashed,dash=6 6]
{
\newpath
\moveto(309.95898364,138.21451028)
\curveto(332.71118699,161.90386948)(370.40289019,162.66445998)(394.0922497,139.91225691)
\curveto(417.78160922,117.16005385)(418.54219972,79.46835115)(395.78999636,55.77899194)
\curveto(388.22483272,47.90222168)(379.99214128,42.91761903)(369.5061073,39.86502943)
}
}
{
\pscustom[linewidth=4,linecolor=black,fillstyle=solid,fillcolor=black]%arrowhead
{
\newpath
\moveto(402.20772702,129.94279017)
\lineto(407.8237468,129.26432)
\lineto(396.03579,137.81090262)
\lineto(401.52925685,124.32677039)
\lineto(402.20772702,129.94279017)
\closepath
}
}
{
\pscustom[linewidth=12,linecolor=black,fillstyle=solid,fillcolor=black]%circle
{
\newpath
\moveto(315.594004,140)
\curveto(315.594004,137.24)(313.354004,135)(310.594004,135)
\curveto(307.834004,135)(305.594004,137.24)(305.594004,140)
\curveto(305.594004,142.76)(307.834004,145)(310.594004,145)
\curveto(313.354004,145)(315.594004,142.76)(315.594004,140)
\closepath
}
}
{
\pscustom[linewidth=12,linecolor=black,fillstyle=solid,fillcolor=black]%circle
{
\newpath
\moveto(375.5940039,40)
\curveto(375.5940039,37.24)(373.3540039,35)(370.5940039,35)
\curveto(367.8340039,35)(365.5940039,37.24)(365.5940039,40)
\curveto(365.5940039,42.76)(367.8340039,45)(370.5940039,45)
\curveto(373.3540039,45)(375.5940039,42.76)(375.5940039,40)
\closepath
}
}
{
\pscustom[linewidth=12,linecolor=black,fillstyle=solid,fillcolor=black]%circle
{
\newpath
\moveto(255.594,40)
\curveto(255.594,37.24)(253.354,35)(250.594,35)
\curveto(247.834,35)(245.594,37.24)(245.594,40)
\curveto(245.594,42.76)(247.834,45)(250.594,45)
\curveto(253.354,45)(255.594,42.76)(255.594,40)
\closepath
}
}
{
\pscustom[linewidth=12,linecolor=black,fillstyle=solid,fillcolor=black]%circle
{
\newpath
\moveto(80.594,140)
\curveto(80.594,137.24)(78.354,135)(75.594,135)
\curveto(72.834,135)(70.594,137.24)(70.594,140)
\curveto(70.594,142.76)(72.834,145)(75.594,145)
\curveto(78.354,145)(80.594,142.76)(80.594,140)
\closepath
}
}
{
\pscustom[linewidth=12,linecolor=black,fillstyle=solid,fillcolor=black]%circle
{
\newpath
\moveto(140.594,40)
\curveto(140.594,37.24)(138.354,35)(135.594,35)
\curveto(132.834,35)(130.594,37.24)(130.594,40)
\curveto(130.594,42.76)(132.834,45)(135.594,45)
\curveto(138.354,45)(140.594,42.76)(140.594,40)
\closepath
}
}
{
\pscustom[linewidth=12,linecolor=black,fillstyle=solid,fillcolor=black]%circle
{
\newpath
\moveto(20.594,40)
\curveto(20.594,37.24)(18.354,35)(15.594,35)
\curveto(12.834,35)(10.594,37.24)(10.594,40)
\curveto(10.594,42.76)(12.834,45)(15.594,45)
\curveto(18.354,45)(20.594,42.76)(20.594,40)
\closepath
}
}
{
\put(60,160){$v$}
\put(385,20){$v$}
}
\end{pspicture}
\caption{\label{threetreegraphspic4}}
\end{figure}
\end{proof}

\begin{claim}\label{meetoneedgeclaim}
The regions $r_0$ and $r_1$ share exactly one edge $e_1$ and meet at no vertices other than the endpoints of $e_1$.
\end{claim}
\begin{proof}
First suppose that $\partial r_o\cap\partial r_1$ has at least two components. Then $\mathcal{G}$ contains the graph shown in Figure \ref{threetreegraphspic5}, where $v'$ and $w'$ may coincide. Then $|\Tr(\mathcal{G},v)|\geq 4$. Thus $\partial r_o\cap\partial r_1$ has only one component.
\begin{figure}[htbp]
\centering
%LaTeX with PSTricks extensions
%%Creator: 0.46
%%Please note this file requires PSTricks extensions
\psset{xunit=.5pt,yunit=.5pt,runit=.5pt}
\begin{pspicture}(200,145)
{
\pscustom[linewidth=1.5,linecolor=black,linestyle=dashed,dash=4 4]
{
\newpath
\moveto(155.02988397,67.78057166)
\curveto(155.02900425,42.94057168)(132.62829028,22.78121427)(105.02829029,22.78200602)
\curveto(77.42829031,22.78279778)(55.0290043,42.94344035)(55.02988403,67.78344034)
\curveto(55.03076375,92.62344032)(77.43147772,112.78279773)(105.03147771,112.78200598)
\curveto(132.63123379,112.78121423)(155.03038059,92.62097681)(155.02988399,67.78119632)
}
}
{
\pscustom[linewidth=1.5,linecolor=black,linestyle=dashed,dash=4 4]
{
\newpath
\moveto(80.16853847,28.88317878)
\curveto(56.31059111,16.39429326)(25.68550034,23.68506076)(11.80896087,45.15721338)
\curveto(-2.0675786,66.629366)(6.03327417,94.19194769)(29.89122153,106.68083322)
\curveto(45.90438649,115.06322174)(64.19170394,115.05753042)(80.19842373,106.66517677)
}
}
{
\pscustom[linewidth=1.5,linecolor=black,linestyle=dashed,dash=4 4]
{
\newpath
\moveto(130.02988153,107.78199245)
\curveto(153.88782889,120.27087798)(184.51291966,112.98011048)(198.38945913,91.50795786)
\curveto(212.2659986,70.03580523)(204.16514583,42.47322354)(180.30719847,29.98433802)
\curveto(164.29403351,21.6019495)(146.00671606,21.60764082)(129.99999627,29.99999446)
}
}
{
\pscustom[linewidth=8,linecolor=black,fillstyle=solid,fillcolor=black]%circle
{
\newpath
\moveto(85.029884,27.782006)
\curveto(85.029884,25.022006)(82.789884,22.782006)(80.029884,22.782006)
\curveto(77.269884,22.782006)(75.029884,25.022006)(75.029884,27.782006)
\curveto(75.029884,30.542006)(77.269884,32.782006)(80.029884,32.782006)
\curveto(82.789884,32.782006)(85.029884,30.542006)(85.029884,27.782006)
\closepath
}
}
{
\pscustom[linewidth=8,linecolor=black,fillstyle=solid,fillcolor=black]%circle
{
\newpath
\moveto(85.029884,107.782006)
\curveto(85.029884,105.022006)(82.789884,102.782006)(80.029884,102.782006)
\curveto(77.269884,102.782006)(75.029884,105.022006)(75.029884,107.782006)
\curveto(75.029884,110.542006)(77.269884,112.782006)(80.029884,112.782006)
\curveto(82.789884,112.782006)(85.029884,110.542006)(85.029884,107.782006)
\closepath
}
}
{
\pscustom[linewidth=8,linecolor=black,fillstyle=solid,fillcolor=black]%circle
{
\newpath
\moveto(125.16854,108.88316523)
\curveto(125.16854,111.64316523)(127.40854,113.88316523)(130.16854,113.88316523)
\curveto(132.92854,113.88316523)(135.16854,111.64316523)(135.16854,108.88316523)
\curveto(135.16854,106.12316523)(132.92854,103.88316523)(130.16854,103.88316523)
\curveto(127.40854,103.88316523)(125.16854,106.12316523)(125.16854,108.88316523)
\closepath
}
}
{
\pscustom[linewidth=8,linecolor=black,fillstyle=solid,fillcolor=black]%circle
{
\newpath
\moveto(125.16854,28.88316523)
\curveto(125.16854,31.64316523)(127.40854,33.88316523)(130.16854,33.88316523)
\curveto(132.92854,33.88316523)(135.16854,31.64316523)(135.16854,28.88316523)
\curveto(135.16854,26.12316523)(132.92854,23.88316523)(130.16854,23.88316523)
\curveto(127.40854,23.88316523)(125.16854,26.12316523)(125.16854,28.88316523)
\closepath
}
}
{
\pscustom[linewidth=3,linecolor=black,fillstyle=solid,fillcolor=black]%arrowhead
{
\newpath
\moveto(5.0298837,62.78201262)
\lineto(9.0298837,58.78201262)
\lineto(5.0298837,72.78201262)
\lineto(1.0298837,58.78201262)
\lineto(5.0298837,62.78201262)
\closepath
}
}
{
\pscustom[linewidth=3,linecolor=black,fillstyle=solid,fillcolor=black]%arrowhead
{
\newpath
\moveto(55.029884,62.78201262)
\lineto(59.029884,58.78201262)
\lineto(55.029884,72.78201262)
\lineto(51.029884,58.78201262)
\lineto(55.029884,62.78201262)
\closepath
}
}
{
\pscustom[linewidth=3,linecolor=black,fillstyle=solid,fillcolor=black]%arrowhead
{
\newpath
\moveto(205.16854,73.88317262)
\lineto(201.16854,77.88317262)
\lineto(205.16854,63.88317262)
\lineto(209.16854,77.88317262)
\lineto(205.16854,73.88317262)
\closepath
}
}
{
\pscustom[linewidth=3,linecolor=black,fillstyle=solid,fillcolor=black]%arrowhead
{
\newpath
\moveto(155.16854,73.88317262)
\lineto(151.16854,77.88317262)
\lineto(155.16854,63.88317262)
\lineto(159.16854,77.88317262)
\lineto(155.16854,73.88317262)
\closepath
}
}
{
\pscustom[linewidth=3,linecolor=black,fillstyle=solid,fillcolor=black]%arrowhead
{
\newpath
\moveto(100.02988,112.78201262)
\lineto(96.02988,108.78201262)
\lineto(110.02988,112.78201262)
\lineto(96.02988,116.78201262)
\lineto(100.02988,112.78201262)
\closepath
}
}
{
\pscustom[linewidth=3,linecolor=black,fillstyle=solid,fillcolor=black]%arrowhead
{
\newpath
\moveto(110.02988,22.78198262)
\lineto(114.02988,26.78198262)
\lineto(100.02988,22.78198262)
\lineto(114.02988,18.78198262)
\lineto(110.02988,22.78198262)
\closepath
}
}
{
\put(75,5){$v$}
\put(125,5){$w$}
\put(75,122){$v'$}
\put(125,122){$w'$}
\put(100,60){$r_0$}
\put(0,110){$r_1$}
}
\end{pspicture}
\caption{\label{threetreegraphspic5}}
\end{figure}

Now suppose $r_0$ and $r_1$ share at least two consecutive edges (since $\partial r_0$ and $\partial r_1$ are circles, such edges must be consecutive both in $\partial r_0$ and in $\partial r_1$). Then the vertex between these edges had in-degree 1. This contradicts (G5).
\end{proof}

We now know $\mathcal{G}$ contains the digraph $\mathcal{G}_1$ shown in Figure \ref{threetreegraphspic6}.
\begin{figure}[htb]
\centering
%LaTeX with PSTricks extensions
%%Creator: 0.46
%%Please note this file requires PSTricks extensions
\psset{xunit=.45pt,yunit=.45pt,runit=.45pt}
\begin{pspicture}(175,190)
{
\pscustom[linewidth=2,linecolor=black,fillstyle=solid,fillcolor=black]%circle
{
\newpath
\moveto(50,80)
\curveto(50,74.48)(45.52,70)(40,70)
\curveto(34.48,70)(30,74.48)(30,80)
\curveto(30,85.52)(34.48,90)(40,90)
\curveto(45.52,90)(50,85.52)(50,80)
\closepath
}
}
{
\pscustom[linewidth=2,linecolor=black,fillstyle=solid,fillcolor=black]%circle
{
\newpath
\moveto(169.6041,80.812204)
\curveto(169.6041,75.292204)(165.1241,70.812204)(159.6041,70.812204)
\curveto(154.0841,70.812204)(149.6041,75.292204)(149.6041,80.812204)
\curveto(149.6041,86.332204)(154.0841,90.812204)(159.6041,90.812204)
\curveto(165.1241,90.812204)(169.6041,86.332204)(169.6041,80.812204)
\closepath
}
}
{
\pscustom[linewidth=2,linecolor=black,fillstyle=solid,fillcolor=black]%circle
{
\newpath
\moveto(110,165)
\curveto(110,159.48)(105.52,155)(100,155)
\curveto(94.48,155)(90,159.48)(90,165)
\curveto(90,170.52)(94.48,175)(100,175)
\curveto(105.52,175)(110,170.52)(110,165)
\closepath
}
}
{
\pscustom[linewidth=2,linecolor=black]
{
\newpath
\moveto(150,80.00000262)
\lineto(50,80.00000262)
}
}
{
\pscustom[linewidth=2,linecolor=black]
{
\newpath
\moveto(40,80.00000262)
\curveto(47.528112,116.78959262)(57.053882,152.01580262)(95,165.00000262)
}
}
{
\pscustom[linewidth=2,linecolor=black,linestyle=dashed,dash=4 4]
{
\newpath
\moveto(40,75.00000262)
\curveto(81.3853,-15.04451738)(121.30798,-10.01891738)(160,75.00000262)
}
}
{
\pscustom[linewidth=2,linecolor=black,linestyle=dashed,dash=4 4]
{
\newpath
\moveto(160,80.00000262)
\curveto(152.47189,116.78959262)(142.94612,152.01580262)(105,165.00000262)
}
}
{
\pscustom[linewidth=1.5,linecolor=black,fillstyle=solid,fillcolor=black]%arrowhead
{
\newpath
\moveto(100,80.00000262)
\lineto(106,86.00000262)
\lineto(85,80.00000262)
\lineto(106,74.00000262)
\lineto(100,80.00000262)
\closepath
}
}
{
\pscustom[linewidth=1.5,linecolor=black,fillstyle=solid,fillcolor=black]%arrowhead
{
\newpath
\moveto(95,9.99998262)
\lineto(89,3.99998262)
\lineto(110,9.99998262)
\lineto(89,15.99998262)
\lineto(95,9.99998262)
\closepath
}
}
{
\pscustom[linewidth=1.5,linecolor=black,fillstyle=solid,fillcolor=black]%arrowhead
{
\newpath
\moveto(48.29179607,116.58359475)
\lineto(50.97507764,108.53375003)
\lineto(55,130.00000262)
\lineto(40.24195135,113.90031318)
\lineto(48.29179607,116.58359475)
\closepath
}
}
{
\pscustom[linewidth=1.5,linecolor=black,fillstyle=solid,fillcolor=black]%arrowhead
{
\newpath
\moveto(145.29179607,124.41641048)
\lineto(137.24195135,127.09969206)
\lineto(152,111.00000262)
\lineto(147.97507764,132.4662552)
\lineto(145.29179607,124.41641048)
\closepath
}
}
{
\put(115,65){$e_1$}
\put(20,100){$e_2$}
\put(115,170){$v$}
\put(95,110){$r_0$}
\put(95,40){$r_1$}
\put(15,155){$\mathcal{G}_1$}
}
\end{pspicture}
\caption{\label{threetreegraphspic6}}
\end{figure}
Let $r_2$ be the region of $\sphere\setminus\mathcal{G}$ that meets $r_0$ along edge $e_2$. Since $\partial r_2$ is a cycle, part of $\partial r_2$ forms a directed path from $v$ to a vertex of $\mathcal{G}_1$. As $\mathcal{G}$ is prime, this path cannot end at $v$. Thus $\mathcal{G}$ contains one of the digraphs $\mathcal{G}_2$--$\mathcal{G}_5$ shown in Figure \ref{threetreegraphspic7}.
\begin{figure}[htb]
\centering
\input{pictexfiles/threetreegraphspic7}
\caption{\label{threetreegraphspic7}}
\end{figure}
Note that $|\Tr(\mathcal{G}_2,v')|=|\Tr(\mathcal{G}_4,v')|=4$, so neither of these cases can occur. In addition, $\mathcal{G}_3$ cannot arise, as otherwise $r_0$ and $r_2$ would contradict Claim \ref{meetoneedgeclaim}.

Thus $\mathcal{G}$ contains $\mathcal{G}_5$.
There is a directed path from $w$, creating one of the digraphs $\mathcal{G}_6$--$\mathcal{G}_{10}$ shown in Figure \ref{threetreegraphspic9}.
\begin{figure}[htb]
\centering
\input{pictexfiles/threetreegraphspic9}
\caption{\label{threetreegraphspic9}}
\end{figure}
If $i\in\{6,7,8,10\}$ then $|\Tr(\mathcal{G}_i,v)|\geq 4$. This leaves $\mathcal{G}_{9}$ as the only possibility.

There is another directed path beginning at $w'$. By discarding cases that have already been considered, we see that one of the digraphs $\mathcal{G}_{11}$--$\mathcal{G}_{13}$ shown in Figure \ref{threetreegraphspic10} is contained in $\mathcal{G}$.
\begin{figure}[htb]
\centering
\input{pictexfiles/threetreegraphspic10}
\caption{\label{threetreegraphspic10}}
\end{figure}
If $i\in\{11,12,13\}$, then $|\Tr(\mathcal{G}_i,w')|>4$.
\end{proof}
%Master document is alexpolypaper.tex.

\begin{lemma}
Let $\mathcal{G}\in\{\mathcal{G}_{\alpha},\mathcal{G}_{\beta},\mathcal{G}_{\gamma},\mathcal{G}_{\delta},\mathcal{G}'_{\delta}\}$, where $\mathcal{G}'_{\delta}$ is the reflection of $\mathcal{G}_{\delta}$. Then $\mmfld{M}{\mathcal{G}}$ is an almost product sutured manifold.
\end{lemma}
\begin{proof}
$\mmfld{M}{\mathcal{G}_{\alpha}}$ is a 3--ball with a single suture, which is a product sutured manifold. The other cases are proved by Kobayashi.
$\mmfld{M}{\mathcal{G}_{\beta}}$ and $\mmfld{M}{\mathcal{G}_{\gamma}}$ give cases of \cite{MR1026928} Example 2.8.
$\nmfld{M}{\mathcal{G}_{\delta}}$ and $\nmfld{M}{\mathcal{G}'_{\delta}}$ are both \cite{MR1026928} Example 4.5, but with opposite orientations.
\end{proof}

\begin{remark}
Considering $M^{\mathcal{H}}$ or $M^{\mathcal{K}}$ would give the same sutured manifolds, but in a different order.
\end{remark}

\begin{theorem}\label{juhaszthmspecial}
Let $L$ be a prime, special, alternating link with $|\nrap{L}{0}|<4$. Then there is a unique incompressible Seifert surface $R$ for $L$.
\end{theorem}
\begin{proof}
Let $D_0$ be a special, alternating diagram for $L$. Then $D_0$ is either positive or negative. Possibly reflecting $D_0$ ensures it is positive. By a sequence of flypes we can also ensure that $D_0$ is twist-reduced. Let $R$ be the surface given by applying Seifert's algorithm to $D_0$. Then $R$ is a minimal genus Seifert surface for $L$.

Let $\mathcal{G}_0=\graph{M}{D_0}$. Apply moves 1 and 2 to $\mathcal{G}_0$ as many times as possible, giving a sequence of digraphs $\mathcal{G}_1,\ldots,\mathcal{G}_n$ where each $\mathcal{G}_{i+1}$ is obtained from $\mathcal{G}_i$ by a move 1 or a move 2 (see Definition \ref{movedefn}). These moves can be chosen so that any move 1 removes an innermost loop $e$, so that $e$ bounds a disc. Then, for any $v\in\V(\mathcal{G}_n)$,
\[
|\Tr(\mathcal{G}_n,v)|=|\Tr(\mathcal{G}_{n-1},v)|=\cdots=|\Tr(\mathcal{G}_0,v)|=\nrap{L}{0} <4.
\]

There are no loops in $\mathcal{G}_n$, since no move 1 is possible. Neither move type can create a cut vertex, and $\mathcal{G}_0$ has no cut vertex because $D_0$ is connected and prime. Suppose there is a simple closed curve $\rho$ that meets the edges of $\mathcal{G}_n$ twice. We can reverse each move in the complement of $\rho$, so that $\rho$ meets the edges of $\mathcal{G}_0$ twice. Then $\rho$ meets two crossings of $D_0$, and otherwise lies in the black regions. As $\mathcal{G}_0$ is twist-reduced, one side of $\rho$ contains only a line of white bigons. This means that, in $\mathcal{G}_0$, that side of $\rho$ contains a single topological arc, composed of edges and two-valent vertices. Each such vertex can be removed at any stage by a move 2. Thus the points where $\rho$ meets $\mathcal{G}_n$ lie on a single edge of $\rho$. We conclude that $\mathcal{G}_n\in\Gamma$.

Therefore, up to reflection, $\mathcal{G}_n\in\{\mathcal{G}_{\alpha},\mathcal{G}_{\beta},\mathcal{G}_{\gamma},\mathcal{G}_{\delta}\}$, so $\mmfld{M}{\mathcal{G}_n}$ is an almost product sutured manifold. Then, from Lemmas \ref{mgraphmoveonelemma}, \ref{mgraphmovetwolemma}, we see that $\mmfld{M}{\mathcal{G}_i}$ is an almost product sutured manifold for each $i\leq n$. In particular, $\mmfld{M}{\mathcal{G}_0}{}=\mmfld{M}{\graph{M}{D_0}{}}$ is an almost product sutured manifold, so $R$ is unique.
\end{proof}

\begin{remark}
The diagram $D_0$ has properties (L1)--(L5). The graphs $\mathcal{G}_1,\cdots,\mathcal{G}_n$ correspond via $\graphm{M}$ to diagrams $D_1,\cdots,D_n$. As we have seen, a move 2 corresponds to removing a white bigon, and a move 1 to untwisting a nugatory crossing. This point of view explains why $D_n\in\Lambda$.
\end{remark}

\begin{corollary}[see \cite{MR2443240} p604]\label{fibrednesscor}
Let $L$ be a special, alternating link.  If $\nrap{L}{0}{}=1$ then $L$ is fibred.
\end{corollary}
\begin{proof}
In this case, $\mathcal{G}_n=\mathcal{G}_\alpha$. Thus $\mmfld{M}{\mathcal{G}_n}$ is a 3--ball with a single suture, which is the complement of a disc. Hence $R$ is obtained from a disc by plumbing with Hopf bands. By Lemma \ref{fibredsumprop}, $L$ is fibred with fibre $R$.
\end{proof}

\begin{theorem}[see \cite{MR2443240} Corollary 2.4]\label{fulljuhaszproof}
Let $L$ be an homogeneous link. If $\nrap{L}{0}<4$ then $L$ has a unique incompressible Seifert surface.
\end{theorem}
\begin{proof}
Let $D$ be a homogeneous diagram for $L$ with no nugatory crossings. 
Break up $D$ into special diagrams, and then break these into prime pieces.
This gives prime, special, alternating link diagrams $D_1,\cdots,D_n$ for some $n\in\mathbb{N}$ such that $D$ can be constructed by combining $D_1,\cdots,D_n$ by taking connected sums and $*$--products. Let $L_i$ be the link with diagram $D_i$ for $1\leq i\leq n$.
Then $\nrap{L}{0}{}=\prod_{i=1}^n\nrap{L_i}{0}$ by Proposition \ref{alexpolyproductprop}. Since $\nrap{L}{0}{}\in\{1,2,3\}$, without loss of generality, $\nrap{L_{1}}{0}{}=\nrap{L}{0}$ and $\nrap{L_{i}}{0}{}=1$ for all $i>1$.

By Theorem \ref{juhaszthmspecial}, $L_{i}$ has a unique incompressible Seifert surface $R_{i}$ for each $i$, which is given by Seifert's algorithm. For $i>1$, $L_{i}$ is fibred with fibre $R_{i}$. By repeated use of Corollary \ref{incompcor}, $L$ has a unique incompressible Seifert surface. 
\end{proof}

In \cite{MR0312487}, Riley proves the following using Hermitian forms. He notes that an alternative proof based on \cite{MR0099665} is possible.

\begin{corollary}[\cite{MR0312487} Theorem]
Choose $g\geq 0$ and $m,n\geq 1$. Then there are finitely many alternating links $L$ with $m$ link components and genus $g$ such that $\nrap{L}{0}=n$.
\end{corollary}
\begin{proof}
Let $L$ be such a link with a reduced, alternating diagram $D$. Let $L_1,L_2$ be non-trivial alternating links with reduced, alternating diagrams $D_1,D_2$. Suppose $L=L_1 * L_2$. Let $R,R_1,R_2$ be the minimal genus Seifert surfaces given by applying Seifert's algorithm to $D,D_1,D_2$ respectively. Then $\chi(R_1),\chi(R_2)>\chi(R)$, and $\nrap{L_1}{0},\nrap{L_2}{0}\leq\nrap{L}{0}$. In addition, there are only finitely many ways of combining $D_1$ and $D_2$ by a $*$--product. The same is true if $L=L_1\#L_2$. These facts allow us to reduce to the case that $D$ is prime and special.

Starting from $\graph{M}{D}$, form a sequence of graphs as in the proof of Theorem \ref{juhaszthmspecial}, with final graph $\mathcal{G}$. 
Then $\mathcal{G}\in\Phi_n$, so by Theorem \ref{gammafinitethm} there are only finitely many possibilities for $\mathcal{G}$. 
We know that $\graph{M}{D}$ is obtained from $\mathcal{G}$ by reversing a finite sequence of moves 1 and 2. 
It therefore remains to bound the length of this sequence. 
Note that no move 1 can be performed on $\graph{M}{D}$ as $D$ is reduced, and each move 2 creates at most one loop. 
It thus suffices to bound the number of times a move 2 occurs.
Since a move 2 reduces the Euler characteristic of the corresponding surface, such a bound exists. 
\end{proof}

\subsection{Limitations of the proof}

It seems unlikely that the methods we have used can be extended to give a complete proof of Theorem \ref{juhaszthm}. 
Our proof relies heavily on using the Seifert surface distinguished by applying Seifert's algorithm to a fixed diagram. 
By examining the link of this surface in $\is(L)$ we can establish that there are no 1--simplices in $\is(L)$ or $\ms(L)$. 
That is, we are able to construct a maximal dimensional simplex. It is not clear how to do so in general. 
Applying Seifert's algorithm to different alternating diagrams of $L$ can yield more than one minimal genus Seifert surface for $L$, but a maximl dimensional simplex need not include such a vertex.
It is not even known whether $\ms(L)$ always contains a maximal dimensional simplex with a vertex that can be formed in this way. 

%------------------------------

\bibliography{alexpolyrefs}
\bibliographystyle{hplain}

%------------------------------

\bigskip
\noindent
Mathematical Institute

\noindent
University of Oxford

\noindent
24--29 St Giles'

\noindent
Oxford OX1 3LB

\noindent
England

\smallskip
\noindent
\textit{jessica.banks[at]lmh.oxon.org}
\end{document}